\documentclass[11pt]{article}
\usepackage{amsthm,amssymb,mathrsfs,mathtools,empheq}
\usepackage{bbm}
\usepackage[shortlabels]{enumitem}
\usepackage[T1]{fontenc}

\usepackage[notref,notcite,final]{showkeys}
\usepackage{fullpage,color}
\usepackage{hyperref}

\usepackage{authblk}

\newtheorem{mainthm}{Theorem}

\newtheorem{theorem}{Theorem}[section]
\newtheorem*{theorem*}{Theorem}
\newtheorem{corollary}[theorem]{Corollary}

\newtheorem{lemma}[theorem]{Lemma}
\newtheorem{proposition}[theorem]{Proposition}
\newtheorem*{proposition*}{Proposition}

\newtheorem*{conjecture*}{Conjecture}

\theoremstyle{definition}
\newtheorem{definition}[theorem]{Definition}
\newtheorem{remark}[theorem]{Remark}
\newtheorem*{rk}{Remark}
\newtheorem{cl}[theorem]{Claim}

\numberwithin{equation}{section}
\def\bR {\mathbb{R}}

\def\d{\displaystyle}

\def\cE {\mathcal{E}}
\def\cF {\mathcal{F}}

\def\F{{\cal F}}

\def\la {\langle}
\def\ra {\rangle}

\def\ps{\partial_s}

\def\sgn{\mathop{\rm sgn}}

\newcommand\bb{{d}}
\newcommand{\ds}{\displaystyle}

\newcommand{\kz}{{K_0}}
\newcommand{\ku}{{K_1}}
\newcommand{\tx}[1]{\mathrm{#1}}

\newcommand{\wt}[1]{\widetilde{#1}}
\newcommand{\bs}[1]{\boldsymbol{#1}}

\newcommand{\eee}{\mathrm e}

\renewcommand\iint{\int_{-1}^1}

\newcommand{\ud}{\mathrm{\,d}}
\newcommand{\vd}{\mathrm{d}}

\newcommand{\dd}[1]{{\frac{\vd}{\vd{#1}}}}

\newcommand{\bd}[1]{\boldsymbol{#1}}
\newcommand{\m}[1]{\mathbbm{#1}}
\newcommand{\q}[1]{\mathcal{#1}}
\newcommand\qlun{{q^l_1}}
\newcommand\qldeux{{q^l_2}}
\newcommand\RR{{\cal R}}
\renewcommand\SS{{\cal S}}
\renewcommand\tt{\mu}
\newcommand{\un}[2]{\m 1_{\{#1<y<#2\}}}
\newcommand\xxi{\zeta}

\renewcommand\H{{\mathcal{H}}}

\newcommand{\vc}[2]{\begin{pmatrix} #1\\#2\end{pmatrix}}

\title{\textbf{Scalar behavior for a complex multi-soliton \\ arising in
    blow-up for a semilinear wave equation}}
\author[*]{Asma Azaiez}
\author[**]{Jacek Jendrej}
\author[***]{Hatem Zaag}
\affil[*]{Department of Mathematics, Faculty of Sciences of Sfax,
  University of Sfax, BP 1171, 3000, Sfax, Tunisia
 \texttt{asma.azaiez@yahoo.fr}}
\affil[**]{Sorbonne Universit\'e, IMJ-PRG, CNRS (UMR 7586), F-75005,
  Paris, France,
  and~AGH~University of Science and Technology, WMS, 30-059 Krak\'ow, Poland,
\texttt{jendrej@imj-prg.fr}}
\affil[***]{Universit\'e Sorbonne Paris Nord,
  LAGA, CNRS (UMR 7539), F-93430, Villetaneuse, France,
  \texttt{hatem.zaag@univ-paris13.fr}}

\begin{document}

\maketitle

\begin{abstract}
This paper deals with blow-up for the complex-valued semilinear wave equation
with power nonlinearity in dimension 1. Up to a rotation of the
solution in the complex plane, we show that near a characteristic blow-up
point, the solution behaves exactly as in the real-valued
case. Namely, up to a rotation in the complex plane, the solution
decomposes into a sum of a finite number of decoupled
solitons with alternate signs. The main novelty of our proof is a resolution
of a complex-valued first order Toda system governing the evolution of the positions and the phases of the solitons.
\end{abstract}

\medskip

{\bf MSC 2010 Classification}:  
%
%
35B40,    	
35B44,    	
35L05,    	
35L51,   	
35L67,    	
35L71    	

\medskip

{\bf Keywords}: Semilinear wave equation, complex-valued equation,
blow-up behavior, first order Toda system, characteristic points

\section{Introduction}
We consider the occurrence of blow-up for  the following complex-valued one-dimensional semilinear wave equation
\begin{equation}\label{equ} 
  \left\{
\begin{array}{l}
\displaystyle\partial^2_{t} u = \partial^2_{x} u+|u|^{p-1}u, \\
u(0)=u_{0} \mbox{ and }  u_{t}(0) = u_{1},
\end{array}
\right . 
\end{equation}
where $u(t):  x\in \mathbb{R}\to u(x,t) \in \mathbb{C}$, $p>1$,
$u_0 \in H^1_{loc,u}$ and $ u_1\in L^2_{loc,u}$ with \\
$\|v\|^2_{L^2_{loc,u}}=\displaystyle\sup\limits_{a\in \mathbb{ R}}
\int_{|x-a|<1}|v(x)|^2 dx $ and $\| v\|^2_{H^1_{loc,u}}=\|
v\|^2_{L^2_{loc,u}}+\|  \nabla v\|^2_{L^2_{loc,u}}$.
The Cauchy problem is locally wellposed. By energy arguments, Levine showed in \cite{Ltams74} the existence of blow-up solutions. 

\medskip

Equation \eqref{equ} can be considered as a lab model for blow-up in hyperbolic equations, because it captures features common to a whole range of blow-up problems arising in various nonlinear physical models, in particular in general relativity (see Donninger, Schlag and Soffer \cite{DSScmp12}), and also for self-focusing waves in nonlinear optics (see Bizo\'n, Chmaj and Szpak \cite{BCSjmp11}). 
Since many of those more physical examples take
the form of systems, we are interested in the complex-valued case in
\eqref{equ}, which has been poorly studied.

\medskip

If $u$ is a blow-up solution of equation \eqref{equ}, we define (see
for example Alinhac \cite{Apndeta95}) a 1-Lipschitz curve $\{(x,T(x))\}$ where $x\in{\bR}$ such that the domain of definition of $u$ is written as 
\begin{equation}\label{defdu}
D=\{(x,t)\;|\; t< T(x)\}.
\end{equation}
The set $\{(x,T(x))\}$ is called the blow-up surface of $u$. 
A point $x_0\in{\bR}$ is a non-characteristic point if there are 
\begin{equation}\label{nonchar}
\delta_0\in(0,1)\mbox{ and }t_0<T(x_0)\mbox{ such that }
u\;\;\mbox{is defined on }{\cal C}_{x_0, T(x_0), \delta_0}\cap \{t\ge t_0\}
\end{equation}
where 
\begin{equation}\label{defcone}
{\cal C}_{\bar x, \bar t, \bar \delta}=\{(x,t)\;|\; t< \bar t-\bar \delta|x-\bar x|\},
\end{equation}
as illustrated in figure \ref{fig1}.
 \begin{figure}
\centering
\includegraphics[width=0.6\textwidth]{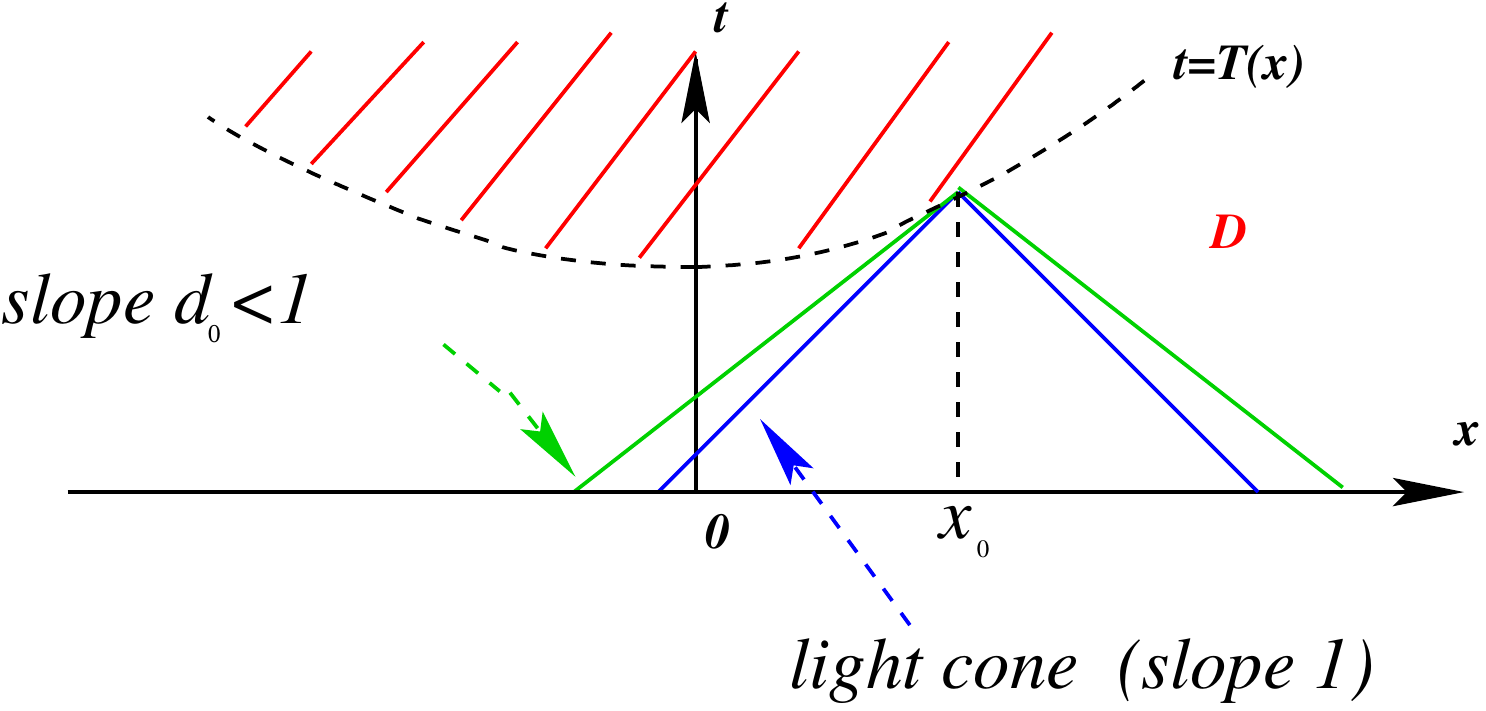}
\caption{$x_0$ is a non-characteristic point ($N=1$).}\label{fig1}
\end{figure}
If not, we say that $x_0$ is a characteristic point.
 We denote by $\RR\subset {\bR}$ the set of non-characteristic
 points and by $\SS$ the set of characteristic points, as illustrated in figure \ref{fig2}.
 \begin{figure}
\centering
\includegraphics[width=0.6\textwidth]{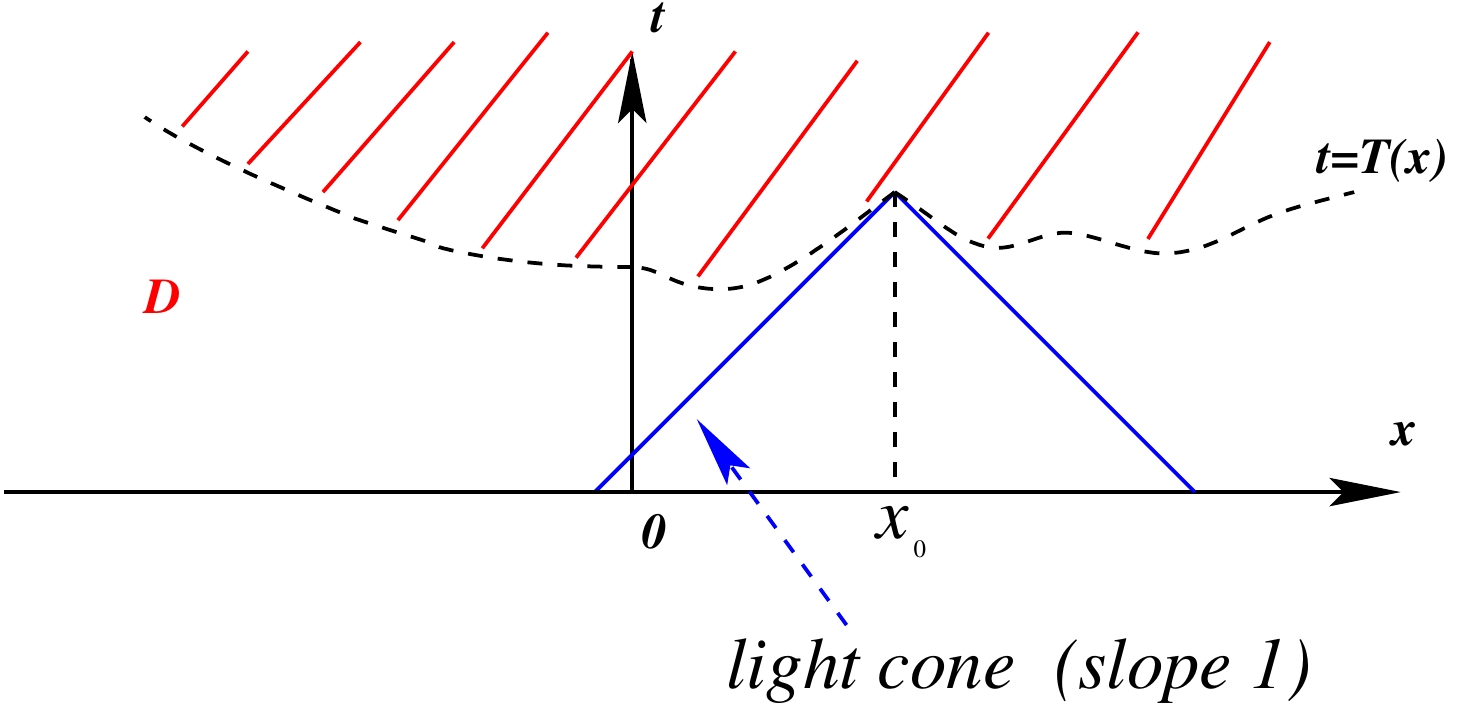}
\caption{$x_0$ is a characteristic point ($N=1$).}\label{fig2}
\end{figure}

\bigskip

In order to state blow-up results for $u$, we introduce the following self-similar change of variables:
\begin{equation}\label{defw}
w_{x_0}(y,s) =(T(x_0)-t)^\frac{2}{p-1}u(x,t), \quad  y=\frac{x-x_0}{T(x_0)-t}, \quad
s=-\log(T(x_0)-t),
\end{equation}
together with  the following energy functional
\begin{equation}\label{defE}
 E(w,\partial_s w)=\int_{-1}^{1} \left( \frac{1}{2} |\partial_s w|^2+\frac{1}{2} |\partial_y w|^2 (1-y^2)+\frac{p+1}{(p-1)^2}|w|^2-\frac{1}{p+1}|w|^{p+1}\right) \rho dy
\end{equation}
where
\begin{equation}\label{defro}
  \rho (y)= (1-y^2)^\frac{2}{p-1}.
  \end{equation}
Note that this functional is well defined in the following Hilbert space
 \begin{equation}\label{defnh}
\mathcal{ H}=\{ q \in H_{loc}^1 \times L_{loc}^2 ((-1,1),\mathbb{C}) \, \|  q \|_{\mathcal{ H}}^2 \equiv \int_{-1}^{1}(|q_1|^2+|q'_1|^2(1-y^2)+|q_2|^2)\rho\;dy< +\infty \} ,
\end{equation}
thanks to some  Hardy-Sobolev identity proved by Merle and Zaag in the appendix of \cite{MZajm03}.
Note that $\mathcal{ H}=\mathcal{ H}_0\times L_{\rho}^2$ where
 \begin{eqnarray}\label{10}
\mathcal{ H}_0=\{ r \in H_{loc}^1  ((-1,1),\mathbb{C}) \,\| r \|_{\mathcal{ H}_0}^2 \equiv \int_{-1}^{1}(|r'|^2(1-y^2)+|r|^2)\rho\;dy< +\infty\} .
 \end{eqnarray}
%
%
 Note also that the change of variables \eqref{defw} transforms the backward light cone with
vertex $(x_0, T(x_0))$ into the infinite cylinder $(y,s)\in (-1,1)
\times [-\log T(x_0),+\infty).$ We also introduce the following solitons
\begin{eqnarray}\label{defk}
\forall (d, y) \in (-1,1)^2,\;\kappa(d,y)= \kappa_0 \frac{(1-d^2)^\frac{1}{p-1}}{(1+dy)^\frac{2}{p-1}} \mbox{ and }\kappa_0=\left(\frac{2(p+1)}{(p-1)^2}\right)^\frac{1}{p-1},
\end{eqnarray}
and
\[
\bar \zeta_{k,i}(s) = \frac{p-1}2 \left(i-\frac{k+1}2\right) \log s + \alpha_{k,i}
\]
(we write $\bar \zeta_i(s)$ for simplicity)
for any $k\ge 2$, $i=1,\dots,k$ and $s>0$, where the $\alpha_{k,i}$ are 
chosen so that $(\bar \zeta_1(s),\dots, \bar\zeta_k(s))$ is a solution
of the following first order (real-valued) Toda system:
\begin{equation}\label{toda}
\forall i=1,\dots,k,\;\;\frac 1A \bar \zeta'_i = e^{-\frac 2{p-1}(\bar\zeta_i-\bar\zeta_{i-1})}- e^{-\frac 2{p-1}(\bar\zeta_{i+1}-\bar\zeta_i)},
\end{equation}
with zero center of mass, in the sense that
\[
\frac{\bar\zeta_1(s)+\dots +\bar\zeta_k(s)}k=0, \forall s> 0.
\]
Note that
\begin{equation}\label{defA}
 A=\frac{p(p-1)}{2\iint \rho(y)dy}\kappa_0^{p-1}2^{\frac
   2{p-1}}\int_{-\infty}^\infty  \cosh^{-\frac{2p}{p-1}}(z)e^{\frac {2z}{p-1}}\tanh z dz >0,
 \end{equation}
is a constant that was first encountered 
in Proposition 3.2 page 590 in Merle and Zaag \cite{MZajm12}. Although
its precise value was not give in that statement, a careful reading of
the proof gives the expression in \eqref{defA}. Note also that we
take by convention $\bar\zeta_0\equiv - \infty$ and  $\bar\zeta_{k+1}\equiv + \infty$.

\bigskip

Let us first recall that the real-valued case has been much more
studied than the complex-valued. In
particular, in one space dimension, all the blow-up behavior has been
studied in a series of papers by Merle and Zaag \cite{MZjfa07,
  MZcmp08, MZxedp10, MZajm12, MZdmj12} and also in C\^ote and Zaag
\cite{CZcpam13}.
See Hamza and Zaag \cite{HZjde19} for a related result.
See also Caffarelli and Friedman in \cite{CFtams86} and \cite{CFarma85} for earlier results. See also Killip and Vi\c san \cite{KV11}.\\ 
In higher dimensions, some partial results are available in
\cite{MZcmp15} and \cite{MZtams16}, and also in \cite{MZcpam18} where
a solution having a pyramid shaped blow-up surface is constructed (see
also the note \cite{MZsls17}).

\bigskip

In the complex-valued case, we are aware of the papers by the first
author Azaiez
\cite{Atams15} and \cite{Acpaa19}, where the blow-up behavior near
non-characteristic points is derived (see Theorem 5 and Proposition 3 in \cite{Atams15}):

\medskip

\textit{
  $\RR$ is open and
  $T_{|\RR}$ is $C^1$. Moreover, if
$x_0\in \RR$, 
then $w_{x_0}(s) \to \kappa(T'(x_0),\cdot)$ and\\
  $E(w_{x_0}(s),\partial_s w_{x_0}(s))\to E(\kappa_0,0)$ as $s\to \infty$.
}
\medskip

A generalization to the vector-valued case (with 3
components and more)  was later obtained by
Azaiez and Zaag in \cite{AZbsm17}. As for the characteristic case, it
remained so far open. This is precisely the aim of our paper, and this is our main result:
\begin{mainthm}[Description of the set of characteristic points and the
  blow-up behavior at characteristic points]\label{th1}
If $x_0\in \SS$, then 
\begin{equation}\label{decom1}
\left\|\begin{pmatrix} w_{x_0}(s)\\\partial_s
  w_{x_0}(s) \end{pmatrix} -e^{i \bar
  \theta(x_0)}\begin{pmatrix}\displaystyle\sum_{j=1}^{\bar k(x_0)}
  (-1)^{j+1}\kappa(\bb_j (s),\cdot)\\0\end{pmatrix}
\right\|_{\mathcal{ H}}\rightarrow 0
\end{equation}
and $E(w_{x_0}(s), \partial_s w_{x_0}(s))\rightarrow \bar k(x_0) E(\kappa_0,0)$
 as $s\rightarrow \infty$, with  $d_j(s)=-\tanh
 (\bar\zeta_j(s)+\bar\zeta(x_0)) \in (-1,1)$ for $j=1,...,\bar k(x_0)$,  for some
\[
 \bar \theta(x_0)\in \bR, \;\;\bar k(x_0)\ge 2\mbox{ and }\bar\zeta(x_0)\in \bR.
\]
\end{mainthm}
\begin{rk} This is the first time a result about characteristic points
  for a vector-valued wave equations is obtained. The fact that the
  nonlinear term in equation \eqref{equ} is a gradient (when $\m C$ is
  identified with $\m R^2$) is crucial. See Proposition \ref{propdyn}
  and its proof for details.
\end{rk}

From \eqref{decom1}, we see that up to some rotation in the complex
plane, the solution is asymptotically real-valued. Relying on the one hand on the
analysis in the real-valued case given by Merle and Zaag in
\cite{MZajm12} and \cite{MZdmj12} together with C\^ote and Zaag in
\cite{CZcpam13}, and on the other hand, on the results of Azaiez in
\cite{Atams15} in the complex-valued case, we obtain the following
description of the blow-up graph near characteristic points:
\begin{mainthm}[Regularity of the blow-up graph near characteristic
  points] \label{cor1}
   Following Theorem \ref{th1}, 
the set $\q S$ consists of
  isolated points, and for any $x_0\in \q S$, $T$ is left and right
  differentiable, with $T'_l(x_0)=1$ and  $T'_r(x_0)=-1$. Moreover, 
 \begin{align}
T'(x)& = -\eta(x)\left(1-\frac{\gamma
       e^{-2\eta(x)\bar \zeta(x_0)}(1+o(1))}{|\log|x-x_0||^{\frac{(\bar
       k(x_0)-1)(p-1)}2}}\right),\label{t'}\\
T(x)&=T(x_0)-|x-x_0|+\frac{\gamma
      e^{-2\eta(x)\bar \zeta(x_0)}|x-x_0|(1+o(1))}{|\log|x-x_0||^{\frac{(\bar
      k(x_0)-1)(p-1)}2}}\nonumber
\end{align}
as $x\to x_0$, where $\eta(x) = \frac{x-x_0}{|x-x_0|}=\pm 1$ and $\gamma=\gamma(p)>0$.
\end{mainthm}
\begin{rk}
Since $\q S$ is made of isolated blow-up points, for $x$ close enough
to $x_0$, $x\in \q R$. From what we have written just before the statement
of Theorem \ref{th1},
$T'(x)$ is well defined, hence \eqref{t'} is meaningful.
\end{rk}

As one may expect, the proof of these results needs two ``classical'' ideas:\\
- Ideas about the characteristic case, introduced in the real-valued
case in \cite{MZajm12}, \cite{MZdmj12} and \cite{CZcpam13};\\
- Ideas about the complex-valued case, introduced by Azaiez in
\cite{Atams15} in the non-characteristic case; in particular, a
modulation technique is important to handle the zero eigenvalue
appearing from the invariance of equation \eqref{equ} by rotations of
the complex plane.

\medskip

In fact, these ideas are far from being enough, in the sense that the
proof exhibits what we call a ``complex-valued version of the first-order Toda system''
\eqref{toda}, which is a coupled set of ODEs satisfied by the
solitons' centers and phases, see below in \eqref{toda-1}-\eqref{toda-2}. Up to our
knowledge, this is the first time such an ODE system is
derived. Accordingly, our analysis for such a system is completely new
and makes the originality of our paper.
The analogous real-valued first-order Toda system was derived and analyzed by C\^ote and Zaag \cite{CZcpam13}. Their approach, based on comparison principles, does not apply in the complex-valued case.

\bigskip

Let us mention that the usual second order Toda system appears in the context of elliptic PDEs as governing the shape of interfaces for solutions of Allen-Cahn equations, see Del Pino, Kowalczyk and Wei \cite{Pino}.
Furthermore, a second order Toda system with attractive interactions is analysed in \cite{JL9}
as governing the motion of multikinks for a wave equation.

\bigskip

The idea of reducing the description of the dynamical behavior of superpositions of soliton-like objects
to an appropriate system of ODEs is well-established in the literature,
in the contexts of both gradient flows and Hamiltonian PDEs.
Without aiming to be exhaustive, let us mention the works on the Ginzburg-Landau gradient flow \cite{JerSon1, JerSon2, BetOrlSme},
on Bose-Einstein condensates \cite{JerSpi},
on Ginzburg-Landau equation of the Schr\"odinger type \cite{OvSi}
and on wave-type equations \cite{DuMa, Stuart, GuSi06}. We also refer to the monograph \cite{MS} for a detailed account of the literature. Let us mention that, in most cases, Hamiltonian PDEs lead to Hamiltonian ODEs for the solitons' motion, and gradient flow PDEs lead to gradient flow ODEs.
In the present paper, however, due to the self-similar change of variables and a kind of ``damping effect'' which it introduces, the resulting ODE system is not Hamiltonian, even though we start with a wave equation.

\bigskip

We  proceed in several sections to give the proofs of Theorems
\ref{th1} and \ref{cor1}: \\ 
- In Section \ref{SecCompl}, we use the strategy introduced by Merle
and Zaag in \cite{MZajm12} and \cite{MZdmj12} in the real-valued
characteristic case, adapted by the formulation introduced by
Azaiez in \cite{Atams15} in the complex-valued non-characteristic
case, to show that the solution approaches a sum of decoupled
solitons. Moreover, we derive the complex-valued first order Toda system mentioned above.\\
- In Section \ref{appdyn}, we analyze the selfsimilar version of the
equation given below in \eqref{eqw}, and prove the various statements
of Section \ref{SecCompl}.\\
- In Section \ref{SecToda}, we analyze the complex-valued first order
Toda system, and show that asymptotically, our solution behaves like
the real-valued case, up to a rotation of the complex plane.\\   
- In Section \ref{SecConcl}, since we know that the question is
reduced to the real-valued case, we use the former results of the
real-valued characteristic case from Merle and Zaag in \cite{MZajm12,
  MZdmj12} and C\^ote and Zaag in \cite{CZcpam13}, together with the
analysis of the complex-valued case in Azaiez \cite{Atams15}, in order
to conclude the proofs of Theorems \ref{th1} and \ref{cor1}. 

\section{Approaching a sum of decoupled solitons}\label{SecCompl}
In this section, we  use the strategy introduced by Merle
and Zaag in \cite{MZajm12} and \cite{MZdmj12} in the real-valued
characteristic case, adapted by the formulation introduced by
Azaiez in \cite{Atams15} in the complex-valued non-characteristic
case, to show that the solution approaches a sum of decoupled
solitons. Moreover, we derive the complex-valued first order Toda system mentioned above.

\bigskip

We proceed in several steps, each given in a separate subsection:\\
- In Section \ref{secstep1}, we show that the solution decomposes into
a sum of $k\ge 0$ solitons, which are decoupled in some sense, if ever
$k\ge 2$. This step follows as in the real-valued case.\\
- In Section \ref{secstep2}, assuming $k\ge 2$, we linearize the
equation around the sum of solitons, and exhibit the ODE system
satisfied by the solitons' centers and phases. 

\bigskip

\subsection{Decomposition into a sum of decoupled solitons}\label{secstep1}

This step follows exactly as in the real-valued
 case, with the exception of the characterization of self-similar
solutions of equation \eqref{equ}, which is more involved and follows
from Azaiez \cite{Atams15} (see Proposition \ref{propchar} below).  This is our statement:
\begin{proposition}[Decomposition of the solution into a sum of decoupled solitons]\label{propdecomp}
Consider $x_0\in \mathcal{S}$. Then,
\[
\left\|\begin{pmatrix} w_{x_0}(s)\\\partial_s w_{x_0}(s) \end{pmatrix} -\sum_{i=1}^{k}e^{i\theta_i(s)}\begin{pmatrix} \kappa(\bb_i (s),\cdot)\\0\end{pmatrix} \right\|_{\mathcal{ H}}\rightarrow 0\mbox{ as }s\rightarrow \infty
\]
for some $k=k(x_0)\ge 0$, and continuous functions $\theta_i(s)\in \mathbb{R}$ and $d_i(s)=-\tanh \zeta_i(s) \in (-1,1)$, for $i=1,...,k$ with
\begin{equation}\label{separation}
 \zeta_1(s)<...<\zeta_k(s) \mbox{ and }\zeta_{i+1}(s)-\zeta_i(s) \rightarrow \infty \mbox{ for all }i=1,...,k-1.
\end{equation}
\end{proposition}
\begin{rk}
If $k=0$, then the sum in this statement has to be understood as $0$
and we have no parameters $\theta_i(s)$ and $d_i(s)$. If $k=1$, then,
we have only one parameter $(\theta_1(s), d_1(s))$ and condition
\eqref{separation} is irrelevant.
\end{rk}
\begin{proof}
  The proof is exactly the same as the real case (see Theorem 2 page 47 in
\cite{MZjfa07}). Let us briefly
  recall its main ingredients. For simplicity, we write $w(y,s)$
  instead of $w_{x_0}(y,s)$.\\
  - First of all, we need to write the PDE satisfied by  the similarity variables version defined in
\eqref{defw}, which is the same as in the real case: for all $|y|<1$ and $s\ge -\log T(x_0)$:
\begin{equation} \label{eqw}
 \partial^2_{s} w=\mathcal{L}w-\frac{2(p+1)}{(p-1)^2}w+|w|^{p-1}w-\frac{p+3}{p-1} \partial_s w- 2 y \partial_{ys} w
\end{equation}
where
 \begin{equation}\label{8}
\mathcal{L} w=\frac{1}{\rho}\partial_y
   (\rho (1-y^2)\partial_y w)
 \end{equation}
and $\rho$ is defined in \eqref{defro}.\\
- Second, we note that the energy functional given in \eqref{defE} is
 well defined in the space $\mathcal H$, thanks to a Hardy-Sobolev
 identity proved in the Appendix of \cite{MZajm03}. More importantly,
 it is a Lyapunov functional for equation \eqref{eqw}, in the sense that
 \[
\frac d{ds} E(w(s),\partial_sw(s))= - \frac 4{p-1} \int_{-1}^1
\frac{|\partial_s w(s)|^2}{1-y^2}\rho dy, 
 \]
 exactly as in the real case. In particular, this functional is
 bounded from above by its initial value.\\
 - Third, given that $w(y,s)$ is by construction defined for all
 $(y,s)\in (-1,1)\times[-\log T(x_0), +\infty)$, it follows that this
 Lyapunov functional is nonnegative, thanks to the following blow-up
 criterion related to equation \eqref{eqw}, which follows exactly as
 in the real-valued case: 
\begin{lemma}[A blow-up criterion for equation \eqref{eqw}]
  Assume that $W(y,s)$ is a solution of equation \eqref{eqw} with
  $E(W(s_0))<0$, for some $s_0\in \bR$. Then, $W(y,s)$ cannot be
  defined for all $(y,s)\in (-1,1)\times[s_0,+\infty)$. 
  \end{lemma}
  \begin{proof} 
See Theorem 2 page 1147 in Antonini and Merle \cite{AMimrn01}.
  \end{proof}
 \noindent - Fourth, the boundedness of all terms in the function $E(w(s), \partial_s
  w(s))$, by the argument of \cite{MZajm03} and also \cite{MZimrn05},
  completed by Proposition 3.5 page 66 in \cite{MZjfa07}.\\
  - Fifth, and this is a specific argument to the complex-valued case,
 we recall from the complex-valued non-characteristic case treated by
 Azaiez \cite{Atams15} that we can characterize all the stationary
 solutions of equation \eqref{eqw} in the energy space as follows:
 \begin{lemma}[Characterization of all stationary solutions of
   equation \eqref{eqw} in the energy space $\q H_0$ \eqref{10}] \label{propchar} Consider $w\in {\mathcal H}_0$ a
   stationary solution of equation \eqref{eqw}. Then:\\
   - either $w\equiv 0$,\\
   - or
$w(y) = e^{i\theta}\kappa(d,y)$,
   for some $\theta\in \bR$ and $d\in(-1,1)$, where $\kappa(d,y)$ is
   defined in \eqref{defk}.
 \end{lemma}
 \begin{proof}
The proof is involved and very different from the real-valued case. It
was already given by Azaiez in Proposition 2 in \cite{Atams15}. 
\end{proof}

With these 5 ingredients, one can follow the real-valued proof in
\cite{MZjfa07} to show that $w(y,s)$ has local limits in space, as
$s\to \infty$ which
are stationary solutions, thanks to the Lyapunov functional property. The rigidity of the shape of
$\kappa(d,y)$ implies that non-zero local limits have their parameters
separated as in \eqref{separation}. The fact that the Lyapunov
functional is finite implies that non-zero local limits are in
finite number. For details, see Section 3.2 page 66 of \cite{MZjfa07}.
\end{proof}
\subsection{Dynamics of the equation around the sum of solitons}\label{secstep2}
Consider $x_0\in \q S$ and  $k(x_0)$ the number of solitons appearing
in the decomposition of $w_{x_0}(y,s)$ given in Proposition \ref{propdecomp}. 
This section is dedicated to the study of the case where
\[
k\equiv k(x_0)\ge 2.
\]
More precisely, we linearize equation \eqref{eqw} around the sum of
solitons given in Proposition \ref{propdecomp}, in order to derive the
equations \eqref{toda-1}-\eqref{toda-2} satisfied by the solitons'
center and phase.  As the solitons in Proposition \ref{propdecomp} are
decoupled, thanks to \eqref{separation}, we will reduce to the case of
a linearization around one single soliton, which justifies Step 1 below.

\medskip

In fact, we proceed in 3 steps, each given in a different section:\\
- First, we recall in Section \ref{sec1}
from \cite{Atams15} the spectral properties of the
linearized operator around a single
soliton. In particular, we recall that it has one positive eigenvalue
($\lambda=1$) together with 2 zero modes.\\
- Then, we introduce in Section \ref{sec2}
a modulation technique to slightly change the
solitons' centers and phases, so that we kill the 2 zero modes
of the linearized operator around each soliton in the decomposition given
in Proposition \ref{propdecomp}.\\
- Finally, in Section \ref{sec3},
we linearize equation \eqref{eqw} around the sum
of the solitons (with the new set of parameters and phases), and show
that the error term is controlled by a function involving the distance
between the solitons' centers. We also derive a refined
ODE system satisfied by the modulation parameters.

\subsubsection{The linearized operator
  around one single  soliton}\label{sec1}

In this step, we recall from \cite{MZjfa07} and \cite{Atams15} the spectral properties of
the linearized operator of equation \eqref{eqw} around one single
soliton. For proofs and details, see Section 4 in \cite{MZjfa07} and Section 3 in \cite{Atams15}.
Consider any complex-valued solution $w(y,s)$ of equation \eqref{eqw} defined for all
$(y,s)\in (-1,1) \times [s_0, +\infty)$ for some $s_0\in \bR$
and a soliton $\kappa(d,y)$ defined in
  \eqref{defk}, with $d\in (-1,1)$.

\medskip

For any complex number $z$ (in particular for all our complex-valued functions), we use in the following the notation
\begin{equation}\label{riz}
  \check z=\Re z,\;\;\tilde z=\Im z
  \mbox{ and the standard notation } \bar z = \Re z - i \Im z.
\end{equation}
If we introduce
$q=(q_1,q_2)=\begin{pmatrix}{q_1}\\{q_2}\end{pmatrix}\in
\mathbb{C}\times \mathbb{C}$ by
\begin{eqnarray*}
\begin{pmatrix} w(y,s)\\\partial_s w(y,s) \end{pmatrix} =\begin{pmatrix} \kappa(d,y)\\0\end{pmatrix} +\begin{pmatrix} q_1(y,s)\\q_2(y,s)\end{pmatrix},
\end{eqnarray*}
then, we see from \eqref{eqw} that the real and imaginary parts of $q$ satisfy the
following equations, for all $s\ge s_0$:
\begin{eqnarray}\frac{\partial }{\partial s}\begin{pmatrix} \check q_1\\\check q_2
\end{pmatrix}=\check L_\bb  \begin{pmatrix} \check q_1 \\\check q_2 \end{pmatrix}+\begin{pmatrix}0\\ \check{f_\bb }\end{pmatrix},
\end{eqnarray}
where
\begin{eqnarray}\label{barL_bb }
\check L_\bb  \begin{pmatrix} \check q_1\\\check q_2 \end{pmatrix}=
\begin{pmatrix}\check q_2 \\\mathcal{L} \check q_1+\check\psi (\bb , y)\check q_1-\frac{p+3}{p-1} \check q_2- 2 y \partial_y \check q_2
\end{pmatrix},
\end{eqnarray}
and
\begin{eqnarray}\frac{\partial }{\partial s} \begin{pmatrix} \tilde{q_1}\\\tilde q_2
\end{pmatrix}=\tilde L_\bb \begin{pmatrix} \tilde q_1\\\tilde q_2 \end{pmatrix}+\begin{pmatrix} 0\\\tilde{f_\bb } \end{pmatrix}, 
\end{eqnarray}
where
\begin{eqnarray}\label{tildeL_bb }
 \tilde L_\bb \begin{pmatrix} \tilde q_1\\\tilde q_2 \end{pmatrix}=\begin{pmatrix} \tilde q_2\\\mathcal{L}\tilde q_1+\tilde\psi (\bb , y)\tilde q_1-\frac{p+3}{p-1} \tilde q_2- 2 y \partial_y \tilde q_2 \end{pmatrix}.
\end{eqnarray}
The potentials are given by
\begin{equation}\label{defpsi}
  \check\psi (\bb , y)=p \kappa(\bb ,y)^{p-1}-\frac{2(p+1)}{(p-1)^2}
  \mbox{ and }
  \tilde\psi (\bb , y)= \kappa(\bb , y)^{p-1}-\frac{2(p+1)}{(p-1)^2}.
\end{equation}
The terms $\check f_d$ and $\tilde f_d$ are quadratic.
\paragraph{Spectral properties of  $\check L_\bb $.}

From Section 4 in
\cite{MZjfa07},
we know that $\check L_\bb $ has two nonnegative eigenvalues $\lambda=1$ and $\lambda=0$ with eigenfunctions
\begin{equation}
\check F_1 (\bb ,y)= (1-\bb^2)^{\frac{p}{p-1}}\begin{pmatrix}(1+\bb y)^{-\frac{p+1}{p-1}}\\(1+\bb y)^{-\frac{p+1}{p-1}}\end{pmatrix}
\mbox{and }\;
\check F_0(\bb ,y)= (1-\bb^2)^{\frac{1}{p-1}}\begin{pmatrix}\frac{y+\bb }{(1+\bb y)^\frac{p+1}{p-1}}\\0\end{pmatrix}. \label{110}
\end{equation}
Note that for some $C_0>0$ and any $\lambda\in \{0,1\}$, we have
\begin{equation}\label{majoration}
\forall |\bb|<1,\;\; \frac{1}{C_0}\leq \|\check F_\lambda(\bb) \|_{\mathcal{H}} \leq C_0 \;\mbox{  and  }\; \|\partial_\bb  \check F_\lambda(\bb) \|_{\mathcal{H}} \leq \frac{C_0}{1-\bb^2}.
\end{equation}
Introducing the inner product $\phi$ associated to the real case of the
Hilbertian norm \eqref{defnh}  defined by
\begin{equation}\label{defphi}
 \phi (q, r)=\phi \left(\begin{pmatrix} q_1\\q_2 \end{pmatrix}, \begin{pmatrix} r_1\\r_2 \end{pmatrix}\right)=\int_{-1}^{1}(q_1 r_1+q'_1  r'_1(1-y^2)+q_2  r_2)\rho(y)\;dy,
\end{equation}
we may consider the conjugate operator of $\check L_\bb $ with
respect to $\phi$, and introduce its eigenfunctions  $\check W_{\lambda}$ corresponding to
the the two nonnegative eigenvalues $\lambda=0$ and $\lambda=1$:
\begin{equation}\label{defwl2}
\check W_{1, 2}(\bb  ,y)= \check c_1 \frac{(1-y^2)(1-\bb^2 )^\frac{1}{p-1}}{(1+\bb y)^\frac{p+1}{p-1}},\,\check W_{0, 2}(\bb  ,y)= \check c_0 \frac{(y+\bb )(1-\bb^2 )^\frac{1}{p-1}}{(1+\bb y)^\frac{p+1}{p-1}},
\end{equation}
with\footnote{ In section 4 of
\cite{MZjfa07}
  we had non explicit normalizing constants $\check c_\lambda=\check
  c_\lambda (\bb )$. In Lemma 2.4 page 33 in
\cite{MZtams16}
  the authors compute the explicit dependence of $\check c_\lambda (\bb )$.}
\begin{equation}\label{defcl}
\frac{1}{\check c_\lambda}=2(\frac{2}{p-1}+\lambda)\int_{-1}^1 (\frac{y^2}{1-y^2} )^{1-\lambda}\rho(y) \,dy,
\end{equation}
and $\check W_{\lambda, 1}(\bb,\cdot)$ is the unique solution of the equation 
\begin{equation}\label{elliptic}
-\mathcal{L} r+ r =\left(\lambda-\frac{p+3}{p-1}\right) r_2- 2 y r'_2 + \frac{8}{p-1} \frac{r_2}{1-y^2}
\end{equation}
with $r_2= \check W_{\lambda, 2}(\bb,\cdot)$. Note that we have the
following estimates for $\lambda=0$ or $\lambda=1$:
\begin{equation}\label{majoration'}
\|\check W_\lambda(d)\|_{\q H}\le C,\;\;
  \phi (\check W_\lambda(\bb,\cdot) ,\check F_\lambda(\bb,\cdot) )=1\mbox{ and }\phi (\check W_\lambda (\bb,\cdot),\check F_{1-\lambda}(\bb,\cdot) )=0.
\end{equation}
Finally, we naturally introduce for $\lambda \in \{0, 1\}$ the
(real-valued) projectors $\check \pi^\bb _\lambda(r)$
for any $r\in\mathcal{H}$ by
\begin{equation} \label{barpi}
\check \pi_\lambda^\bb  (r)=\phi (\check W_\lambda (\bb,\cdot) , r).
\end{equation}
We also recall that
\begin{equation}\label{pld1}
\check \pi_\lambda^\bb  (\check L_\bb r)=\lambda \check \pi_\lambda^\bb  (r).
\end{equation}
Note that the analysis of the operator $\check L_d$ in Merle and Zaag
\cite{MZjfa07} shows that the eigenfunctions in \eqref{110} are in
some sense the only nonnegative directions. In order to control the
projection of any real-valued $r\in \q H$ on the negative part
of the spectrum of $\check L_d$, we introduce the following supplemental
of those directions defined by
\begin{equation*}
\check{\mathcal{H}}_{-}^d \equiv \{r \in \mathcal{H} \,|\, r\mbox { is
  real-valued}, \check \pi_1^d (r)=\check \pi_0^d(r)=0\},
\end{equation*}
together with the following bilinear form
\begin{eqnarray}\label{134'}
 \check \varphi_{d} (q,r)&=&\int_{-1}^{1} (-\check \psi(d,y) q_1r_1+q_1' r_1'(1-y^2)+q_2 r_2 ) \rho dy,
\end{eqnarray}
where $\check \psi(d,y)$ is defined in \eqref{defpsi}. Then, we recall from
Proposition 4.7 page 90 in Merle and Zaag \cite{MZjfa07}
that there exists $C_0>0$ such that for all $|d|<1$, for all
real-valued $r\in \check{\mathcal{H}}^d_{-},$
\begin{align}\frac{1}{C_0} \|r\|_{\mathcal{H} }^2\le \check \varphi_d (r,r)\le C_0 \|r\|_\mathcal{H} ^2.
\label{36,5}
\end{align}

\paragraph{Spectral properties of $\tilde L_\bb $.}
From Section 3 in \cite{Atams15}, we know that $\tilde L_\bb$ has
$\lambda=0$ as eigenvalue, with the following eigenfunction
\begin{equation}\label{F}
\tilde F_0(\bb  ,y)= \begin{pmatrix}\kappa (\bb ,y)\\0\end{pmatrix}.
\end{equation}
Moreover, it holds that 
\begin{equation}\label{majoration1}
\forall |\bb|<1,\;\; \frac{1}{C_0}\leq \|\tilde F_0(\bb,\cdot)\|_{\mathcal{H}} \leq C_0 \;\mbox{  and  }\; \|\partial_\bb  \tilde F_0(\bb,\cdot)\|_{\mathcal{H}} \leq \frac{C_0}{1-\bb^2}.
\end{equation}
for some $C_0>0$.\\
We also know that the conjugate operator of $\tilde L_\bb$ has also
$\lambda=0$ as eigenvalue, with the following eigenfunction  $\tilde
W_0 \in \mathcal{H}$ continuous in terms of $\bb$,
where 
\begin{equation}\label{W}
\tilde W_{0, 2}(\bb  ,y)= \tilde c_0 \kappa(\bb , y) \mbox{ and }\frac{1}{\tilde c_0}=\frac{4\kappa_0^2}{p-1} \int_{-1}^1 \frac{\rho(y)}{1-y^2}dy
\end{equation}
and $\tilde W_{0, 1}$ is the unique solution of equation
\eqref{elliptic} 
with $\lambda=0$ and $r_2= \tilde W_{0, 2}$. Moreover, we have
\begin{equation}\label{1}
\phi (\tilde W_0,\tilde{F_0})=1.
\end{equation}
Finally, we naturally introduce the (real-valued) $\H$ projector on
$\tilde F_0(\bb,y)$ by
\begin{equation} \label{pi}
\tilde \pi_0^\bb  (r)=\phi (\tilde W_0(\bb,\cdot), r).
\end{equation}
We also recall that
\begin{equation}\label{pld2}
  \tilde \pi_0^\bb  (\tilde L_\bb r)=0.
\end{equation}
Note that the analysis of the operator $\tilde L_d$ in Azaiez
\cite{Atams15} shows that the eigenfunction in \eqref{F} is in some
sense the only nonnegative direction. In order to control the
projection of any real-valued $r\in \q H$ on the negative part
of the spectrum of $\tilde L_d$, we introduce the following supplemental
of that direction defined by
\begin{equation*}
\tilde{\mathcal{H}}_{-}^d \equiv \{r \in \mathcal{H} \,|\, r\mbox { is
  real-valued}, \tilde \pi_0^d(r)=0\},
\end{equation*}
together with the following bilinear form
\begin{eqnarray}\label{134''}
 \tilde \varphi_{d} (q,r)&=&\int_{-1}^{1} (-\tilde \psi(d,y) q_1r_1+q_1' r_1'(1-y^2)+q_2 r_2 ) \rho dy,
\end{eqnarray}
where $\tilde \psi(d,y)$ is defined in \eqref{defpsi}. Then, we recall
from Proposition 3.7 page 5906 in Azaiez \cite{Atams15} that there
exists $C_0>0$ such that for all $|d|<1$, for all real-valued $r\in
\tilde{\mathcal{H}}^d_{-},$ 
\[
  \frac{1}{C_0} \|r\|_{\mathcal{H} }^2\le \tilde \varphi_d (r,r)\le C_0 \|r\|_\mathcal{H} ^2.
\]

\subsubsection{Decomposition and modulation of the solution}\label{sec2}

From Proposition \ref{propdecomp}, given that the solitons are
decoupled by \eqref{separation}, we see that for each $l=1,\dots,k$,
$e^{-i\theta_l}q$ is locally close to the soliton
$\kappa(d_l)$. Therefore, from the spectral information of Section
\ref{sec1}, this means that our problem has $3k$ nonnegative
directions, namely
\begin{equation}\label{directions}
  \{e^{i\theta_l}\check F_1(d_l,y),\; e^{i\theta_l}\check
  F_0(d_l,y),\; ie^{i\theta_l}\tilde F_0(d_l,y)\; |\; l=1,\dots,k\},
\end{equation}
corresponding to eigenvalues shown in the subscript of the eigenfunctions,
with projections given by 
\begin{equation}\label{components0}
  \{\check  \pi_1^{d_l}[\Re(e^{-i\theta_l}q)] \check F_1(d_l,y),\;
  \check\pi_0^{d_l}[\Re(e^{-i\theta_l}q)] \check F_0(d_l,y),\;
  \tilde \pi_0^{d_l}[\Im(e^{-i\theta_l}q)] \tilde F_0(d_l,y)\;
|\; l=1,\dots,k\}. 
\end{equation}
Since no other nonnegative direction seems to exist, as we asserted in
Section \ref{sec1}, introducing the operator
$\pi_-^{d_1,\theta_1,\dots,d_k,\theta_k}$ (or $\pi_-$ for short) such that
\begin{equation}\label{decompq0}
q(y,s) = \sum_{l=1}^k e^{i\theta_l}\left\{\sum_{\lambda=0}^1\check
  \pi_\lambda^{d_l}[\Re(e^{-i\theta_l}q)] \check F_\lambda(d_l,y) + i\tilde
  \pi_0^{d_l}[\Im(e^{-i\theta_l}q)] \tilde F_0(d_l,y) \right\}
+\pi_-(q)(y,s), 
\end{equation}
it is natural to expect that $\pi_-(q)$
corresponds to the ``negative'' component of the solution.

\medskip

In fact, we will slightly modify the parameters $d_i(s)$ and
$\theta_i(s)$ given in that proposition, in order to kill the
projections of the solutions on the null directions given in Section
\ref{sec1}. More precisely, this is our statement:
\begin{lemma}[A modulation technique]
  \label{lemode0}
  Assume that $k\ge 2$. 
There exist $E_0>0$ and $\epsilon_0>0$ such that for all  $E\ge E_0$ and $\epsilon\le \epsilon_0$, 
if $v\in \H$ and for all $i=1,...,k$, $(\hat d_i,\hat \theta_i)\in(-1,1)\times \bR$ are such that
\begin{equation*}
\hat \zeta_{i+1}-\hat \zeta_i\ge E\mbox{ and }\|\hat q\|_{\H}\le \epsilon
\end{equation*}
where $\hat q = v-\ds\sum_{j=1}^k e^{i\hat \theta_j}(\kappa(\hat d_j),0)$
and $\hat d_i = -\tanh \hat \zeta_i$, then, there exist $(d_j,\theta_j) _{j=1,\dots,k}$ such that for all $i=1,\dots,k$,
\begin{eqnarray*}
  -&& \check \pi_0^{\bb_i}[\Re(e^{-i\theta_i}q)]
      =\tilde \pi_0^{\bb_i}[\Im(e^{-i\theta_i}q)]=0 
\mbox{ where }q=v-\sum_{j=1}^k e^{i\theta_j}(\kappa(d_j),0),\\
-&&\left| \theta_i-\hat \theta_i\right|+|\zeta_i-\hat \zeta_i|\le C\|\hat q\|_{\H}\le C\epsilon,\\
-&&
\zeta_{i+1}-\zeta_i\ge \frac E2\mbox{ and }\|q\|_{\H} \le C\|\hat
    q\|_{\H}\le C\epsilon,
\end{eqnarray*}
where $d_j(s) = - \tanh\zeta_j(s)$. 
\end{lemma}
\begin{rk} 
The proof is in fact  a mixture between the real-valued multi-soliton case given in
Proposition 3.1 page 2847 in \cite{MZdmj12}, and the complex-valued
single-soliton case given in Proposition 4.1 page 5912 in
\cite{Atams15}. In other words, no new idea is needed, and the proof
is only technical. In order to keep the paper in a reasonable length,
we don't give the proof here, and kindly refer the interested reader
to those papers.
\end{rk}
\begin{rk}
  As the proof follows from the application of the implicit function
theorem, this implies in particular that the new parameters
$(d_j,\theta_j) _{j=1,\dots,k}$ are in fact of class $C^1$ in terms of
$v$ and the parameters $(\hat d_j,\hat \theta_j)_{j=1,\dots,k}$.
\end{rk}

From the decomposition of the solution given in Proposition
\ref{propdecomp}, we can apply the modulation technique of Lemma
\ref{lemode0} to derive the existence of a new set of parameters
$(d_l(s),\theta_l(s)) _{l=1,\dots,k}$, (denoted the same for
simplicity), $C^1$ as functions of $s$, so that 
\begin{align}
  - &
  \forall s\ge s_0,\;\;\forall l=1,\dots,k,\;\;
      \check \pi_0^{\bb_l(s)}[\Re(e^{-i\theta_l(s)}q(s))]
      =\tilde \pi_0^{\bb_l(s)}[\Im(e^{-i\theta_l(s)}q(s))]=0, \label{orth}\\
  - & \|q(s)\|_{\H}\to 0\mbox{ and }
       \forall l =1,\dots,k-1,\;\; \zeta_{l+1}(s)-\zeta_l(s) \to
       \infty \mbox{ as }s\to \infty\nonumber
       \end{align}
       where
\begin{equation}\label{defzetaj}
  d_j(s) = - \tanh\zeta_j(s), \mbox{ hence }
  \zeta_j'(s) = - \frac{d_j'(s)}{1-d_j(s)^2}
\end{equation}
  for some $s_0\in \bR$,
with
\begin{equation}\label{defq}
q(y,s) =\vc{q_1(y,s)}{q_2(y,s)}= \vc{w(y,s)}{\partial_s w(y,s)}-\vc{\sum_{j=1}^k e^{i\theta_j(s)}\kappa(d_j(s),y)}{0}.
\end{equation}
Introducing for each $l=1,\dots,k$ and $\lambda=0$ or $1$
\begin{equation}\label{defa1l}
\check \alpha_\lambda^l(s) = \check
\pi_\lambda^{\bb_l(s)}[\Re(e^{-i\theta_l(s)}q(s))],\;
\tilde \alpha_0^l(s) = \tilde \pi_0^{d_l(s)}[\Im(e^{-i\theta_l(s)}q(s))],
\end{equation}
we see from the modulation product given in \eqref{orth} that
\begin{equation}\label{zero}
  \check \alpha_0^l(s) \equiv  \tilde \alpha_0^l(s)  \equiv 0,
  \mbox{ hence }
  \left(\check \alpha_0^l\right)'(s) \equiv
  \left(\tilde \alpha_0^l\right)'(s)  \equiv 0.
\end{equation}
Introducing then
\begin{equation}\label{defq-}
  q_-=\pi_-(q)
\end{equation}
  where the projector
$\pi_-=\pi_-^{d_1,\theta_1,\dots,d_k,\theta_k}$ was introduced in
\eqref{decompq0}, we see from the orthogonality relations in
\eqref{orth} that 
\begin{equation}\label{estq-}
  q(y,s) = \sum_{l=1}^k e^{i\theta_l} \check \alpha_1^l(s) \check F_1(d_l,y) +
  q_-(y,s).
\end{equation}
In particular, 
the control of $q(y,s)$ is equivalent to the control of
\begin{equation}\label{components}
(\check \alpha_1^l(s), d_l(s),\theta_l(s))_{l=1,\dots,k} \mbox{ together with }
q_-(y,s).
\end{equation}
In fact, in order to control $\|q_-(s)\|_{\q H}$, we will
in fact control
\begin{equation}\label{defa-}
A_- = \varphi(q_-, q_-),
\end{equation}
where $\varphi$ is the bilinear
form defined for all $(r, \boldsymbol r)\in \q H^2$ by
\begin{align}
  \varphi(r, \bd r) =
  & \Re \int_{-1}^1
\left(r_2\bd {\bar r}_2+r_1'\bd
                     {\bar r}_1'(1-y^2)+\left(\frac{2(p+1)}{(p-1)^2}-|K|^{p-1}\right)r_1\bd
    {\bar r}_1  \right)\rho dy\nonumber\\
  &-(p-1)\iint |K|^{p-3}\Re(\bar K r_1)\Re(K \bd {\bar r}_1) \rho dy \label{defvarphi}
\end{align}
with
\begin{equation}\label{defK}
K(y,s)=\sum_{j=1}^{k}e^{i\theta_j(s)}\kappa(d_j(s), y).
\end{equation}
As a matter of fact, we have the following result:
\begin{lemma}[Results related to the bilinear form
  $\varphi$]\label{lemequiv} For $s$ large enough, the following holds:\\
  (i) (\textit{Continuity of $\varphi$}) For all $r$ and $\bd r$ in
  $\q H$, we have
  \[
|\varphi(r, \bd r)|\le C\|r\|_{\q H}\|\bd r\|_{\q H}.
  \]
(ii) (\textit{Equivalence of norms})
  \[
    \frac 1 C \|q_-(s)\|_{\H}^2 - C\bar J(s)^2\|q(s)\|_{\H}^2
\le A_-(s)
    \le C\|q_-(s)\|_{\H}^2,
\]
where $A_-(s)$ is introduced in \eqref{defa-} and 
\begin{equation}\label{defjb}
\bar J(s)= \displaystyle \sum_{j=1}^{k-1} (\zeta_{j+1}(s)-\zeta_j(s))e^{-\frac
  2{p-1}(\zeta_{j+1}(s)-\zeta_j(s))},
\end{equation}
with the convention that
\begin{equation}\label{conv}
\zeta_0(s) \equiv -\infty\mbox{ and }\zeta_{k+1}(s) \equiv +\infty.
  \end{equation}
 \end{lemma}
 \begin{proof}
   The proof is similar to the real-valued case. For that case, we
   leave it to Section \ref{appequiv} in the appendix.
 \end{proof}
 \begin{rk}
Note that the bilinear form $\varphi(r,\bd r)$ is the second differential of the
Lyapunov functional defined in \eqref{defE} near the multisoliton
$K(y,s)$. This property is crucial in deriving the differential
inequality on $A_-(s)$ \eqref{defa-} stated below in item (ii) of
Proposition \ref{propdyn}. 
Note also that our definition of the bilinear form in
\eqref{defvarphi} is an important novelty in our paper, because of the
complex structure. Indeed, the preceding work in the complex setting
was the one-soliton case in Azaiez \cite{Atams15} and there, we had
two different real-valued bilinear forms for the real and the
imaginary parts of the solution, after phase shift which takes us
near the real-valued case. Here, with the multisoliton $K(y,s)$
\eqref{defK}, no global phase shift is good, unless all the phases are
the same (up to some multiple of $\pi$), and there is no prior reason to be
in that case (in fact, that will be the case at the end of the proof,
as our statement in Theorem \ref{th1} shows).  In one word, in this
complex-valued multi-soliton setting, we cannot reduce the problem in
such a way so
we fall near the real-valued case, and we need to introduce a
\textit{new} bilinear form in \eqref{defvarphi},
creating this way an original situation in this paper.  
 \end{rk}
 
 \textit{Conclusion of Section \ref{sec2}: From \eqref{components} together
 with Lemma \ref{lemequiv}, we see that the control of $\|q(s)\|_{\q
   H}$ reduces to the control of $(\check \alpha_1^l(s),
 d_l(s),\theta_l(s))_{l=1,\dots,k}$ together with $A_-(s)$}.

\subsubsection{Dynamics of the solution near the multi-soliton}\label{sec3}
From equation \eqref{eqw}, we can write the following equation
satisfied by $q$ \eqref{defq} for all $y\in(-1,1)$ and $s\ge s_0$:
\begin{equation}\label{eqq}
  \partial_s q = Lq
+\begin{pmatrix}0\\ f(q_1)\end{pmatrix}+\begin{pmatrix}0\\ R(y,s)\end{pmatrix}
    -\sum_{j=1}^k\begin{pmatrix}
e^{i\theta_j}(i\theta_j'\kappa(d_j,y)+d_j'\partial_d
  \kappa(d_j))\\0 \end{pmatrix},  
\end{equation}
where
\begin{align}
Lq=
  &\begin{pmatrix}q_2\\\mathcal{L}q_1
+\psi(q_1)-\frac{p+3}{p-1} q_2- 2 y \partial_y q_2 \end{pmatrix}, \label{defL}\\
            \psi(\xi) = & |K|^{p-1}\xi+(p-1)|K|^{p-3}K\Re(\bar K\xi)
                          -\frac{2(p+1)}{(p-1)^2}\xi,\nonumber\\
          f(\xi)=&|K+\xi|^{p-1}(K+\xi)-|K|^{p-1}K-|K|^{p-1}\xi-(p-1)|K|^{p-3}K\Re(\bar
                   K\xi), \nonumber\\
R(y,s)=&|K|^{p-1}K-\sum_{j=1}^{k}e^{i\theta_j}\kappa^{p}(d_j, y),\nonumber
\end{align}
and $K(y,s)$ is defined in \eqref{defK}.
Let us introduce
\begin{equation}\label{defJ}
  J(s) = \sum_{l=1}^{k-1} e^{-\frac 2{p-1}(\zeta_{l+1}(s)-\zeta_l(s))}\mbox{ and }
  \check J(s) =\sum_{l=1}^{k-1} h(\zeta_{l+1}(s)-\zeta_l(s))
\end{equation}
where
\begin{equation*}
h(\zeta)= e^{-\frac p{p-1} \zeta} \mbox{ if }p<2,\;\;h(\zeta)= e^{-2 \zeta} \sqrt \zeta\mbox{ if }p=2 \mbox{ and }h(\zeta)= e^{-\frac 2{p-1}\zeta}\mbox{ if }p>2.
\end{equation*}
Projecting equation \eqref{eqq} along the directions given in
\eqref{directions}, then, using energy techniques, we may write the following
differential inequalities satisfied by the components of the solution
shown in \eqref{components}:
\begin{proposition}[Differential equations satisfied by the different
  components] \label{propdyn}
Assume that $k\ge 2$.
 For all $s$ large enough, the following holds:
\item{(i)}(Control of the positive modes and the modulation
  parameters) For all $i=1,\dots,k$, it holds that
\begin{eqnarray*}
  |\left(\check\alpha_1^i\right)'-\check\alpha_1^{i}|+ |\zeta_i'|+|\theta_i'|\le C J+C \|q\|^2_{\mathcal{ H}}.
\end{eqnarray*}
\item{(ii)} (A differential inequality satisfied by $A_-(s)$)
  \begin{eqnarray}
\left(\frac 12 A_--R_-\right)'&\le& - \frac 3{p-1}\iint |q_{-,2}|^2 \frac \rho{1-y^2} dy
                                    +o\left(\|q(s)\|_{\H}^2\right)
                                    +C\check J(s)^2\nonumber\\
                                   & &
                                          + \; C\; J(s)\sqrt{|A_-(s)|}\label{eqam}
\end{eqnarray}
for some $R_-(s)$ satisfying 
\begin{equation}\label{boundr-}
|R_-(s)|\le C \|q(s)\|_{\H}^{\bar p+1}
\end{equation}
where $\bar p = \min(p,2)$.
\item{(iii)} (An additional relation)
  \[
    \frac d{ds} \Re \iint q_1 \bar q_2 \rho dy
    \le - \frac 45 A_- +C\iint|q_{-,2}|^2 \frac\rho{1-y^2} dy
    +C\sum_{l=1}^k|\check \alpha_1^l|^2 +C \check J^2.    
  \]
\end{proposition}
\begin{rk} The differential inequality satisfied by $A_-(s)$ given in
  item (ii) crucially needs the fact that the nonlinear term in the
  original equation \eqref{equ} is a gradient when $\m C$ is
  identified with $\m R^2$. See the proof for details.
\end{rk}
\begin{proof}[Proof of Proposition \ref{propdyn}]
The argument is similar to the real case studied in Merle and Zaag
\cite{MZajm12}. For that reason, we postpone the proof to Section
\ref{appdyn}.
  \end{proof}
Using the previous proposition, we prove the following:
\begin{proposition}[Size of $q$ in terms of the distance between
  solitons]\label{propsize}
Assume that $k\ge 2$.
  For all $s\ge s_0$, we have
\begin{equation}
  \|q(s)\|_{\mathcal{H}}\le C\check J(s),
\end{equation}
where $\check J$ is introduced in \eqref{defJ}.
\end{proposition}
\begin{proof}
  This statement is the analogous of Proposition 3.8 page 602 in the
  real-valued case given in \cite{MZajm12}, and its proof is quite
  similar, thanks to the equations given in Proposition
  \ref{propdyn}.
  See pages 606-609 in \cite{MZajm12} for details.
\end{proof}

Thanks to the previous proposition, we are in a position to refine the
differential equations satisfied by $\zeta_j(s)$ and $\theta_j(s)$
previously given in Proposition \ref{propdyn}, deriving the following result:
 \begin{proposition}[Refined equations on the solitons' centers and
   phases] \label{propref}
     Assume that $k\ge 2$. 
   There exists $\delta>0$ such that for $s$ large enough and for all $i=1,\dots,k$,
   \begin{align}
     \frac{1}{A} \zeta_i' &=-\cos(\theta_{i-1}-\theta_{i}) e^{-\frac{2}{p-1}  ( \zeta_i-\zeta_{i-1 }) }+ \cos(\theta_{i+1}-\theta_{i}) e^{-\frac{2}{p-1}(\zeta_{i+1}-\zeta_i ) }  + \mathcal{O}(J^{1+\delta}),\label{toda-1}\\
  \frac{1}{B} \theta_i' &=\sin(\theta_{i-1}-\theta_{i}) e^{-\frac{2}{p-1}  ( \zeta_i-\zeta_{i-1 }) }+ \sin(\theta_{i+1}-\theta_{i}) e^{-\frac{2}{p-1}(\zeta_{i+1}-\zeta_i ) } + \mathcal{O}(J^{1+\delta}),\label{toda-2}
   \end{align}
   where $J$ is defined in \eqref{defJ} and
\begin{align}
  A&=\frac{p(p-1)}{2\iint \rho(y)dy}\kappa_0^{p-1}2^{\frac 2{p-1}}\int_{-\infty}^\infty \cosh^{-\frac{2p}{p-1}}(z)e^{\frac {2z}{p-1}}\tanh z dz >0,\label{defA2}\\
  B&=\frac{p-1}{(p+3)\iint \rho(y)dy}\kappa_0^{p-1}2^{\frac
     2{p-1}}\int_{-\infty}^\infty \cosh^{-\frac{2p}{p-1}}(z)e^{\frac {2z}{p-1}} dz.\nonumber
\end{align}
\end{proposition}
\begin{proof}
  See Section \ref{appdyn}.
\end{proof}
\begin{rk}
We still use the convention \eqref{conv}, together with the convention
$\theta_0\equiv\theta_{k+1}\equiv 0$.
\end{rk}
\begin{rk}
  Neglecting the error terms in  \eqref{toda-1}-\eqref{toda-2}, we see
  that the second equation has the following trivial solution
  $(\theta_i(s))_i=(\theta_{00}+i\pi)_i$, for some $\theta_{00}\in \m R$, regardless of the value of $(\zeta_1,\dots, \zeta_k)$. In that case, the first equation reduces
  to the real-valued first order Toda system \eqref{toda}. As a matter
  of fact, the constant $A$ mentioned here and there is the same. For
  that reason, we call system \eqref{toda-1}-\eqref{toda-2} \textit{``the
  complex-valued first order Toda system''}. Note that
the core of
our paper is to prove that asymptotically as $s\to \infty$,
system \eqref{toda-1}-\eqref{toda-2} behaves like the real-valued case
\eqref{toda}, in the sense that $\theta_{i+1}-\theta_i \to \pi$ and
$\zeta_i - [\zeta_{00}+\bar \zeta_i(s)]\to 0$ as $s\to \infty$ for
some constant $\zeta_{00}$, where $(\bar \zeta_1(s),\dots,\bar
\zeta_k(s))$ is the solution of \eqref{toda} mentioned right before
that equation. This fact will be proved in Proposition \ref{prop:edo-main} in
Section \ref{SecToda} below.
\end{rk}

\section{Projection of the PDE on the different components}\label{appdyn}
This section is devoted to the proof of different statements following
from the projection of equation \eqref{eqq} on the components
mentioned in \eqref{components}, namely, Propositions \ref{propdyn}
and \ref{propref} (note that Proposition \ref{propsize} follows
exactly as in the real-valued case). Most of the estimates follow as in
the real-valued case studied in \cite{MZajm12}, and others need some
delicate treatment. Therefore, we will recall the former estimates and
focus on the latter. We proceed in two subsections, each dedicated to
the proof of one proposition. Of course, as in Section \ref{secstep2},
we assume here that
\[
k\equiv k(x_0)\ge 2,
\]
where $x_0\in \q S$ and  $k(x_0)$ is the number of solitons appearing
in the decomposition of $w_{x_0}(y,s)$ given in Proposition \ref{propdecomp}. 


\subsection{Differential inequalities on the different components}
We prove here Proposition \ref{propdyn}.
\begin{proof}[Proof of Proposition \ref{propdyn}]
$ $\\
(i) Consider some $l=1,\dots,k$. The proof follows from the projection
of equation \eqref{eqq} satisfied by $q(y,s)$ \eqref{defq} on the 3
directions $e^{i\theta_l}\check F_1(d_l,y)$, $e^{i\theta_l}\check
  F_0(d_l,y)$ and $e^{i\theta_l}\tilde F_0(d_l,y)$ shown in
  \eqref{directions}. Since the components along those directions
  involve projections of 
\begin{equation}\label{defql}
q^l\equiv e^{-i\theta_l}q,
\end{equation}
as one can see from \eqref{components0}, 
it is natural to write the following PDE satisfied by $q^l(y,s)$:
\begin{align} 
\frac{\partial }{\partial s} \begin{pmatrix} \qlun\\\qldeux \end{pmatrix} =&\; L_{d_l}\begin{pmatrix}\qlun\\\qldeux\end{pmatrix}+\begin{pmatrix}0\\  V_l(y,s,\qlun)\end{pmatrix}+\begin{pmatrix}0\\ G_l(\qlun)\end{pmatrix}+\begin{pmatrix}0\\ R_l(y,s)\end{pmatrix}-i\theta_l'\begin{pmatrix} \qlun(y,s)\\\qldeux(y,s)\end{pmatrix}\label{linearise-complexe}\\
&+\begin{pmatrix}-i\theta_l'\kappa(d_l)-d_l' \partial_d
  \kappa(d_l)-\sum_{j\neq
    l}e^{i(\theta_j-\theta_l)}(i\theta_j'\kappa(d_j,y)+d_j'\partial_d
  \kappa(d_j))\\0 \end{pmatrix},  \nonumber
\end{align}
where
\begin{align}
L_{d_l}\begin{pmatrix}\qlun\\\qldeux\end{pmatrix}=&\begin{pmatrix}\qldeux \\\mathcal{L}\qlun-\frac{2(p+1)}{(p-1)^2}\qlun+p \kappa^{p-1}(d_l)\check \qlun+i\kappa^{p-1}(d_l)\tilde \qlun-\frac{p+3}{p-1} \qldeux- 2 y \partial_y \qldeux \end{pmatrix}, \nonumber\\
V_l(y,s,\qlun)=\;&|K|^{p-1}\qlun+(p-1) |K|^{p-3}K
e^{-i\theta_l}[\qlun \cdot (e^{-i\theta_l} K)]-p \kappa(d_l)^{p-1}\check \qlun-i\kappa(d_l)^{p-1}\tilde \qlun,\nonumber \\
   G_l(\qlun)=\;&e^{-i\theta_l}{f}(e^{i\theta_l} \qlun),\label{defG}\\
          f(\xi)=\;&|K+\xi|^{p-1}(K+\xi)-|K|^{p-1}K-|K|^{p-1}\xi-(p-1)|K|^{p-3}K(\xi
            \cdot K), \nonumber\\
R_l(y,s) =\;&e^{-i\theta_l}R(y,s),\nonumber\\
R(y,s)=\;&|K|^{p-1}K-\sum_{j=1}^{k}e^{i\theta_j(s)}\kappa^{p}(d_j(s), y)\nonumber
\end{align}
and $K(y,s)$ is defined in \eqref{defK}.
Note that the notation $(X \cdot Y)$ stands for the standard inner
product in $\m R^2$, and that whenever $X\in \m C$, we identify it
with $(\Re X, \Im X)$ which is in $\m R^2$.\\
%
%
Dissociating the real and imaginary parts in
\eqref{linearise-complexe} as in \eqref{riz}, we write
\begin{equation}\label{riql}
  q^l=\check q^l + i \tilde q^l\mbox{ and }
  \|q^l\|_{\q H}^2 = \|\check q^l\|_{\q H}^2 + \|\tilde q^l\|_{\q H}^2, 
\end{equation}
then we get for all $s\ge s_0$:
\begin{eqnarray}\label{170'}\frac{\partial }{\partial s}\begin{pmatrix} \check \qlun\\\check \qldeux
\end{pmatrix}&=\check L_{d_l} \begin{pmatrix} \check \qlun \\\check \qldeux \end{pmatrix}+\begin{pmatrix}0\\ \check V_l(y,s,\qlun)\end{pmatrix}+\begin{pmatrix}0\\ \check G_l(\qlun)\end{pmatrix}+\begin{pmatrix}0\\ \check R_l(y,s)\end{pmatrix}+\theta_l'\begin{pmatrix} \tilde \qlun(y,s)\\ \tilde \qldeux(y,s)\end{pmatrix}\notag\\&+\begin{pmatrix}-d_l' \partial_d \kappa(d_l)-\sum_{j\neq l}\left(\cos(\theta_j-\theta_l) d_j'\partial_d \kappa(d_j)-\sin (\theta_j-\theta_l)\theta_j'(s)\kappa(d_j(s),y)\right)\\0 \end{pmatrix}
\end{eqnarray}
and
\begin{align}\label{170''}&\frac{\partial }{\partial s} \begin{pmatrix} \tilde \qlun\\\tilde \qldeux
\end{pmatrix}=\tilde L_{\bb_l}\begin{pmatrix} \tilde \qlun\\\tilde \qldeux \end{pmatrix}+\begin{pmatrix}0\\ \tilde V_l(y,s,\qlun)\end{pmatrix}+\begin{pmatrix}0\\ \tilde G_l(\qlun)\end{pmatrix}+\begin{pmatrix}0\\ \tilde R_l(y,s)\end{pmatrix}-\theta_l'\begin{pmatrix} \check \qlun(y,s)\\ \check \qldeux(y,s)\end{pmatrix}\notag\\&+\begin{pmatrix}-\theta_l'  \kappa(d_l)+\sum_{j\neq l}\cos(\theta_j-\theta_l) \theta_j'(s)\kappa(d_j(s),y)+\sin (\theta_j-\theta_l)d_j'(s)\partial_d \kappa(d_j(s),y)\\0 \end{pmatrix}
\end{align}
where the operators $\check L_d$ and $\tilde L_d$ are introduced in
\eqref{barL_bb } and \eqref{tildeL_bb }.

\medskip

The proof will follow from projecting equation \eqref{170'} with the
projector $\check \pi^{d_l(s)}_\lambda$ \eqref{barpi} for $\lambda =0$
and $1$, and also the projection of equation \eqref{170''} with the
projector $\tilde \pi^{d_l(s)}_0$ defined in \eqref{pi}. The following
2 steps are respectively devoted to the projections of equations
\eqref{170'} and \eqref{170''}.

\bigskip

{\bf Step 1: Projection of equation (\ref{170'}) with
  $\check \pi^{d_l(s)}_\lambda$}

 Projecting equation (\ref{170'}) with the projector $\check \pi_\lambda^{\bb_l} $ (\ref{barpi}) for $\lambda=0$ and $\lambda=1$, we write
 \begin{eqnarray}
   &\check \pi_\lambda^{\bb_l}( \partial_s\check q_l)=\check
     \pi_\lambda^{\bb_l}( \check L_d\check q_l)+\check
     \pi_\lambda^{\bb_l} \begin{pmatrix}0\\ \check
       V_l(y,s,\qlun)\end{pmatrix}+\check
   \pi_\lambda^{\bb_l} \begin{pmatrix}0\\ \check
     G_l(\qlun)\end{pmatrix}+\check \pi_\lambda^{\bb_l} \begin{pmatrix}0\\
     \check R_l(y,s)\end{pmatrix}+\theta_l'\check
   \pi_\lambda^{\bb_l} \begin{pmatrix} \tilde \qlun(y,s)\\ \tilde
     \qldeux(y,s)\end{pmatrix}\notag\\
   & +\check \pi_\lambda^{\bb_l} \begin{pmatrix}-d_l' \partial_d \kappa(d_l)-\sum_{j\neq l}\left( \cos(\theta_j-\theta_l) d_j'\partial_d \kappa(d_j)-\sin (\theta_j-\theta_l)\theta_j'(s)\kappa(d_j(s),y) \right)\\0 \end{pmatrix}\label{projpdl}
\end{eqnarray}
Many projections in the previous line were already estimated in the
previous literature. To save space, we give those estimates with no
proof and kindly refer the reader to those papers for a
justification.

\medskip

- From pages 105 to 106 in the real-valued one-soliton case in \cite{MZjfa07}, together with
\eqref{defa1l}, \eqref{defzetaj} and \eqref{riql}, it follows that
\begin{align}
\left|\check \pi_\lambda^{\bb_l(s)} ( \partial_s \check q^l(s))-
\left(\check\alpha_\lambda^l(s)\right)'\right|&\le \frac{C_0}{1-\bb_l(s)^2} |\bb_l'(s)|\cdot
\|\check q^l(s)\|_{\mathcal{ H}} \le C_0 |\zeta_l'(s)|\cdot \|
q(s)\|_{\mathcal{ H}},\nonumber\\
\check \pi_\lambda^{\bb_l(s)} (\check L_{\bb(s)}(\check q^l(s)))&=\lambda \check
                                                 \alpha_\lambda^l(s),\nonumber\\
\bb_l'(s)\check \pi_\lambda^{\bb_l(s)}\begin{pmatrix}\partial_d
  \kappa(d_l(s))\\0\end{pmatrix}&=\frac{2\kappa_0}{p-1}\zeta_l'(s)\delta_{\lambda,0},\label{adel}
\end{align}
where $\check \alpha_\lambda^l(s)$ is defined in \eqref{defa1l} and
the Kronecker symbol satisfies hereafter $\delta_{0,0}=1$ and $\delta_{1,0}=0$.

\medskip

- From pages 631 to 635 in the real-valued multi-soliton case in
\cite{MZajm12}, together with \eqref{defzetaj}, \eqref{defql} and
\eqref{defG}, we have
\begin{align}
  \left|\check \pi_\lambda^{\bb_l(s)} \begin{pmatrix}0\\\check
      G_l(\qlun)\end{pmatrix}\right|
 \le C\iint \kappa(d_l(s),y)|f(q_1)|\rho(y)dy
  &\le C {\m 1}_{\{p\ge 2\}}\|q\|_{\mathcal{ H}}^p + C
    J_l(s)\|q\|_{\mathcal{ H}}^2,\label{nl}\\
 \left| \check \pi_\lambda^{\bb_l(s)} \begin{pmatrix}0\\   \check
     R_l(y,s)\end{pmatrix}\right|
\le C \iint \kappa(d_l(s),y)|R(y,s)|\rho(y)dy
  & \le CJ(s),\label{rl}\\
  \left|d_j'(s)\check
  \pi_\lambda^{d_l(s)}\begin{pmatrix}
     \partial_d\kappa(d_j(s))\\0 \end{pmatrix}\right|
  \le C|d_j'(s)|\iint \frac{\kappa(d_l,y)}{1-y^2} |\partial_d \kappa(d_j,y)| \rho(y) dy&\nonumber\\
\le C |\zeta_j'(s)|\int_{-1}^1\frac{\kappa(d_l,y)}{1-y^2} \kappa(d_j,y) \rho(y) dy
  &\le C \bar  J(s)  |\zeta_j'(s)|, \label{adam}
\end{align}
where $j\neq l$,
\begin{equation}\label{defjis}
J_l(s) = \iint \kappa(d_l(s),y)|K(y,s)|^{p-2} dy,
\end{equation}
$K(y,s)$, $J(s)$ and $\bar J(s)$ are defined in \eqref{defK}, \eqref{defJ} and
\eqref{defjb}, and the notation ${\m 1}_{\{p\ge 2\}}$ stands hereafter
for the indicator function of the set $\{p\ge 2\}$. Moreover, following our estimates in the same place in \cite{MZajm12}, we know that
\begin{equation}\label{leila}
  \left||K|^{p-1}-\kappa(d_l(s),y)^{p-1}\right|
  +\left||K|^{p-3}[e^{-i\theta_l(s)} K]^2-\kappa(d_l(s),y)^{p-1}\right|\le
  C \bar V(y,s)
\end{equation}
 where
\begin{equation}\label{defvb}
\bar V(y,s)=\left\{\un{y_{l-1}}{y_l} \sum_{j\neq
  l}\kappa(d_l,y)^{p-2}\kappa(d_j,y) + \sum_{j\neq l}
\kappa(d_j,y)^{p-1}\un{y_{j-1}}{y_j}\right\},
\end{equation}
$y_0=-1$, $y_j= \tanh(\frac{\zeta_j+\zeta_{j+1}}2)$ if $j=1,..,k-1$ and $y_k=1$. 
We also have the following:
\begin{cl}[Projection of the potential term with $\check \pi_\lambda^{d_l(s)}$]\label{clkooli}
Whenever $V=V(y,s)$ satisfies $|V(y,s)|\le C |q_1(y,s)|\bar V(y,s)$
where $\bar V$ is defined in \eqref{defvb}, it holds that
  \begin{equation*}
    \left| \check\pi_\lambda^{\bb_l(s)}\vc{0}{V}\right|
\le C \iint \kappa(d_l(s),y)|V(y,s)|\rho(y)dy
    \le C \|q(s)\|_{\mathcal{ H}}^2+CJ(s)^{1+\delta},
\end{equation*}
where $J(s)$
is defined in \eqref{defJ} and $\delta>0$.
\end{cl}

\medskip

- From page 5917 in the complex-valued one-soliton case in
\cite{Atams15}, together with \eqref{riql}, we have
\begin{equation}\label{boris}
\left|\check \pi_\lambda^{\bb_l(s)}  \begin{pmatrix}
    \tilde \qlun(s)\\\tilde\qldeux(s) \end{pmatrix}\right|
\le C \|q(s)\|_{\mathcal{ H}}.
\end{equation}

From these estimates, we see that some projections of \eqref{projpdl}
are already controlled.
In the following part, we estimate the other projections.

\medskip

- \textit{Estimate of $\check \pi_\lambda^{\bb_l(s)} \begin{pmatrix}0\\ \check
       V_l(y,s,\qlun)\end{pmatrix}$}. By definition \eqref{defG} of
   $V_l$, recalling the decomposition \eqref{riql} of $\qlun$, we
   write
   \begin{align*}
  V_l(y,s,\qlun)=W_1+W_2+W_3
   \end{align*}
   where
   \begin{align*}
     W_1= & [|K|^{p-1}- \kappa(d_l)^{p-1}]\qlun,\;
     W_2 = (p-1)[[|K|^{p-3}[e^{-i\theta_l}K]^2 -
     \kappa(d_l)^{p-1}]\check\qlun,\\
     W_3 = &(p-1)|K|^{p-3}K e^{-i\theta_l}
                   \left\{[\qlun \cdot (e^{-i\theta_l} K)] -
                   e^{-i\theta_l}K\check \qlun\right\}.
   \end{align*}
   Since
   $|e^{-i\theta_l}K-\kappa(d_l(s))|\le \sum_{j\neq l}\kappa(d_j(s))$
   by definition \eqref{defK} of $K$, using again \eqref{riql}, we see that
\[
|W_3|\le C|K|^{p-2}|\qlun|\sum_{j\neq l}\kappa(d_j(s)).
\]
Proceeding as in the proof of \eqref{leila} (which follows from pages
633 to 635 in \cite{MZajm12}), we easily see that $|W_3|\le C |q_1(y,s)|\bar
V(y,s)$ where $\bar V$ is defined in \eqref{defvb}. Using \eqref{leila}, we see that 
\begin{equation}\label{marzouk}
|V_l(y,s,\qlun)|\le C  |q_1(y,s)|\bar V(y,s).
\end{equation}
Using Claim \ref{clkooli}, we see that
\begin{equation}\label{hassine}
\left|\check \pi_\lambda^{\bb_l(s)} \begin{pmatrix}0\\ \check
    V_l(y,s,\qlun)\end{pmatrix}\right|
\le C \iint \kappa(d_l(s),y)| V_l(y,s,\qlun)|\rho(y) dy
\le C \|q(s)\|_{\mathcal{ H}}^2+CJ(s)^{1+\delta}
\end{equation}
for some $\delta>0$.

\medskip

- \textit{Estimate of $\check\pi_\lambda^{\bb_l(s)} \begin{pmatrix}0\\ \check
    G_l(\qlun)\end{pmatrix}$}.

The following lemma allows us to get the good estimate:
\begin{lemma}[Control of $J_i(s)$]\label{lemjis}
For any $s\ge s_0$ and $i=1,...,k$, it holds that $J_i(s)\le C$, where $J_i(s)$
is defined in \eqref{defjis}.
\end{lemma}
Indeed, thanks to this lemma, using \eqref{nl} and the boundedness of $\|q(s)\|_{\q H}$ (see
\eqref{orth}), we see that
\begin{equation}\label{nl1}
\left|\check\pi_\lambda^{\bb_l(s)} \begin{pmatrix}0\\ \check
    G_l(\qlun)\end{pmatrix}\right|
\le C\iint \kappa(d_l(s),y)|f(q)|\rho(y)dy
\le C \|q(s)\|_{\mathcal{ H}}^2.
\end{equation}
Let us then prove this lemma.
\begin{proof}[Proof of Lemma \ref{lemjis}]
 Let us recall that the real-valued case was handled in Lemma E.1 page 644 in
\cite{MZajm12}. In fact, because of the phase factor in the definition
\eqref{defK} of $K(y,s)$ in this complex-valued case, the situation is
more delicate. Anyway, starting as in \cite{MZajm12}, we use the
change of variables $y=\tanh \xi$ and write 
\begin{equation}\label{defkb}
J_i = \kappa_0^{p-1}\int_{\m R} \cosh^{-\frac
  2{p-1}}(\xi-\zeta_i)|\kz(\xi,s)|^{p-2} d\xi
\mbox{ where } \kz(\xi,s) = \sum_{j=1}^k e^{i\theta_j(s)}\cosh^{-\frac 2{p-1}}(\xi-\zeta_j).
\end{equation}
If $p \ge 2$, then $|\kz(\xi,s)| \le C$ and $|J_i(s)| \le C$.\\
If $p <2$,
taking advantage of the decoupling in the sum of the solitons (see \eqref{orth}), we write
\begin{equation}\label{b0} 
J_i = \kappa_0^{p-1}\sum_{j=1}^k \int_{\tt_{j-1}+A}^{\tt_j+A} \cosh^{-\frac 2{p-1}}(\xi-\zeta_i)|\kz(\xi,s)|^{p-2} d\xi  
\end{equation}
where $\tt_0 = - \infty$, $\tt_j = \frac{\zeta_j+\zeta_{j+1}}2$  if $j=1,..,k-1$, $\tt_k=\infty$ and $A=A(p)$ is fixed such that
\begin{equation}\label{mehdi}
e^{\frac {2A}{p-1}} \ge 5 e^{-\frac {2A}{p-1}}.
\end{equation}
This partition isolates each soliton in the definition of
$\kz(\xi,s)$. We don't take $A=0$, since in that case, for 
$j=1,\dots,k-1$, the solitons
number $j$ and $j+1$ are equal at $\xi =\tt_j$, by symmetry. With $A>0$ and large
as in \eqref{mehdi}, the $(j+1)$-th soliton is larger than the $j$-th.
Since $\cosh^{-\frac 2{p-1}}(\xi-\zeta_i)\le \cosh^{-\frac
  2{p-1}}(\xi-\zeta_j)$ on the interval $(\tt_{j-1}+A,\tt_j+A)$ for
large $s$, we write from \eqref{b0}
\begin{equation}\label{b1}
J_i \le \kappa_0^{p-1}\sum_{j=1}^k L_j\mbox{ where }
L_j=\int_{\tt_{j-1}+A}^{\tt_j+A} \cosh^{-\frac 2{p-1}}(\xi-\zeta_j)|\kz(\xi,s)|^{p-2} d\xi.
\end{equation}
Consider now some $j=1,...,k$, and let us work for $\xi$ in the
interval $(\tt_{j-1}+A,\tt_j+A)$ in order to bound the $L_j$ term
shown in \eqref{b1}. Since the exponent $p-2<0$, the idea is to bound $|\kz(\xi,s)|$  from below.\\
Note first from \eqref{mehdi} and \eqref{orth} that whenever $j\ge 2$, it follows that
\[
\cosh^{-\frac 2{p-1}}(\xi-\zeta_{j-1}(s)) \le \frac 14 \cosh^{-\frac 2{p-1}}(\xi-\zeta_j(s))
\]
for $s$ large enough.
Moreover, from the decoupling of the solitons shown in \eqref{orth},
it follows that for $s$ large enough,
\[
\sum_{|l-j|\ge 2} \cosh^{-\frac 2{p-1}}(\xi-\zeta_l(s)) \le \frac 14 \cosh^{-\frac 2{p-1}}(\xi-\zeta_j(s)).
\]
Now, if $j=k$, using a triangular inequality, we write
\begin{equation}\label{zahira}
  |\kz(\xi,s)|\ge
  \cosh^{-\frac 2{p-1}}(\xi-\zeta_k)
  - \cosh^{-\frac 2{p-1}}(\xi-\zeta_{k-1})
  - \sum_{l=1}^{k-2} \cosh^{-\frac 2{p-1}}(\xi-\zeta_l)
  \ge \frac 12 \cosh^{-\frac 2{p-1}}(\xi-\zeta_k).
\end{equation}
By definition \eqref{b1}, we see that
\[
L_k \le C\int_{\tt_{k-1}+A}^\infty \cosh^{-2}(\xi-\zeta_k(s)) d\xi \le C.
\]
Now, if $j\le k-1$, proceeding similarly,
we write
\begin{align}
   |\kz(\xi,s)|
  \ge &\;|e^{i\theta_j}\cosh^{-\frac 2{p-1}}(\xi-\zeta_j) +
  e^{i\theta_{j+1}}\cosh^{-\frac 2{p-1}}(\xi-\zeta_{j+1})|\nonumber\\
&   - \cosh^{-\frac 2{p-1}}(\xi-\zeta_{j-1})
        - \sum_{|l-j|\ge 2} \cosh^{-\frac 2{p-1}}(\xi-\zeta_l)\nonumber\\
  \ge &\;  |e^{i\theta_j}\cosh^{-\frac 2{p-1}}(\xi-\zeta_j) +
        e^{i\theta_{j+1}}\cosh^{-\frac 2{p-1}}(\xi-\zeta_{j+1})|
        - \frac 12 \cosh^{-\frac 2{p-1}}(\xi-\zeta_j).\label{rahma}
\end{align}
Let us now recall the following consequence of the Cauchy-Schwarz
inequality, for any nonnegative $X$ and $Y$:
\begin{equation}\label{cs}
|e^{i\theta_j}X + e^{i\theta_{j+1}}Y|^2
\ge (X^2+Y^2)\min(1,1+\cos(\theta_{j+1}-\theta_j)).
\end{equation}
Accordingly, we distinguish between 2 cases:\\
\textbf{- Case 1}: If $\cos(\theta_{j+1}-\theta_j)\ge - \frac 7{16}$, then we
see from \eqref{cs} that
\[
  |e^{i\theta_j}X + e^{i\theta_{j+1}}Y|\ge \frac 34 X,
\]
hence, using \eqref{rahma}, we write
\begin{equation}\label{low}
  |\kz(\xi,s)| \ge \frac 14 \cosh^{-\frac 2{p-1}}(\xi-\zeta_j),
\end{equation}
and $J_j(s) \le C$ as for the case $j=k$.\\
\textbf{- Case 2}: If $\cos(\theta_{j+1}-\theta_j)\le - \frac 7{16}$,
then restricting the integration to the interval $(\tt_{j-1}+A,
\tt_j-A)$, we remark that
\[
\cosh^{-\frac 2{p-1}}(\xi-\zeta_{j+1}(s)) \le \frac 14 \cosh^{-\frac 2{p-1}}(\xi-\zeta_j(s)),
\]
thanks to \eqref{mehdi}. Therefore, using a triangular inequality as
for \eqref{zahira}, we see that \eqref{low} holds again, hence, 
\[
\int_{\tt_{j-1}+A}^{\tt_j-A} \cosh^{-\frac
  2{p-1}}(\xi-\zeta_j)|\kz(\xi,s)|^{p-2} d\xi
\le C
  \]
as before. It remains then to bound the integral on the interval
$(\tt_j-A,\tt_j+A)$. Accordingly, we work only in that interval in
what follows.
Let us first write 
  \[
    |\kz(\xi,s)| \ge |\ku(\xi,s)|,
  \]
  where
  \[
    \ku(\xi,s) = \Re(e^{-i\theta_j(s)} \kz(\xi,s))
    =\cosh^{-\frac 2{p-1}}(\xi-\zeta_j)
    +\sum_{l\neq j} \cos(\theta_l-\theta_j) \cosh^{-\frac 2{p-1}}(\xi-\zeta_l).
  \]
Using the decoupling in the solitons' sum (see \eqref{orth}), we write
\begin{align}
  \ku(\xi,s)
   =& \;\cosh^{-\frac 2{p-1}}(\xi-\zeta_j)
   +\cos(\theta_{j+1}-\theta_j) \cosh^{-\frac 2{p-1}}(\xi-\zeta_{j+1})\nonumber\\
&   +\sum_{l\not\in\{j,j+1\}} \cos(\theta_l-\theta_j) \cosh^{-\frac 2{p-1}}(\xi-\zeta_l)\nonumber\\
=&\;  2^{-\frac 2{p-1}}e^{-\frac 1{p-1}(\zeta_{j+1}-\zeta_j)}[e^{-\frac {2(\xi-\tt_j)}{p-1}}+ \cos(\theta_{j+1}-\theta_j) e^{\frac {2(\xi-\tt_j)}{p-1}}+o(1)]\label{kexp}
 \end{align}
 as $s\to \infty$. In particular, using \eqref{mehdi} and the fact that $-1\le \cos
(\theta_{j+1}-\theta_j) \le - \frac 7{16}$, it holds that
 \begin{align*}
   \ku(\tt_j-A,s) & = 2^{-\frac 2{p-1}} e^{-\frac 1{p-1}(\zeta_{j+1}-\zeta_j)}[e^{\frac {2A}{p-1}}+ \cos(\theta_{j+1}-\theta_j) e^{-\frac {2A}{p-1}}+o(1)]>0,\\
\ku(\tt_j+A,s) & = 2^{-\frac 2{p-1}} e^{-\frac 1{p-1}(\zeta_{j+1}-\zeta_j)}[e^{-\frac {2A}{p-1}}+ \cos(\theta_{j+1}-\theta_j) e^{\frac {2A}{p-1}}+o(1)]<0,
\end{align*}
as $s\to \infty$. Therefore, there exists  $\check \xi_j(s) \in
(\tt_j-A,\tt_j+A)$ such that $\ku(\check\xi_j(s),s)=0$. Since
\eqref{kexp} holds also after differentiation, we easily see that
\begin{align*}
  \partial_\xi \ku(\xi,s)
  & \le  \frac 2{p-1} 2^{-\frac 2{p-1}}
                                e^{-\frac
                                1{p-1}(\zeta_{j+1}-\zeta_j)}[-e^{-\frac
                                {2A}{p-1}}+
                                \cos(\theta_{j+1}-\theta_j) e^{-\frac
    {2A}{p-1}}+o(1)]\\
  & \le -  \frac 2{p-1} 2^{-\frac 2{p-1}}
                                e^{-\frac
                                1{p-1}(\zeta_{j+1}-\zeta_j)}e^{-\frac
                                {2A}{p-1}} (1+\frac 7{16})/2<0
\end{align*}
for $s$ large enough. Accordingly,
\[
|\kz(\xi,s)|^{p-2} = |\kz(\xi,s)-\kz(\check
\xi_j(s),s)|^{p-2} \le C(A)|\xi-\check \xi_j(s)|^{p-2} e^{-\frac{p-2}{p-1}(\zeta_{j+1}-\zeta_j)}.
\]
 Therefore, since $\cosh^{-\frac 2{p-1}}(\xi-\zeta_j) \le C(A) \cosh^{-\frac 2{p-1}}(\tt_j-\zeta_j)\le C(A) e^{-\frac{\zeta_{j+1} - \zeta_j}{p-1}}$, 
 it follows that
\begin{eqnarray*}
\int_{\tt_j-A}^{\tt_j+A} \cosh^{-\frac 2{p-1}}(\xi-\zeta_j)|
  \kz(\xi,s)|^{p-2} d\xi &\le& C(A) e^{-(\zeta_{j+1}-
                             \zeta_j)}\int_{\mu_j-A}^{\mu_j+A}|\xi-\check
                             \xi_j(s)|^{p-2} d\xi\\
& \le&  C(A) e^{-(\zeta_{j+1}- \zeta_j)}
\end{eqnarray*}
because $|\check \xi_j(s) -\tt_j(s)|\le A$. This concludes the proof
of Lemma \ref{lemjis}.

\end{proof}

\medskip

- \textit{Estimate of $\check \pi_\lambda^{\bb_l(s)}\begin{pmatrix}
    \kappa(d_j(s),y)\\0 \end{pmatrix}$ for $j\neq l$}.
By definitions \eqref{barpi} and \eqref{defphi} of $\check
\pi_\lambda$ and the inner product $\phi$, we get after an
integration by parts
\[
\check \pi_\lambda^{\bb_l(s)}\begin{pmatrix}
    \kappa(d_j(s),y)\\0 \end{pmatrix} = \iint \kappa(d_j,y)(-\mathcal{L}\check  W_{0,1}(d_l,y)+\check W_{0,1}(d_l,y)) \rho(y) dy.
\]
Since we have from the definitions \eqref{defk} and \eqref{defwl2} of
$\kappa(d_l,y)$ and $\check W_{\lambda,2}(d_l,y)$
together with \eqref{elliptic} that
$$|\check W_{0,2}(d_l,y)|+|\mathcal{L}\check W_{0,1}(d_l,y)-\check
W_{0,1}(d_l,y)| \le C \frac{\kappa(d_l,y)}{1-y^2},$$
using \eqref{adam}, we get
\begin{equation}\label{giada}
\left|\check \pi_\lambda^{\bb_l(s)}\begin{pmatrix} \kappa(d_j(s),y)\\0 \end{pmatrix}\right|\le C\int_{-1}^1\kappa(d_j,y) \frac{\kappa(d_l,y)}{1-y^2} \rho(y)
    dy \le C \bar J(s)
\end{equation}
defined in \eqref{defjb}.

\medskip

Compiling all the previous information for $\lambda=0$ and $\lambda=1$, and recalling that
$\left(\check \alpha_0^l\right)'(s)\equiv 0$ by \eqref{zero}, we obtain 
\begin{eqnarray}\label{n}
 \left|\left(\check\alpha_1^l\right)'-\check\alpha_1^l\right| + |\zeta_l'|\le C|\zeta_l'|\cdot\|q\|_\mathcal{H}+ C J+C \|q\|^2_{\mathcal{ H}}+ C|\theta_l'|\cdot\|q\|_{\mathcal{ H}}+ C \bar J \sum_{j\neq l}\left(|\zeta'_j|+|\theta_j'|\right),
\end{eqnarray}
where $\bar J$ and $J$ are defined in \eqref{defjb} and \eqref{defJ}.

\bigskip

{\bf Step 2: Projection of equation (\ref{170''}) with the projector
  $\tilde \pi_0^{d_l(s)}$}\\
 Projecting equation (\ref{170''}) with the projector $\tilde \pi_0^{\bb_l(s)}$ (\ref{barpi}), we write
 \begin{align}
   &\tilde \pi_0^{\bb_l(s)}(\partial_s \tilde q_l)=\tilde
               \pi_0^{\bb_l(s)}(\tilde L_d \tilde q_l)+\tilde
               \pi_0^{\bb_l(s)} \begin{pmatrix}0\\ \tilde
                 V_l(y,s,\qlun)\end{pmatrix}
             +\tilde
  \pi_0^{\bb_l(s)}\begin{pmatrix}0\\ \tilde
    G_l(\qlun)\end{pmatrix}+\tilde \pi_0^{\bb_l(s)}\begin{pmatrix}0\\
    \tilde R_l(y,s)\end{pmatrix}\nonumber\\
   &+\theta_l'\tilde
  \pi_0^{\bb_l(s)} \begin{pmatrix} \check \qlun(y,s)\\ \check
    \qldeux(y,s)\end{pmatrix}\label{projstep2}\\
             &+\tilde \pi_0^{\bb_l(s)}\begin{pmatrix}-\theta_l'  \kappa(d_l,y)+\sum_{j\neq l}\left( \cos(\theta_j-\theta_l) \theta_j'(s)\kappa(d_j(s),y)+\sin (\theta_j-\theta_l)d_j'(s)\partial_d \kappa(d_j(s),y)\right)\\0 \end{pmatrix}.\nonumber
\end{align}
From \eqref{zero} and the study in the case of one soliton given in page 5917
of \cite{Atams15}, we have
\begin{align}
  \left|\tilde \pi_0^{\bb_l(s)}(\partial_s \tilde q_l)\right|
  &\le C |\zeta_l'(s)|\cdot \|q(s)\|_{\mathcal{H}},&
  \tilde \pi_0^{\bb_l(s)}(\tilde L_d \tilde q_l)&=0,\label{pascaline}\\
  \left|\tilde \pi_0^{\bb_l(s)} \begin{pmatrix} \check
      \qlun(y,s)\\ \check \qldeux(y,s)\end{pmatrix}\right|
  &\le C\| q(s)\|_{\mathcal{ H}},&
  \tilde \pi_0^{\bb_l(s)}
  \begin{pmatrix} \kappa(d_l)\\0\end{pmatrix} & =1.\label{neveu}
\end{align}
Recalling the definition \eqref{pi} of $\tilde \pi_0^{d_l(s)}$ and the
fact that
\[
  \tilde W_{0,2}(d_l.y) = \tilde c_0 \kappa(d_l,y) \mbox{ and }
  |\mathcal{L}\tilde W_{0,1}(d_l,y)-\tilde W_{0,1}(d_l,y)| \le C \frac{\kappa(d_l,y)}{1-y^2},
\]
where $\tilde c_0>0$ is given in \eqref{W}, the remaining terms in
\eqref{projstep2} can be bounded from \eqref{hassine}, \eqref{defG},
\eqref{defql}, \eqref{nl}, \eqref{rl} and \eqref{adam}. Note in
particular that
\begin{align}
 \left\|\tilde
               \pi_0^{\bb_l(s)} \begin{pmatrix}0\\ \tilde
                 V_l(y,s,\qlun)\end{pmatrix}\right\|_{\q H} \le &\;
             C\|q(s)\|_{\q H}^2+C J^{1+\delta},\;\;
&\left|\tilde\pi_\lambda^{\bb_l(s)} \begin{pmatrix}0\\ \tilde
    G_l(\qlun)\end{pmatrix}\right|\le &\; C \|q(s)\|_{\mathcal{ H}}^2,\nonumber\\
\left|\tilde \pi_\lambda^{d_l(s)}\vc{0}{\tilde
     R_l}\right|\le &\;J,
  &\left|\tilde \pi_0^{\bb_l(s)}\vc{\kappa(d_j(s))}{0}\right|
    \le &\; C\bar J, \label{hassine2}\\
  \left|d_j'(s)\tilde \pi_\lambda^{d_l(s)} \vc{\partial_d \kappa(d_j(s))}{0}\right|
    \le &\; C|\zeta_j'(s)|\cdot \bar J(s)\label{hassine3}
\end{align}
for $j\neq l$ and for some $\delta>0$, where $J$ and $\bar J$ are
defined in \eqref{defJ} and \eqref{defjb}.
Compiling all
this information, we see that
\begin{eqnarray}\label{nnn}
 |\theta_l'|\le C|\zeta_l'\||q\|_\mathcal{H}+ C J+C \|q\|^2_{\mathcal{ H}}+ C|\theta_l'\||q\|_{\mathcal{ H}}+ C \bar J \sum_{j\neq l}(|\theta_j'|+|\zeta'_j|). 
\end{eqnarray}
Introducing $A=\displaystyle\sum_{l=1}^k(|\zeta_l'|+|\theta_l'|+
\left|\left(\check\alpha_1^l\right)'-\check\alpha_1^l\right|)$, then using (\ref{n}) and (\ref{nnn}), we see that
\begin{eqnarray*}
 A\le CA(\|q\|_{\q H}+\bar J)+ C J+C \|q\|^2_{\mathcal{ H}}.
\end{eqnarray*}
Since $\|q\|_{\q H}+\bar J\to 0$ as $s\to \infty$ from \eqref{orth}
and \eqref{defjb}, the conclusion of item (i) of Proposition
\ref{propdyn} follows.\\
(ii) We proceed similarly as in page 635 of \cite{MZajm12} where the
real-valued case is treated. The adaptation goes on smoothly, except
for some delicate estimates related to the complex value of the
solution. As in \cite{MZajm12}, we proceed in 2 steps:\\
- We first project equation \eqref{eqq} with the projector
$\pi_-=\pi_-^{d_1,\theta_1,\dots,d_k,\theta_k}$ defined in
\eqref{decompq0} and write a differential inequality satisfied by
$q_-$ defined in \eqref{defq-}.\\
- Then, we use that inequality to write a differential inequality
satisfied by $\varphi(q_-(s), q_-(s))$ where the bilinear form
$\varphi$ was introduced in \eqref{defvarphi}. 

\bigskip

\textbf{Step 1: A differential inequality satisfied by $q_-(y,s)$}

We claim the following:
\begin{cl}[A partial differential inequality for $q_-$] \label{clrachid} For $s$ large enough, we have
\begin{eqnarray*}
  &&\left\|\partial_s q_--L q_-
     - \sum_{l=1}^k \check \alpha_1^l e^{i\theta_l}
     \vc{0}{\check V_l(y,s,\check F_{1,1})}
     -\vc{0}{f(q_1)}-\vc{0}{R}\right\|_{\H}\\
& \le& C J + C\|q\|_{\H}^2
\end{eqnarray*}
where $J(s)$ is defined in \eqref{defJ}.
\end{cl}
\begin{proof}
The proof follows the scheme of the real-valued case given in Claim
C.2 page 635 in \cite{MZajm12} and the following. However, due to the
complex structure, some estimates become tricky. For that reason, we
give details in the following.

\medskip

Applying the operator $\pi_-$ defined in \eqref{decompq0} to equation
\eqref{eqq}, we write
\begin{align}
\pi_-\left(\partial_s q\right)
&
    \;= \pi_-\left(L q\right)
+  \pi_-\begin{pmatrix}0\\ f(q_1)\end{pmatrix}
  +\pi_-\begin{pmatrix}0\\ R(y,s)\end{pmatrix}\label{rachid}\\
 &
   -\sum_{j=1}^k
   \pi_-\left[\begin{pmatrix}
e^{i\theta_j}(i\theta_j'(s)\kappa(d_j(s))+d_j'(s)\partial_d
  \kappa(d_j(s)))\\0 \end{pmatrix}\right],\nonumber
\end{align}
In what follows,
we handle each term appearing in this identity.

\medskip

- \textit{Estimate of $\pi_-\left(\partial_s q\right)$}: 
Differentiating \eqref{estq-} then using the decomposition
\eqref{decompq0} with $\partial_s q$ instead of $q$, we write two
different expressions for $\partial_s q$ in the following:
\begin{align}
  \partial_s q(y,s) =&
 \sum_{l=1}^k e^{i\theta_l} \left\{\left[i\theta_l'\check \alpha_1^l
                       +\left(\check\alpha_1^l\right)'\right] \check F_1(d_l,y)
+\check\alpha_1^ld_l'\partial_d \check F_1(d_l,y)\right\}
  +\partial_s q_-(y,s),\label{ihp1}\\
  \partial_s q(y,s) = &\sum_{l=1}^k e^{i\theta_l}\left\{\check
  \pi_1^{d_l}[\Re(e^{-i\theta_l}\partial_sq)] \check F_1(d_l,y) + \check
  \pi_0^{d_l}[\Re(e^{-i\theta_l}\partial_s q)] \check F_0(d_l,y) + i\tilde
  \pi_0^{d_l}[\Im(e^{-i\theta_l}\partial_s q)] \tilde F_0(d_l,y) \right\}\nonumber\\
&+\pi_-(\partial_s q)(y,s). \label{ihp2}
\end{align}
Let us note by definition \eqref{defql} of $q_l$ that $\partial_s q_l
= e^{-i\theta_l}(-i\theta_l'q+\partial_s q)$, hence, by the
decomposition into real and imaginary parts given in \eqref{riql},
together with \eqref{boris}, we have for $\lambda=0$ or $1$, 
\begin{equation}\label{david1}
  \left|\check \pi_\lambda^{d_l}[\Re(e^{-i\theta_l}\partial_sq)]
    - \check \pi_\lambda^{d_l}[\partial_s \check q_l]\right|
  \le C|\theta_l'|\|q\|_{\q H}.
\end{equation}
Similarly, it holds that
\begin{equation}\label{david2}
  \left|\tilde \pi_0^{d_l}[\Im (e^{-i\theta_l}\partial_sq)]
    - \tilde \pi_0^{d_l}[\partial_s \tilde q_l]\right|
  \le  C|\theta_l'|\|q\|_{\q H}.
\end{equation}
Making the difference between these two expressions in \eqref{ihp1}
and \eqref{ihp2}, then using \eqref{majoration}, \eqref{david1},
\eqref{zero}, \eqref{defzetaj}, \eqref{david2}, \eqref{pascaline}, we see that
\[
  \|\pi_-(\partial_s q) - \partial_s q_-\|_{\q H} \le
 C\sum_{l=1}^k|\check\alpha_1^l|[|\theta_l'|+|\zeta_l'|]
 + C\|q\|_{\q H}\sum_{l=1}^k|\theta_l'|+|\zeta_l'|.
\]
Since $|\check \alpha_1^l|\le C\|q\|_{\q H}$ by definition
\eqref{defa1l} and by the same argument as for \eqref{boris}, using
item (i) of Proposition \ref{propdyn}, we see that
\begin{equation}\label{rachid1}
\|\pi_-(\partial_s q) - \partial_s q_-\|_{\q H} \le C J\|q\|_{\q H}+C\|q\|_{\q H}^3.
\end{equation}

- \textit{Estimate of $\pi_-(Lq)$}: Applying the operator $L$ to
\eqref{estq-} on the one hand, and using \eqref{decompq0} with $L q$
instead of $q$ on the other hand, we write
\begin{align}
  L q=& \;\sum_{l=1}^k \check \alpha_1^l L ( e^{i\theta_l}\check F_1(d_l,y))
         + L q_-\nonumber\\
 L q = &\sum_{l=1}^k e^{i\theta_l}\left\{\check
  \pi_1^{d_l}[\Re(e^{-i\theta_l}L q)] \check F_1(d_l,y) + \check
  \pi_0^{d_l}[\Re(e^{-i\theta_l}L q)] \check F_0(d_l,y) + i\tilde
  \pi_0^{d_l}[\Im(e^{-i\theta_l}L q)] \tilde F_0(d_l,y) \right\}\nonumber\\
&+\pi_-(L q). \nonumber
\end{align}
Therefore, recalling that $\check \alpha_0^l=0$ by \eqref{zero}, we write
\begin{align}
  \pi_-(L q) - L q_- =&\; \sum_{\lambda=0}^1\sum_{l=1}^k 
  [\check \alpha_\lambda^l L(e^{i\theta_l} \check F_\lambda(d_l,y))
  - e^{i\theta_l}\check
  \pi_\lambda^{d_l}[\Re(e^{-i\theta_l}L q)] \check F_\lambda(d_l,y) ]\nonumber\\
 &\; -\sum_{l=1}^k ie^{i\theta_l} \tilde
  \pi_0^{d_l}[\Im(e^{-i\theta_l}L q)] \tilde F_0(d_l,y).\label{mahmoud}
\end{align}
Note first by definition of $L$, $L_{d_l}$ and $V_l$ given in \eqref{defL} and
\eqref{defG} that
\begin{equation}\label{Lq}
Lq=L(e^{i\theta_l} q^l)=e^{i\theta_l}L_{d_l}(q^l)+e^{i\theta_l}\vc{0}{V_l(y,s,q^l_1)},
\end{equation}
where $q^l=e^{-i\theta_l}q$ from \eqref{defql}. Therefore, by
definitions \eqref{barL_bb } and \eqref{tildeL_bb } of $\check
L_{d_l}$ and $\tilde L_{d_l}$, we write
\[
\Re(e^{-i\theta_l}Lq)= \check L_{d_l}(\check q_l)+\vc{0}{\check
  V_l(y,s,q^l_1)}
\mbox{ and }
\Im(e^{-i\theta_l}Lq)= \tilde L_{d_l}(\tilde q_l)+\vc{0}{\tilde
  V_l(y,s,q^l_1)}.
\]
In particular, for $\lambda=0$ or $1$, we derive from \eqref{pld1},
\eqref{pld2} and \eqref{defa1l} that
\begin{equation}\label{hela}
  \check \pi_\lambda^{d_l}[\Re(e^{-i\theta_l}Lq)]=
 \check \pi_\lambda^{d_l}[\check L_{d_l}(\check q_l)]+
 \check \pi_\lambda^{d_l}\vc{0}{\check V_l(y,s,q^l_1)}
= \lambda \check \alpha_\lambda^l+
 \check \pi_\lambda^{d_l}\vc{0}{\check V_l(y,s,q^l_1)}
\end{equation}
and
\begin{equation}\label{hela2}
  \tilde \pi_0^{d_l}[\Im(e^{-i\theta_l}Lq)]= 
  \tilde \pi_0^{d_l}[\tilde L_{d_l}(\tilde q_l)]+
  \tilde \pi_0^{d_l}\vc{0}{\tilde V_l(y,s,q^l_1)}
  =\tilde \pi_0^{d_l}\vc{0}{\tilde V_l(y,s,q^l_1)}.
\end{equation}
Using \eqref{Lq} with $q^l$ replaced by $\check F_\lambda$, which is
an eigenfunction for $\check L_{d_l}$ as mentioned right before \eqref{110}, we write
\begin{align}
L (e^{i\theta_l}\check F_\lambda(d_l,y)) 
  =& \; e^{i\theta_l}\check L_{d_l}\check F_\lambda(d_l,y)
     + e^{i\theta_l} \vc{0}{V_l(y,s,\check F_{\lambda,1}(d_l,y))}\nonumber\\
  =& \;\lambda e^{i\theta_l} \check F_\lambda(d_l,y)
 + e^{i\theta_l} \vc{0}{V_l(y,s,\check F_{\lambda,1}(d_l,y))}.\label{bouzid}
\end{align}
Using \eqref{mahmoud}, \eqref{bouzid}, \eqref{zero}, \eqref{hela},
\eqref{hassine}, \eqref{majoration}, \eqref{hela2}, \eqref{hassine2}
and \eqref{majoration1},  we get
\begin{equation}\label{rachid2}
 \left \|\pi_-(Lq) - Lq_-
  -\sum_{l=0}^k \check \alpha_1^l e^{i\theta_l}\vc{0}{V_l(y,s,\check
    F_{1,1})}\right\|_{\q H}  \le C \|q\|_{\H}^2 + CJ^{1+\delta}
\end{equation}
for some small $\delta>0$, where $J$ is defined in \eqref{defJ}.

 \medskip

- \textit{Estimate of $\pi_-\vc{0}{f(q_1)}$}: By definition
\eqref{decompq0} of the operator $\pi_-$, together with \eqref{defG}, \eqref{defql},
\eqref{majoration} and \eqref{majoration1}, \eqref{nl1} and
\eqref{hassine2}, we have
\begin{equation}\label{rachid3}
  \left\|\pi_-\vc{0}{f(q_1)}- \vc{0}{f(q_1)}\right\|_{\H} \le
C\sum_{\lambda=0}^1\sum_{l=1}^ k
 \left|\check\pi^{d_l(s)}_\lambda\vc{0}{\check G_l(q^l_1)}\right|
   + C\sum_{l=1}^ k \left|\tilde\pi^{d_l(s)}_\lambda\vc{0}{\tilde G_l(q^l_1)}\right|                                     \le C\|q(s)\|_{\H}^2.
\end{equation}

- \textit{Estimate of $\pi_-\vc{0}{R}$}: Similarly, using \eqref{rl}
and \eqref{hassine2}, we write
\begin{equation}\label{rachid4}
  \left\|\pi_-\vc{0}{R}- \vc{0}{R}\right\|_{\H}
  \le C\sum_{\lambda=0}^1\sum_{l=1}^ k
  \left|\check\pi^{d_l(s)}_\lambda\vc{0}{\check R_l}\right|
+C\sum_{l=1}^k\left|\tilde \pi_\lambda^{d_l(s)}\vc{0}{\tilde R_l}\right|
  \le CJ(s).
\end{equation}

- \textit{Estimate of $\pi_-\vc{ie^{i\theta_j}\kappa(d_j,y)}{0}$ where
  $j=1,\dots,k$}: By definition \eqref{decompq0} of $\pi_-$, we write
\begin{align*}
\pi_-\vc{ie^{i\theta_j}\kappa(d_j,y)}{0} =&
\vc{ie^{i\theta_j}\kappa(d_j,y)}{0}
- \sum_{l=1}^k\sum_{\lambda=0}^1 e^{i\theta_l}\check
  \pi_\lambda^{d_l}\left[\Re \vc{ie^{i(\theta_j-\theta_l)}\kappa(d_j,y)}{0}\right]
  \check F_\lambda(d_l,y)\\
&  - \sum_{l=1}^k e^{i\theta_l} i\tilde
  \pi_0^{d_l}\left[\Im \vc{ie^{i(\theta_j-\theta_l)}\kappa(d_j,y)}{0}\right] \tilde F_0(d_l,y).
\end{align*}
The only relevant term in the sums in the right-hand side is from the
last sum when $l=j$, in the sense that we write the following from
\eqref{F} and \eqref{pascaline}
\[
e^{i\theta_j} i\tilde
  \pi_0^{d_j}\left[\Im \vc{i\kappa(d_j,y)}{0}\right] \tilde
  F_0(d_j,y)=
  i e^{i\theta_j} \tilde \pi_0^{d_j}\vc{\kappa(d_j,y)}{0}
  =\vc{ie^{i\theta_j}\kappa(d_j,y)}{0}.
\]
Note that this term cancels with the first on the right-hand side.
All the other projections in those sums are small, thanks to
\eqref{giada} and \eqref{hassine2}. Using the bounds
\eqref{majoration} and \eqref{majoration1} on the norms of the
eigenfunctions, we see that
\begin{equation}\label{rachid5}
\left\|\pi_-\vc{ie^{i\theta_j}\kappa(d_j,y)}{0} \right\|_{\q H} \le
C\bar J
\end{equation}
defined in \eqref{defjb}.

\medskip

- \textit{Estimate of $\pi_-\vc{e^{i\theta_j}\partial_d \kappa(d_j,y)}{0}$ where
  $j=1,\dots,k$}: By definition \eqref{decompq0} of $\pi_-$, we write
\begin{align*}
\pi_-\vc{e^{i\theta_j}\partial_d \kappa(d_j,y)}{0} =&
\vc{e^{i\theta_j}\partial_d \kappa(d_j,y)}{0}
- \sum_{l=1}^k\sum_{\lambda=0}^1 e^{i\theta_l}\check
  \pi_\lambda^{d_l}\left[\Re \vc{e^{i(\theta_j-\theta_l)}\partial_d \kappa(d_j,y)}{0}\right]
  \check F_\lambda(d_l,y)\\
 &  - \sum_{l=1}^k e^{i\theta_l} i\tilde
  \pi_0^{d_l}\left[\Im \vc{e^{i(\theta_j-\theta_l)}\partial_d\kappa(d_j,y)}{0}\right] \tilde F_0(d_l,y).
\end{align*}
Proceeding as with the previous terms and using in particular
\eqref{adel}, \eqref{adam} and \eqref{hassine3}, we write
\begin{equation}\label{rachid6}
  \left\| \pi_-\vc{e^{i\theta_j}\partial_d \kappa(d_j,y)}{0} \right\|_{\q H} \le
C\bar J.
\end{equation}
Starting from \eqref{rachid} then using the estimates in
\eqref{rachid1}, \eqref{rachid2} and \eqref{rachid3} and
\eqref{rachid4}, \eqref{rachid5}, together with item (i) of
Proposition \ref{propdyn}, we conclude the proof of Claim
\ref{clrachid} (remember that $\|q(s)\|_{\q H}$, $J(s)$ and $\bar J(s)$ are bounded
by \eqref{orth}, \eqref{defJ} and \eqref{defjb}).

\end{proof}

\textbf{Step 2: A differential inequality satisfied by
  $A_-(s)$}

By definitions \eqref{defa-} and \eqref{defvarphi} of $A_-$ and $\varphi$, it holds that 
\begin{equation}\label{jazz}
A_-'(s) =2 \varphi(\ps q_-, q_-) -(p-1)[I_1+(p-3)I_2+2I_3]
\end{equation}
where
\begin{align*}
  I_1 =&\;\iint \Re(\bar K \ps K)|K|^{p-3}|q_{-,1}|^2\rho dy,\;\;
  I_2 = \iint \Re(\bar K \ps K)|K|^{p-5}(\Re(\bar Kq_{-,1}))^2\rho dy,\\
  I_3 = &\;\iint |K|^{p-3}\Re(\bar K q_{-,1}) \Re (\partial_s \bar K
  q_{-,1})\rho dy, 
\end{align*}
and 
$K(y,s)$ is defined in \eqref{defK}. Since we have by definition
\eqref{defk},
$|\partial_d \kappa(d,y)|\le C\frac{\kappa(d,y)}{1-d^2}$,
using \eqref{defK} together with item
(i) of Proposition \ref{propdyn} and \eqref{defzetaj} that
\begin{align}
  |\ps K(y,s)|\le&\; C\sum_{l=1}^k |\theta_l'(s)|\kappa(d_l(s),y)
  + C\sum_{l=1}^k |d_l'(s)|\cdot |\partial_d\kappa(d_l(s),y)|\nonumber\\
  \le&\; C\sum_{l=1}^k (|\theta_l'(s)|+|\zeta_l'(s)|)\kappa(d_l(s),y)
  \le C(\|q(s)\|_{\q H}^2+J(s)) \sum_{l=1}^k \kappa(d_l(s),y),\label{familial}
\end{align}
where $J(s)$ is introduced in \eqref{defJ}. Therefore, for $i=1$, $2$
or $3$, we write
\begin{equation}\label{rougui}
  |I_i| \le C (\|q(s)\|_{\q H}^2+J(s)) \|q_{-,1}(1-y^2)^{\frac
    1{p-1}}\|_{L^\infty}^2\sum_{l=1}^k J_l(s)
\end{equation}
where $J_l(s)$ is defined in \eqref{defjis}. Recalling the following
embedding from Lemma 2.2 page 51 of \cite{MZjfa07},
\begin{equation}\label{embed}
  \forall v\in \q H_0,\;\; \|v\|_{L^2_{\frac \rho{1-y^2}}} +
  \|v(1-y^2)^{\frac 1{p-1}}\|_{L^\infty}\le C\|v\|_{\q H_0}
\end{equation}
where the space $\q H_0$ is defined in \eqref{10}, we derive the
following bound
\begin{equation}\label{cq-}
  \|q_-\|_{\q H}
  \le C\|q\|_{\q H} + C \sum_{l=1}^k |\check \alpha_1^l|
  \le C\|q\|_{\q H},
\end{equation}
from \eqref{estq-}, \eqref{majoration}, \eqref{defa1l}, \eqref{barpi},
\eqref{majoration'} and the Cauchy-Schwarz inequality applied to the
inner product \eqref{defphi}. Using \eqref{jazz}, \eqref{rougui},
Claim \ref{lemjis}, \eqref{embed}, the definitions of the spaces $\q
H$ \eqref{defnh} and $\q H_0$ \eqref{10} and \eqref{cq-}, we see that 
\begin{equation}\label{zz1}
\left|A_-'(s)- 2\varphi(\ps q_-, q_-)\right|
\le C\|q\|_{\H}^2\left(\|q\|_{\H}^2 + J\right).
\end{equation}
This way, the question reduces to the evaluation of $\varphi(\ps q_-,
q_-)$. Using Claim \ref{clrachid} together with
Lemma \ref{lemequiv} and \eqref{cq-},
we write the following:
\begin{align}
&\left|\varphi(\ps q_-, q_-) - \varphi(L q_-, q_-) -
\Re \iint \bar q_{-,2} f(q_1) \rho dy
-\Re \iint \bar q_{-,2} G \rho dy\right|\label{zz2}\\
\le&\; C\|q_-\|_{\H}\left(J + \|q\|_{\H}^2\right)
      \le CJ\sqrt{|A_-|} +C J\bar J\|q\|_{\H} + C\|q\|_{\H}^3. \nonumber
\end{align}
where
\begin{equation*}
G(y,s) = \sum_{l=1}^k \check\alpha_1^l(s)e^{i\theta_l(s)}\check V_l(y,s, \check F_{1,1}(d_l(s),y)) +R(y,s).
\end{equation*}

In the following part of the proof, we estimate every term of \eqref{zz2} in order to
finish the proof of item (ii) of Proposition \ref{propdyn}.
\medskip

- \textit{Estimate of $\varphi(L q_-, q_-)$}: By definition
\eqref{defvarphi} and \eqref{defL} of the bilinear form $\varphi$ and
the operator $L$, we write
\begin{align*}
  \varphi(L q_-, q_-) = &\;\Re \iint \left(\mathcal{L}q_{-,1}
+\psi(q_{-,1})-\frac{p+3}{p-1} q_{-,2}- 2 y \partial_y q_{-,2} \right) \bar
  q_{-,2} \rho dy\\
 &+ \Re \iint \left(\partial_y q_{-,2} \partial_y\bar
  q_{-,1}(1-y^2)
   +\left(\frac{2(p+1)}{(p-1)^2}-|K|^{p-1}\right)q_{-,2}
   \bar q_{-,1}  \right)\rho dy\nonumber\\
  &-(p-1)\iint |K|^{p-3}\Re(\bar K q_{-,2})\Re(K \bar q_{-,1}) \rho dy 
\end{align*}
where $\psi(q_{-,1})$ and $K(y,s)$ are given in \eqref{defL} and
\eqref{defK}. Using integration by parts as in page 107 of
\cite{MZjfa07}, based in particular on the operator $\q L$ defined in
\eqref{8}, we see that 
\begin{align*}
  \varphi(L q_-, q_-) =&
 +(p-1)\Re \iint|K|^{p-3}K(K\cdot q_{-,1})  \bar q_{-,2} \rho dy
 -(p-1)\iint |K|^{p-3}\Re(\bar K q_{-,2})\Re(K \bar q_{-,1}) \rho dy\\
& -\frac 4{p-1} \iint |q_{-,2}|^2 \frac \rho{1-y^2} dy
\end{align*}
where $\cdot$ refers to the inner product in $\m R^2$. Since $K\cdot
q_{-,1}= \Re(K \bar q_{-,1})$, it follows that 
 \begin{equation}\label{zz3}
\varphi(L q_-, q_-) = -\frac 4{p-1} \iint |q_{-,2}|^2 \frac \rho{1-y^2} dy.
\end{equation}

- \textit{Estimate of  $\Re \iint \bar q_{-,2} f(q_1) \rho dy$}: Since
we have from the definitions \eqref{110} and \eqref{defk} of $\check
F_1(d,y)$ and $\kappa(d,y)$, 
\begin{equation*}
|\check F_{1,1}(d,y)|\le C \kappa(d,y),
\end{equation*}
using \eqref{estq-}, \eqref{cq-} and \eqref{nl1}, we write
\begin{equation}\label{moncef}
\left|\Re\iint \bar q_{-,2} f(q_1) \rho dy- \Re\iint \bar q_2 f(q_1) \rho dy\right|
\le C\sum_{l=1}^k |\check\alpha_1^l| \iint \kappa(d_l) |f(q_1)| \rho dy
\le C\|q\|_{\H}^3.
\end{equation}
Since $q_2 = \partial_s q_1 + \displaystyle\sum_{l=1}^k
e^{i\theta_l}(i\theta_l'\kappa(d_l)+d_l'\partial_d\kappa(d_l))$ by the
equation \eqref{defq} satisfied by $q(y,s)$, using \eqref{familial}
and \eqref{nl1}, we see that
\begin{align}
 &\;   \left|\Re\iint \bar q_2 f(q_1) \rho dy
    - \Re \iint \partial_s \bar q_1 f(q_1) \rho dy\right|
\le \sum_{l=1}^k
  \iint \left(|\theta_l'|\kappa(d_l)+|d_l'|\partial_d \kappa(d_l)|\right)
 |f(q_1)| \rho dy\nonumber\\
  &\;  \le C(\|q\|_{\q H}^2 +J) \sum_{l=1}^k\iint \kappa(d_l) |f(q_1)| \rho dy
    \le  C \|q\|_{\q H}^2  (\|q\|_{\q H}^2 +J). \label{moncef2}
\end{align}
In order to estimate $\Re \iint \partial_s \bar q_1 f(q_1) \rho dy$,
as we have already announced in the remark following Proposition 
\ref{propdyn}, here comes a key issue in this work: the nonlinear term $f(\xi)$
\eqref{defL} is a gradient, when $\m C$ is identified with $\m
R^2$. More precisely, introducing
\[
  \F(K,\xi) = \frac{|K+\xi|^{p+1}}{p+1} - \frac{|K|^{p+1}}{p+1}
  -|K|^{p-1}(K\cdot\xi)
  - \frac{|K|^{p-1}}2|\xi|^2 -\frac{(p-1)}2|K|^{p-3}(K\cdot \xi)^2
\]
we see after differentiation (in the $\m R^2$ sense) that
\[
\nabla_\xi\F(K,\xi) = f(\xi).
\]
Introducing
$R_- = \iint \F(q_1) \rho dy$,
  then assuming formally
that $K(y,s)$ \eqref{defK} does not depend on $s$, we see by a direct
differentiation that ``$R_-'=\Re \iint \partial_s \bar q_1 f(q_1)
\rho dy$''. Since $K(y,s)$ does depend on $s$, we have some additional
terms in that derivative, as we will shortly see. Since
\[
  R_-'= \Re \iint \partial_s \bar q_1 \nabla_\xi\F(K,q_1) \rho dy
+\Re \iint \partial_s \bar K \nabla_K\F(K,q_1) \rho dy\]
and 
\[
\left|\nabla_K\F(K,\xi) -f(\xi)\right|\le C|K|^{p-2}|\xi|^2
\]
from a direct computation, we use
\eqref{familial} to write
\[
  \left| R_-'- \Re \iint \partial_s \bar q_1 f(q_1) \rho dy\right|
  \le C(J+\|q\|_{\q H}^2)\sum_{l=1}^k\iint
  \kappa(d_l(s))[|f(\xi)|+|K|^{p-2}|q_1|^2]\rho dy.
\]
Using \eqref{nl1}, the embedding \eqref{embed} and Lemma \ref{lemjis},
we see that
\[
  \left| R_-'- \Re \iint \partial_s \bar q_1 f(q_1) \rho dy\right|
  \le C\|q\|_{\q H}^2 (J+\|q\|_{\q H}^2).
\]
Combining this with \eqref{moncef} and \eqref{moncef2}, we see that
\begin{equation}\label{zz40}
  \left|\Re \iint \bar q_2 f(q_1) \rho dy- R_-'\right|
  \le C\|q\|_{\q H}^2 (J+\|q\|_{\q H}^2).
\end{equation}
Now we claim that
\begin{equation}\label{boundF}
|\F (K,\xi)|\le C |\xi|^{p+1} + C {\m 1}_{\{p\ge 2\}} |K|^{p-2}|\xi|^3.
\end{equation}
Indeed, if $K=0$, then this is clear. If not, we may introduce $\nu$
and $\zeta$ in $\m C$, so that $K=|K|\nu$, $\xi=|K|\zeta$. Note that 
$|\nu|=1$. If $|\zeta|\le \frac 12$, making a Taylor expansion, we bound
$\F$ by $|K|^{p+1}|\zeta|^3$.
If $|\zeta|\ge \frac 12$, then we bound it by $|K|^{p+1}|\zeta|^{p+1}$.
Putting together both estimates yields \eqref{boundF}. 
Using the same argument as for estimate (C.35) page 639 in
\cite{MZajm12} where the real case is handled, we see that estimate
\eqref{boundr-} holds.
\medskip

- \textit{Estimate of $\Re \iint \bar q_{-,2} G \rho dy$}: Since
$\check V_l(y,s,q_1^l)$ satisfies the bound \eqref{marzouk}, it is
easy to follow the real-valued case given in pages 640 and 641 of
\cite{MZajm12} to derive the following bound:
\begin{equation}\label{zz50}
  \left|\Re \iint \bar q_{-,2} G \rho dy\right| \le
  \frac 1{p-1}  \iint |q_{-,2}|^2 \frac \rho{1-y^2} dy
  + C\check J(s)^2+o(\|q(s)\|_{\q H}^2)
\end{equation}
where $\check J(s)$ is introduced in \eqref{defJ}.

\medskip

Recalling that $\|q(s)\|_{\q H} \to 0$ by \eqref{orth} and $J\bar J
=o(\check J)$ by \eqref{defjb} and \eqref{defJ}, 
we may combine
all the estimates in \eqref{zz1}, \eqref{zz2}, \eqref{zz3},
\eqref{zz40} and \eqref{zz50}
to conclude the proof of item (ii) in Proposition \ref{propdyn} (note
that we have just proved estimate \eqref{boundr-} above).\\
(iii) With some little care to the complex structure, the proof is
quite similar to the real-valued case, in the one-soliton case (see
pages 110-112 in \cite{MZjfa07}) and the multi-soliton case (see page
641 in \cite{MZajm12}). For that reason, the proof is omitted.\\
This concludes the proof of Proposition \ref{propdyn}.
\end{proof}

\subsection{Refined equations on the modulation parameters}
We prove Proposition \ref{propref} here.

\begin{proof}[Proof of Proposition \ref{propref}]
Take $l=1,\dots,k$.
  This statement is in fact a refinement of the proof of item (i) of
  Proposition \ref{propdyn}. More precisely, all we need to do is to
  refine the projections of the term $R_l(y,s)$ shown in
  \eqref{linearise-complexe}.
  Indeed:\\
  - from equations \eqref{projpdl} and \eqref{projstep2}, together
  with all the subsequent estimates in the proof;\\
- from item (i) of Proposition \ref{propdyn} and Proposition
\ref{propsize}, together with the definitions \eqref{defJ} and
\eqref{defjb} of $J$, $\check J$ and $\bar J$, we see that
\begin{equation}\label{elsa}
  \left|\frac{2\kappa_0}{p-1}\zeta_l'(s)
    -\check\pi^{d_l}_0\vc{0}{\check R_l(y,s)}\right| +
 \left|\theta_l'(s)
    -\tilde\pi^{d_l}_0\vc{0}{\tilde R_l(y,s)}\right|
 \le CJ^{1+\delta}  
\end{equation}
  for some small enough $\delta>0$. It remains then to refine the
  projections shown in the previous estimate. The key idea is to
  obtain the following Taylor expansion for $R(y,s)$ defined in
  \eqref{defL}: 
  \begin{lemma}[A Taylor expansion of $R(y,s)$]\label{lemtayR}
    For $s$ large enough, it holds that
\begin{align*}
\Big|R-&\sum_{l=1}^k e^{i\theta_l(s)}\kappa(d_l)^{p-1}
         {\m 1}_{\{y_{l-1}<y<y_l\}}\sum_{j\neq l} \kappa(d_j)
         \left(p\cos(\theta_j-\theta_l)+i
         \sin(\theta_j-\theta_l)\right)\Big|\notag\\&\le C\sum_{l=1}^k
  \kappa(d_l)^{p-\bar p} {\m 1}_{\{y_{l-1}<y<y_l\}}\sum_{j\neq l} \kappa(d_j)^{\bar p},
\end{align*}
where $\bar p=\min(p,2)$, and $y_j(s)$ are defined right after \eqref{defvb}.  
\end{lemma}
\begin{proof}
  This is a simple Taylor expansion.
  The proof is omitted, since it is the same as in the real-valued
  case stated in (3.99) page 610 in \cite{MZajm12}. 
\end{proof}
With this statement and the integral table given in Lemma \ref{E}, one
can proceed as in pages 610 and 611 in \cite{MZajm12} to find an
equivalent of the projections shown in \eqref{elsa} and finish the proof.
Note that for the exact values of $A$
and $B$ follows from \eqref{elsa} together with the definitions
\eqref{barpi} and \eqref{pi} of the projections and Lemma \ref{E}:
$A= \frac{p(p-1)\check c_0 \check c_2}{2\kappa_0^2}$ and $B=\tilde c_0
\tilde c_2$, where $\kappa_0$, $\check c_0$, $\tilde c_0$, $\check
c_2$ and $\tilde c_2$ are given in \eqref{defk}, \eqref{defcl},
\eqref{W} and Lemma \ref{E}. Since $\partial_y \rho = - \frac 4{p-1}y
\frac \rho{1-y^2}$ by definition \eqref{defro} of $\rho$, using an
integration by parts, we write
\[
  \frac 1{\check c_0} = \frac 4{p-1}\iint \frac{y^2}{1-y^2}\rho dy
  = -\iint y \partial_y \rho dy = \iint \rho dy, 
\]
hence,
\begin{align*}
  \frac 1{\tilde c_0}& = \frac{4\kappa_0^2}{p-1} \iint \frac \rho{1-y^2} dy
  = \frac{4\kappa_0^2}{p-1}\iint \frac{1-y^2+y^2}{1-y^2}\rho dy
  =  \frac{4\kappa_0^2}{p-1} \left(1+\frac{p-1}4\right)\iint \rho dy\\
&  = \frac{p+3}{p-1}\kappa_0^2\iint \rho dy.
\end{align*}
This concludes the proof of Proposition \ref{propref}.

\end{proof}

\section{Asymptotically real-valued behavior in the characteristic case}\label{SecToda}
This section is devoted to the analysis of the  complex-valued first
order Toda system obtained in Proposition \ref{propref}, which we
recall below in \eqref{toda1}-\eqref{toda2}. Accordingly, we assume
throughout the section that
\[
  k\ge 2.
\]
This part is the main novelty of this paper, since it can by no
means be reduced to the real-valued case. As a matter of fact, our aim
is to show that the complex-valued case behaves like the real-valued
case (see Proposition \ref{prop:edo-main} below). However, such a result needs
some hard ODE analysis, relying on the introduction of some Lyapunov
functional below in \eqref{defcE}, the use of the discrete Dirichlet
Laplacian together with the Perron-Frobenius theorem, as the reader
may see below.
Some elements of our analysis are inspired by \cite[Section 6]{JL9}.
However, for the reader's convenience, we provide detailed proofs.

\medskip

We proceed in 3 steps, each making a separate subsection:\\
- In Subsection \ref{sub1}, we recall the first order Toda system from
Proposition \ref{propref}, and state the main result of the section,
where we find the asymptotic behavior of that system.\\
- In Subsection \ref{sub2}, we exhibit some energy structure of that
system, and show that it can be expressed thanks to the discrete
Dirichlet Laplacian (see equations \eqref{eq:a'} and \eqref{eq:b'}
below), where the Perron-Frobenius theorem applies.\\
- In Subsection \ref{sub3}, we study the evolution of the energy
functional, which turns to be a Lyapunov functional, and conclude the
derivation of the asymptotic behavior of our system.
\subsection{Asymptotic behavior of the modulation parameters} \label{sub1}
Let $\bs \xxi(s) = (\xxi_1(s), \ldots, \xxi_k(s)) \in \bR^k$ and $\bs \theta(s) = (\theta_1(s), \ldots, \theta_k(s)) \in \bR^k$
be the modulation parameters introduced in Section \ref{sec2} right
before \eqref{orth}. 
Recall that we defined in \eqref{defJ}
\begin{equation}
\label{eq:J-def}
J(s) = \sum_{j=1}^{k-1}\eee^{-\frac{2}{p-1}(\xxi_{j+1}(s) - \xxi_j(s))}.
\end{equation}
From \eqref{orth}, we know that $\lim_{s \to \infty}(\xxi_{j+1}(s) -
\xxi_j(s)) = \infty$, for all $j \in \{1, \ldots,
k-1\}$. Equivalently, it holds that
\begin{equation}
\label{eq:J-to-0}
\lim_{s\to\infty} J(s) = 0.
\end{equation}
In Proposition \ref{propref},
we proved that there exists $\delta > 0$ such that the following
relations hold for $s$ large enough:
\begin{equation}
\label{toda1}
\begin{aligned}
\frac{1}{A}\xxi_j'(s) = &-\cos(\theta_j(s) - \theta_{j-1}(s))\eee^{-\frac 2{p-1}(\xxi_j(s) - \xxi_{j-1}(s))} \\
&+ \cos(\theta_{j+1}(s) - \theta_{j}(s))\eee^{-\frac 2{p-1} (\xxi_{j+1}(s) - \xxi_{j}(s))} + O(J(s)^{1 + \delta}), \\
\end{aligned}
\end{equation}
\begin{equation}
\label{toda2}
\begin{aligned}
\frac 1B\theta_j'(s) = &-\sin(\theta_j(s) - \theta_{j-1}(s))\eee^{-\frac 2{p-1} (\xxi_j(s) - \xxi_{j-1}(s))} \\
&+ \sin(\theta_{j+1}(s) - \theta_{j}(s))\eee^{-\frac 2{p-1} (\xxi_{j+1}(s) - \xxi_{j}(s))} + O(J(s)^{1 + \delta}),
\end{aligned}
\end{equation}
where
\begin{align*}
  A&=\frac{p(p-1)}{2\iint \rho(y)dy}\kappa_0^{p-1}2^{\frac 2{p-1}}\int_{-\infty}^\infty \cosh^{-\frac{2p}{p-1}}(z)e^{\frac {2z}{p-1}}\tanh z dz >0,\\
  B&=\frac{p-1}{(p+3)\iint \rho(y)dy}\kappa_0^{p-1}2^{\frac 2{p-1}}\int_{-\infty}^\infty \cosh^{-\frac{2p}{p-1}}(z)e^{\frac {2z}{p-1}} dz.
\end{align*}
By a simple integration by parts, we see that
\[
  \int_{-\infty}^\infty \cosh^{-\frac{2p}{p-1}}(z)e^{\frac {2z}{p-1}} dz
  = p \int_{-\infty}^\infty \cosh^{-\frac{2p}{p-1}}(z)e^{\frac {2z}{p-1}}\tanh z dz,
\]
therefore,
\[
\frac AB = \frac{(p+3)}2.  
  \]
The aim of this section is to prove the following result.
\begin{proposition}
  \label{prop:edo-main}
  Take $k\ge 2$. 
If $\bs \xxi$ and $\bs \theta$ are $C^1$ functions satisfying
\eqref{eq:J-to-0}, \eqref{toda1} and \eqref{toda2} for $s$ large
enough, then, there exist $\delta > 0$, $\xxi_{00} \in \bR$ and $\theta_{00}
\in \bR$ such that for all $j=1,\dots,k$,
\begin{align}
\xxi_j(s) &= \xxi_{00} + \kappa_j \log s +\alpha_{k,j}+ O(s^{-\delta}), \label{expxxi}\\
\theta_j(s) &\equiv \theta_{00} + j \pi + O(s^{-\delta})\quad (\tx{mod}\ 2\pi),\nonumber
\end{align}
where $\kappa_j = \frac{p-1}2 \left(j-\frac{k+1}2\right)$ and
$\alpha_{k,j}$ were defined right before \eqref{toda}.
\end{proposition}
\begin{rk}
Another way to express \eqref{expxxi} would be to write $\xxi_j(s) =
\xxi_{00} + \bar \zeta_j(s) + O(s^{-\delta})$, where $(\bar
\zeta_1(s),\dots, \bar \zeta_k(s))$ is the solution of system \eqref{toda}
with zero center of mass defined in the introduction right before \eqref{toda}. Note that 
the number $\delta$ in Proposition~\ref{prop:edo-main} may differ from the number $\delta$ used in \eqref{toda1} and \eqref{toda2}.
\end{rk}

\bigskip

\subsection{Preliminaries} \label{sub2}
 For any $(\bs \xxi, \bs \theta) \in \bR^k\times \bR^k$, we define the
 following functional
\begin{equation}\label{defcE}
\cE(\bs \xxi, \bs \theta) := -\sum_{j=1}^{k-1}\cos(\theta_{j+1} - \theta_j)\eee^{-\frac{2}{p-1}(\xxi_{j+1} - \xxi_j)}.
\end{equation}
In \eqref{eq:lyap}, we will see that this is a Lyapunov functional, if
one neglects the error terms of order $J^{1+\delta}$ in system
\eqref{toda1}-\eqref{toda2}, before showing it is indeed a Lyapunov
functional, without neglecting those terms (see \eqref{Elyap}
below). Let us compute
\begin{equation} \label{dzje}
\begin{aligned}
\partial_{\xxi_j}\cE(\bs \xxi, \bs \theta) &= -\partial_{\xxi_j}\big(\cos(\theta_{j+1} - \theta_j)\eee^{-\frac{2}{p-1}(\xxi_{j+1} - \xxi_j)}
+\cos(\theta_{j} - \theta_{j-1})\eee^{-\frac{2}{p-1}(\xxi_{j} - \xxi_{j-1})}\big)\\
&= -\frac{2}{p-1}\cos(\theta_{j+1} - \theta_j)\eee^{-\frac{2}{p-1}(\xxi_{j+1} - \xxi_j)}
+ \frac{2}{p-1}\cos(\theta_{j} -
\theta_{j-1})\eee^{-\frac{2}{p-1}(\xxi_{j} - \xxi_{j-1})} 
\end{aligned}
\end{equation}
and
\begin{equation}\label{dtje}
\begin{aligned}
\partial_{\theta_j}\cE(\bs \xxi, \bs \theta) &= -\partial_{\theta_j}\big(\cos(\theta_{j+1} - \theta_j)\eee^{-\frac{2}{p-1}(\xxi_{j+1} - \xxi_j)}
+\cos(\theta_{j} - \theta_{j-1})\eee^{-\frac{2}{p-1}(\xxi_{j} - \xxi_{j-1})}\big) \\
&=-\sin(\theta_{j+1} - \theta_j)\eee^{-\frac{2}{p-1}(\xxi_{j+1} - \xxi_j)}
+ \sin(\theta_{j} - \theta_{j-1})\eee^{-\frac{2}{p-1}(\xxi_{j} - \xxi_{j-1})}.
\end{aligned}
\end{equation}
These formulas are valid also for $j = 1$ and $j = k$, if we apply the
usual conventions recalled in the remark right after the statement of
Proposition \ref{propref}.
Our differential system can thus be rewritten as follows:
\begin{equation}
\label{eq:gradient}
\begin{aligned}
\bs \xxi'(s) &= -\frac{(p-1)A}{2}\,\partial_{\bs \xxi}\cE(\bs \xxi(s), \bs \theta(s)) + O(J(s)^{1 +\delta}), \\
\bs \theta'(s) &= -B\,\partial_{\bs \theta}\cE(\bs \xxi(s), \bs \theta(s)) + O(J(s)^{1 +\delta}).
\end{aligned}
\end{equation}
As we have announced earlier, if we neglect the error term of order $J^{1 +\delta}$, these relations immediately imply
that $\cE$ is a Lyapunov functional of the system, which is the crucial observation for our analysis.
To be more precise, noting that $|\partial_{\bs \xxi}\cE| +
|\partial_{\bs \theta}\cE| \lesssim J$ from \eqref{eq:J-def}, we have
\begin{equation}
\label{eq:lyap}
\begin{aligned}
\dd s\cE(\bs \xxi(s), \bs \theta(s)) &= \la \partial_{\bs \xxi}\cE(\bs \xxi(s), \bs \theta(s)), \bs\xxi'(s) \ra
+ \la \partial_{\bs \theta}\cE(\bs \xxi(s), \bs \theta(s)), \bs\theta'(s)\ra \\
&= -\frac{(p-1)A}{2}\,|\partial_{\bs \xxi}\cE(\bs \xxi, \bs \theta)|^2 - B\,|\partial_{\bs \theta}\cE(\bs \xxi, \bs \theta)|^2
+ O(J^{2+\delta}) \\
&= -\frac{2}{(p-1)A}|\bs \xxi'|^2 - \frac{1}{B}|\bs \theta'|^2 + O(J^{2 + \delta}).
\end{aligned}
\end{equation}
Here, and everywhere in this section, the brackets $\la \cdot, \cdot
\ra$ denote the standard Euclidean inner product.
We also denote $\cos \bs b := (\cos b_1, \ldots, \cos b_{k-1})$, $\sin \bs b := (\sin b_1, \ldots, \sin b_{k-1})$,
$\eee^{-\bs a} := (\eee^{-a_1}, \ldots, \eee^{-a_{k-1}})$ etc.
Moreover, $\cos \bs b\,\eee^{-\bs a}$ denotes the \emph{vector} with components $\cos b_j\, \eee^{-a_j}$,
similarly for $\sin \bs b\,\eee^{-\bs a}$ etc.

\medskip

Set
\begin{equation}\label{defajbj}
a_j(s) := \frac{2}{p-1}(\xxi_{j+1}(s) - \xxi_j(s))\mbox{ and }b_j(s) :=
\theta_{j+1}(s) - \theta_j(s).
\end{equation}
With this notation, we have
\begin{equation*}
\begin{aligned}
\partial_{\xxi_j} \cE &= -\frac{2}{p-1}\cos b_j \eee^{-a_j} + \frac{2}{p-1}\cos b_{j-1}\eee^{-a_{j-1}}, \\
\partial_{\theta_j} \cE &= -\sin b_j\eee^{-a_j} + \sin b_{j-1}\eee^{-a_{j-1}}.
\end{aligned}
\end{equation*}
Observe from \eqref{toda1} that
\begin{equation}\label{eqaj}
a_j'(s) = \frac{2A}{p-1}\Big(\cos (b_{j+1}(s)) \eee^{-a_{j+1}(s)} -
2\cos (b_j(s)) \eee^{-a_j(s)} + \cos
(b_{j-1}(s))\eee^{-a_{j-1}(s)}\Big) +O(J(s)^{1+\delta}).
\end{equation}
In shorter form,
\begin{equation}
\label{eq:a'}
\bs a'(s) = -\frac{2A}{p-1}\Delta^D \big(\cos\bs b(s) \eee^{-\bs
  a(s)}\big) +O(J(s)^{1+\delta}),
\end{equation}
where $\Delta^D$ is defined as follows.
\begin{definition}
The (positive) \emph{discrete Dirichlet Laplacian} in dimension 1 is the matrix
$\Delta^D = (\Delta_{jl}) \in \bR^{(k-1)\times (k-1)}$
defined by the conditions
$\Delta_{jj} = 2$ for $j = 1, \ldots, k-1$, $\Delta_{j, j+1} = \Delta_{j+1, j} = -1$ for $j = 1, \ldots, k-2$
and $\Delta_{jl} = 0$ if $|j - l| \geq 2$.
\end{definition}
Similarly as \eqref{eq:a'}, we have
\begin{equation}
\label{eq:b'}
\bs b'(s) = -B\Delta^D \big(\sin\bs b(s) \eee^{-\bs a(s)}\big) +O(J(s)^{1+\delta}).
\end{equation}

For later use, we study in more detail the matrix $\Delta^D$. First, note that the associated quadratic form is
\begin{equation}\label{vdv}
\la \bs v, \Delta^D \bs v\ra = \sum_{j=1}^k (v_j - v_{j-1})^2,\quad\text{for }\bs v \in \bR^{k-1}\text{ and }v_0 = v_k = 0.
\end{equation}
Denote $\bs 1 := (1, \ldots, 1) \in \bR^{k-1}$ and $\bs \sigma := (\sigma_j)_{j=1}^{k-1}$ with $\sigma_j := \frac{j(k-j)}{2}$.
By an elementary computation, one checks that
\begin{equation}\label{lapsig}
\Delta^D\bs\sigma = \bs 1.
\end{equation}
We also denote $\Pi := \{\bs v \in \bR^{k-1}: \la \bs \sigma, \bs v\ra
= 0\}$,
\begin{equation}\label{defc0}
c_0 := \frac{1}{\la \bs \sigma, \bs 1\ra}
= \frac{12}{(k-1)k(k+1)},
\end{equation}
\begin{equation}\label{psv}
P_\sigma \bs v := \bs v - \frac{\la \bs \sigma, \bs v\ra}{\|\bs
  \sigma\|^2}\bs \sigma
\end{equation}
the orthogonal projection of $\bs v$ on $\Pi$, and $P_1 \bs v := \bs v - c_0\la \bs \sigma, \bs v\ra\bs 1$
the projection of $\bs v$ on $\Pi$ along the direction $\bs 1$.
Observe that $P_1 \Delta^D\bs v = \Delta^D \bs v - c_0\la \bs \sigma, \Delta^D\bs v\ra\bs 1 = \Delta^D\bs v - c_0 \la \bs 1, \bs v\ra \bs 1$.
\begin{lemma}
\label{lem:matrix-D}
The matrix $\Delta^D$ is positive definite. There exists a constant $c_1 > 0$ such that for all $\bs v \in \bR^{k-1}$
\begin{equation*}
\la \bs v, P_1 \Delta^D\bs v\ra = \la \bs v, \Delta^D\bs v\ra - c_0\la \bs 1, \bs v\ra^2 \geq c_1 | P_\sigma \bs v |^2.
\end{equation*}
\end{lemma}
\begin{proof}
Consider the matrix $\wt \Delta^D = (\wt \Delta_{jl})$ given by $\wt \Delta_{jl} := 2\delta_{jl} + c_0 - \Delta_{jl}$.
The desired inequality is equivalent to
\begin{equation}
\label{eq:perron}
\la \bs v, \wt \Delta^D\bs v\ra \leq 2|\bs v|^2 - c_1 | P_\sigma \bs v |^2.
\end{equation}
Noting that the matrix $\wt \Delta^D$ has positive entries,
the Perron-Frobenius theorem applies, and we have the existence of an
eigenvector $\bs \tau$ with positive entries such that $\wt
\Delta^D\bs \tau = \alpha \bs \tau$ where $\alpha>0$ is
a simple eigenvalue and the largest in modulus. Since $\wt \Delta^D$
is symmetric, hence real diagonalizable in some orthogonal basis, it
has no other eigenvector with positive entries, up to a multiplying factor. Since
\begin{equation}
\wt \Delta^D\bs \sigma = 2\bs\sigma + c_0\la \bs 1, \bs \sigma\ra\bs 1 - \Delta^D\bs \sigma = 2\bs\sigma
\end{equation}
and $\bs \sigma$ has positive entries, it follows by uniqueness that
$\alpha =2$ and we can take $\bs \tau = \bs \sigma$.  
Moreover, all the eigenvalues of the
restriction of  $\wt\Delta^D$ to the hyperplane $\Pi$ are strictly less
than $2$ (in modulus).
Decomposing the inner product in \eqref{eq:perron} into a
contribution along $\sigma$ and another in $\Pi$, \eqref{eq:perron}
follows. This concludes the proof of Lemma \ref{lem:matrix-D}. 
 \end{proof}
Using \eqref{psv}, we get the following consequence of the lemma:
\begin{corollary}\label{corv2}
For some $c_2>0$ and for all $v\in \m R^{k-1}$, it holds that $\la \bs v, \Delta^D\bs v\ra \ge c_2|v|^2$.
\end{corollary}

\subsection{Differential inequalities for Lyapunov functionals}\label{sub3}
\begin{lemma}
\label{lem:a-bound}
For all $j \in \{1, \ldots, k-1\}$ the function $a_j(s) - \log s$ is bounded.
\end{lemma}
\begin{proof}
By a slight abuse of notation, we denote $\cE(s) := \cE(\bs \xxi(s),
\bs \theta(s))$ introduced in \eqref{defcE}. Using the expressions \eqref{dzje} and \eqref{dtje}
of the derivatives of $\cE$, together with \eqref{vdv} and Corollary
\ref{corv2}, we see that
\begin{equation}\label{coer}
|\partial_{\bs \xxi}\cE|^2 + |\partial_{\bs \theta}\cE|^2\gtrsim |\cos\bs b\,\eee^{-\bs a}|^2 + |\sin\bs b\,\eee^{-\bs a}|^2
\gtrsim \la \bs 1, \eee^{-\bs a}\ra^2,
\end{equation}
where the last inequality follows from equivalence of norms.
By \eqref{eq:lyap}, the last bound and the definition
\eqref{eq:J-def} of $J(s)$, for all $s$ large enough we have
\begin{equation}\label{Elyap}
\cE'(s) \leq -C^{-1}J(s)^2.
\end{equation}
We also have $|J'(s)| \leq CJ(s)^2$ from
\eqref{eq:J-def}, \eqref{defajbj} and \eqref{eqaj}, thus $\cE'(s) \leq
-C^{-2}|J'(s)|$. Recalling that both $J(s)$ and $\cE(s)$ go to $0$ as
$s\to \infty$ from \eqref{eq:J-to-0} and \eqref{defcE}, we may
integrate in $s$ to obtain
\begin{equation}
\label{eq:E-and-J}
\cE(s) \geq C^{-2}\int_s^\infty |J'(\tau)|\ud \tau \geq C^{-2}J(s).
\end{equation}
Next, arguing as for \eqref{coer}, then applying Lemma~\ref{lem:matrix-D} with $\bs v = \cos\bs b\,\eee^{-\bs a}$, we get
\begin{equation*}
\begin{aligned}
|\partial_{\bs \xxi}\cE|^2 &= \frac{4}{(p-1)^2}\sum_{j=1}^k(\cos b_j \eee^{-a_j} - \cos b_{j-1}\eee^{-a_{j-1}})^2
= \frac{4}{(p-1)^2}\la \cos\bs b\,\eee^{-\bs a}, \Delta^D (\cos\bs b\,\eee^{-\bs a})\ra \\
&\geq \frac{4c_0}{(p-1)^2}\la \bs 1, \cos\bs b\,\eee^{-\bs a}\ra^2 = \frac{4c_0}{(p-1)^2}\cE^2.
\end{aligned}
\end{equation*}
Hence, \eqref{eq:lyap} and \eqref{eq:E-and-J} yield
\begin{equation*}
\cE'(s) \leq -\frac{2Ac_0}{p-1}\cE(s)^2 + O(J(s)^{2+\delta}) \leq \Big({-}\frac{2Ac_0}{p-1} + O(J(s)^\delta)\Big)\cE(s)^2.
\end{equation*}
Dividing both sides by the positive number $\cE(s)^2$ and using the Chain Rule, we have
\begin{equation}
\label{eq:E-1-deriv}
(\cE(s)^{-1})' \geq \frac{2Ac_0}{p-1} + O(J(s)^\delta),
\end{equation}
and an integration in time yields
\begin{equation*}
\cE(s)^{-1} \geq \Big(\frac{2Ac_0}{p-1} + o(1)\Big)s,
\end{equation*}
implying
\begin{equation}
\label{eq:E-ubound}
\cE(s) \leq \Big(\frac{p-1}{2Ac_0} + o(1)\Big)s^{-1}.
\end{equation}
The bound \eqref{eq:E-and-J} yields
\begin{equation}
\label{eq:J-ubound}
J(s) \leq Cs^{-1}.
\end{equation}
Plugging this into \eqref{eq:E-1-deriv} and integrating in time leads to
\begin{equation*}
\cE(s)^{-1} \geq \Big(\frac{2Ac_0}{p-1}+ O(s^{-\delta})\Big)s,
\end{equation*}
thus we can improve \eqref{eq:E-ubound} and write
\begin{equation}
\label{eq:E-ubound-ref}
\cE(s) \leq \frac{p-1}{2Ac_0}s^{-1} + O(s^{-1-\delta}).
\end{equation}
The bound \eqref{eq:J-ubound} also implies by definitions
\eqref{eq:J-def} and \eqref{defajbj} of $J(s)$ and $a_j(s)$ that
\begin{equation}
\label{eq:a-lbound}
a_j(s) \geq \log s - C, \qquad\text{for all }j \in \{1, \ldots, k-1\}.
\end{equation}

In order to obtain the upper bound, we decompose
\begin{equation}\label{defrz}
\bs a(s) = r(s) \bs 1 + \bs z(s), \qquad r(s) = c_0\la \bs \sigma, \bs a(s)\ra, \qquad \la \bs \sigma, \bs z(s)\ra = 0.
\end{equation}
Using the equation \eqref{eq:a'} satisfied by $\bs a(s)$, we have
\begin{equation*}
r'(s)\bs 1 + \bs z'(s) = \bs a'(s) = {-}\frac{2A}{p-1}\Delta^D(\cos \bs b(s)\,\eee^{-\bs a(s)}) + O(s^{-1-\delta}).
\end{equation*}
Taking the inner product with $\bs \sigma$, then using the fact that
$\Delta^D$ is symmetric and that $\Delta^D \bs \sigma =\bs 1$ (use
\eqref{lapsig}), we see by definitions \eqref{defcE} and \eqref{defc0}
of $\cE$ and $c_0$ that 
\begin{equation}
\label{eq:dsr}
r'(s) = \frac{2Ac_0}{p-1}\cE(s) + O(s^{-1-\delta}).
\end{equation}
Integrating in time and using \eqref{eq:E-ubound-ref}, we get
\begin{equation*}
r(s) \leq \log s + C.
\end{equation*}
But $r$ is a weighted mean of $a_1, \ldots, a_{k-1}$ by definition
\eqref{defc0} of $c_0$, so \eqref{eq:a-lbound} implies that we also have the desired upper bound
\begin{equation}
\label{eq:a-ubound}
a_j(s) \leq \log s + C, \qquad\text{for all }j \in \{1, \ldots, k-1\}.
\end{equation}
This concludes the proof of Lemma \ref{lem:a-bound}.
\end{proof}

We further study the evolution in the variables $(r, \bs z, \bs b)$.
From Lemma~\ref{lem:a-bound} and by definition \eqref{defrz} of $r$
and $\bs z$, we know that  $r(s) - \log s$ and $\bs z(s)$ are bounded functions.
In particular, \eqref{eq:J-ubound} holds and we may 
rewrite \eqref{eq:a'} and \eqref{eq:b'} as
\begin{equation}
\label{eq:a'b'}
\begin{aligned}
\bs a'(s) &= -\frac{2A}{p-1}\eee^{-r(s)}\Delta^D\big(\cos \bs b(s)\,\eee^{-\bs z(s)}\big) + O(s^{-1-\delta}), \\
\bs b'(s) &= -B\eee^{-r(s)}\Delta^D\big(\sin \bs b(s)\,\eee^{-\bs z(s)}\big) + O(s^{-1-\delta}).
\end{aligned}
\end{equation}
Applying $P_1$ to the first equation above, we obtain
\begin{equation}
\label{eq:z'}
\bs z'(s) = -\frac{2A}{p-1}\eee^{-r(s)}P_1\Delta^D\big(\cos \bs b(s)\,\eee^{-\bs z(s)}\big) + O(s^{-1-\delta}).
\end{equation}

Consider the functional $\cF: \Pi \times \bR^{k-1} \to \bR$ defined by
\begin{equation}\label{defcF}
\cF(\bs z, \bs b) := -\sum_{j=1}^{k-1}\cos(b_j)\eee^{-z_j} = -\la \cos \bs b, \eee^{-\bs z}\ra.
\end{equation}
Note that
\begin{equation}
\label{eq:F-er-E}
\cF(\bs z(s), \bs b(s)) = \eee^{r(s)}\cE(\bs\xxi(s), \bs\theta(s)).
\end{equation}
It turns out that, even though $\eee^{r(s)}$ grows linearly, $\cF$ is still a Lyapunov functional, up to error terms.
In order to see this, we write
\begin{equation}
\label{eq:dsF}
\begin{aligned}
&\dd s\cF(\bs z(s), \bs b(s)) = \la \bs z', \cos \bs b\,\eee^{-\bs z}\ra + \la \bs b', \sin\bs b\,\eee^{-\bs z}\ra \\
&\quad= -\frac{2A}{p-1}\eee^{-r}\la P_1 \Delta^D(\cos \bs b\,\eee^{-\bs z}), \cos \bs b\,\eee^{-\bs z}\ra - B\eee^{-r}\la \Delta^D(\sin \bs b\,\eee^{-\bs z}), \sin\bs b\,\eee^{-\bs z}\ra + O(s^{-1-\delta}) \\
&\quad \lesssim -\eee^{-r}\big(|P_\sigma (\cos \bs b\,\eee^{-\bs z})|^2 + |\sin \bs b\,\eee^{-\bs z}|^2\big) + O(s^{-1-\delta}),
\end{aligned}
\end{equation}
where in the last step we use Lemma~\ref{lem:matrix-D} and Corollary
\ref{corv2}. We need the following fact. 
\begin{lemma}
\label{lem:Ps-lbound}
Set $\bs z_{\tx{cr}} := -\log\bs\sigma + c_0 \la \bs\sigma, \log\bs\sigma\ra\bs 1 \in \bR^{k-1}$.
\begin{enumerate}[(i)]
\item
For any $R > 0$ and $\epsilon > 0$ small enough there exists $\eta = \eta(R, \epsilon) > 0$ such that if
\begin{equation}\label{eq:zb-far}
\bs z \in \Pi,\;\;
  |\bs z| \leq R, \qquad |\bs z - \bs z_{\tx{cr}}| + \min(|\cos\bs b-\bs 1|, |\cos\bs b+\bs 1|) \geq \epsilon,
\end{equation}
then
\begin{equation*}
|P_\sigma (\cos\bs b\,\eee^{-\bs z})| + |\sin \bs b\,\eee^{-\bs z}|\geq \eta.
\end{equation*}
\item There exist $\epsilon_0 > 0$ and $C > 0$ such that if 
\begin{equation*}
|\bs z - \bs z_{\tx{cr}}| + |\bs b-\pi\bs 1| \leq \epsilon_0,
\end{equation*}
then
\begin{equation*}
|P_\sigma (\cos\bs b\,\eee^{-\bs z})| + |\sin \bs b\,\eee^{-\bs z}|\geq C^{-1}\big(|\bs z - \bs z_{\tx{cr}}| + |\bs b - \pi\bs 1|\big).
\end{equation*}
\end{enumerate}
\end{lemma}
\begin{proof}[Proof of (i)]
If $|\sin b_j| \gtrsim 1$ for some $j$, then the conclusion is immediate.
Hence, we can assume $|\sin b_j| \ll 1$, which implies $\min(|\cos b_j - 1|, |\cos b_j + 1|) \ll 1$ for all $j$.
We claim that, if the lemma was to fail, this has to be either $1$ for all $j$, or $-1$ for all $j$.
Indeed, if the sign changes, then the angle between the directions of $\cos \bs b\,\eee^{-\bs z}$ and $\bs \sigma$ is bounded from below,
implying $|P_\sigma(\cos\bs b\,\eee^{-\bs z})| \gtrsim |\cos\bs
b\,\eee^{-\bs z}|$, from where the conclusion follows quite
easily. 

Assume $|\cos \bs b + \bs 1| \ll 1$ (the other case follows by changing $\bs b$ to $\pi\bs 1 - \bs b$).
Upon shifting the components of $\bs b$ by multiples of $2\pi$, we can assume $|\bs b-\pi\bs 1| \ll 1$.
In particular, \eqref{eq:zb-far} implies $|\bs z - \bs z_{\tx{cr}}| \geq \frac 12 \epsilon$.
We have $\cos\bs b\,\eee^{-\bs z} = {-}\eee^{-\bs z}+ (\bs 1 + \cos\bs b)\,\eee^{-\bs z}$, thus
\begin{equation}
\label{eq:Ps-triangle}
\begin{aligned}
|P_\sigma (\cos\bs b\,\eee^{-\bs z})| + |\sin \bs b\,\eee^{-\bs z}| &\geq |P_\sigma \eee^{-\bs z}| - |(\bs 1 +\cos\bs b)\,\eee^{-\bs z}| + |\sin \bs b\,\eee^{-\bs z}| \\
&\geq |P_\sigma \eee^{-\bs z}| + \frac 12 |(\bs b-\pi\bs 1)\,\eee^{-\bs z}|,
\end{aligned}
\end{equation}
where the last inequality follows from $|\bs b-\pi\bs 1|$ being small.
We only need to show that $|\bs z - \bs z_{\tx{cr}}| \geq \frac 12 \epsilon$ implies $|P_\sigma \eee^{-\bs z}| \geq \eta$.

We can rewrite the condition $|P_\sigma \eee^{-\bs z}| \leq\eta$ as follows:
\begin{equation*}
\eee^{-\bs z} = \lambda \bs \sigma + \bs u, \qquad \lambda \in \bR,\ \bs u \in \Pi,\ |\bs u| \leq \eta.
\end{equation*}
Note that $\lambda$ is of size $1$. We have
\begin{equation*}
|{-}\bs z - \log(\lambda\bs \sigma)| = |\log(\lambda \bs \sigma + \bs u) - \log(\lambda\bs \sigma)| \lesssim |\bs u| \lesssim \eta.
\end{equation*}
The vector $\bs z_{\tx{cr}}$ is in fact determined by the expression
$\bs z_{\tx{cr}} = -\log(\lambda_{\tx{cr}}\bs \sigma)$, where
$\lambda_{\tx{cr}}$ is the unique value for which $\bs z_{\tx{cr}} \in
\Pi$. We thus get from the above 
\begin{equation*}
|\bs z - \bs z_{\tx{cr}} + (\log\lambda-\log\lambda_{\tx{cr}})\bs 1| \lesssim \eta.
\end{equation*}
Taking the inner product with $\sigma$ we find out that $|\log\lambda-\log\lambda_{\tx{cr}}| \lesssim \eta$,
which yields $|\bs z- \bs z_{\tx{cr}}| \lesssim \eta$ and finishes the proof.\phantom\qedhere
\end{proof}
\begin{proof}[Proof of (ii)]
By \eqref{eq:Ps-triangle}, it suffices to check that $|P_\sigma \eee^{-\bs z}| \gtrsim |\bs z - \bs z_\tx{cr}|$.
But in the proof of (i) above, we obtained that $|P_\sigma \eee^{-\bs z}| \leq \eta$ implies $|\bs z - \bs z_\tx{cr}| \lesssim \eta$, which is the transposition of the desired estimate.
\end{proof}
\begin{remark}
\label{rem:critical}
The function
$\bs z \mapsto \la \bs 1, \eee^{-\bs z}\ra$, defined for $\bs z \in \Pi$,
is strictly convex and tends to infinity when $|\bs z| \to \infty$
(use the fact that whenever $|\bs z_n| \to \infty$ for some sequence
$\bs z_n \in \Pi$, then, $z_{n,j(n)}\to -\infty$ for some choice of $j(n)$). 
Thus it has a unique critical point $\bs z_{\tx{cr}}$, which is its global minimum.
The Euler-Lagrange equation yields $\eee^{-\bs z_{\tx{cr}}} = \lambda_{\tx{cr}}\bs\sigma$,
where $\lambda_{\tx{cr}}\in \bR$ is the unique number for which $\bs z_{\tx{cr}} \in \Pi$.

The function $\cF:(\bs z, \bs b) \mapsto -\la \cos \bs b, \eee^{-\bs z}\ra$,
defined in \eqref{defcF} on the set $\{(\bs z, \bs b) \in \Pi \times \bR^{k-1}\}$,
has a discrete set of critical points $(\bs z_{\tx{cr}}, \bs b_{\tx{cr}})$, where $\bs b_{\tx{cr}}$
is any vector such that $\cos \bs b_{\tx{cr}} = \bs 1$ or $\cos \bs b_{\tx{cr}} = -\bs 1$.
A priori, $\cos \bs b_{\tx{cr}}$ could be a vector whose components include both $1$ and $-1$, but it is not so.
Indeed, from the Euler-Lagrange equation,
if the function $\bs z \mapsto \sum_{j=1}^{k-1}{\iota_j}\eee^{-z_j}$
has a critical point on $\Pi$, then all the signs $\iota_j$ are the
same. 

Note that $(\bs z_{\tx{cr}}, \pi\bs 1)$ is a saddle point of $\cF(\bs z, \bs b)$,
related to $\max_{\bs b}\min_{\bs z}\cF(\bs z, \bs b)$.
\end{remark}
\begin{proof}[Proof of Proposition~\ref{prop:edo-main}]

  $ $
  
\textbf{Step 1.} (Convergence to the saddle point).
Our first goal is to prove that $\lim_{s \to \infty}\bs z(s) = \bs z_\tx{cr}$
and $\lim_{s\to\infty}\bs b(s) \ \tx{mod}\ 2\pi = \pi\bs 1$.

For the sake of simplicity, in this proof we write $\cF(s) := \cF(\bs z(s), \bs b(s))$,
so that $\cF(s) = \eee^{r(s)}\cE(s)$, see \eqref{eq:F-er-E}.
Since $\cE(s) > 0$ for all $s$ large enough by \eqref{eq:E-and-J} and
\eqref{eq:J-def}, we also have
\begin{equation}\label{cFpos}
\cF(\bs z(s), \bs b(s)) > 0
\end{equation}
  for all $s$ large enough.
Observe that, by \eqref{eq:dsF}, $\cF$ is a sum of a decreasing function
and a term of size $O(s^{-\delta})$, in particular $\lim_{s\to\infty}\cF(s)$ exists.
If it is not true that $\bs z(s) \to \bs z_\tx{cr}$, then there exist $\epsilon > 0$ and a sequence
$s_n$ such that $s_n \to \infty$ and 
\begin{equation}\label{contra}
|\bs z(s_n) - \bs z_\tx{cr}| \geq 2\epsilon, \qquad\text{for all }n.
\end{equation}
Since $|\bs z'(s)|\lesssim s^{-1} \leq s_n^{-1}$ for $s \geq s_n$ (see
\eqref{eq:z'}, the definition \eqref{defrz} of $r(s)$, Lemma
\ref{lem:a-bound}, and remember that $\bs b$ is bounded as we wrote 2
lines before \eqref{eq:a'b'}),  there exists $C_1 > 1$ such that for all $n$ we have
\begin{equation*}
|\bs z(s) - \bs z_\tx{cr}| \geq \epsilon,\qquad\text{for all }s \in [s_n, s_n + C_1^{-1}s_n].
\end{equation*}
By Lemma~\ref{lem:Ps-lbound}, \eqref{eq:dsF}, the definition
\eqref{defrz} of $r(s)$ and Lemma \ref{lem:a-bound}, there exist $C_2, C_3 > 0$ such that for all $s \in [s_n, s_n + C_1^{-1}s_n]$ we have
\begin{equation}
\cF'(s) \leq -C_2^{-1}\eee^{-r(s)} \leq -C_3^{-1}(s_n+C_1s_n)^{-1} \leq -(2C_3)^{-1}s_n^{-1},
\end{equation}
which implies $\cF(s_n + C_1^{-1}s_n) - \cF(s_n) \leq -(2C_1C_3)^{-1}$ for all $n$,
contradicting the fact that $\lim_{s\to\infty}\cF(s)$ exists. 
Thus, $\lim_{s\to\infty}\bs z(s) = \bs z_\tx{cr}$. 

Analogously, $\lim_{s\to\infty} \min(|\cos\bs b(s)-\bs 1|, |\cos\bs b(s)+\bs 1|) = 0$,
which, by continuity, implies that $\lim_{s\to \infty} \cos \bs b(s) =
\bs 1$ or $\lim_{s\to \infty} \cos \bs b(s) = -\bs 1$.
Recalling that $\cF(s)>0$ by \eqref{cFpos}, we see that  the former
option is excluded, because it would imply $\cF(s) < 0$ for $s$ large
enough by \eqref{defcF}. 

\medskip

\textbf{Step 2.} (Polynomial convergence rate.)
We now prove that there exists $\delta > 0$ such that
\begin{equation}
\label{eq:poly-rate}
|\bs z(s) - \bs z_\tx{cr}| + |\bs b(s)\ \tx{mod}\ 2\pi - \pi\bs 1| \leq s^{-\delta},\qquad\text{for }s \gg 1.
\end{equation}
By adding suitable multiples of $2\pi$ to the functions $\theta_j(s)$, we can assume without loss of generality that
$\lim_{s\to\infty} \bs b(s) = \pi\bs 1$, so that we need to estimate $|\bs b(s) - \pi\bs 1|$.

Let $\cF_0 := \la \bs 1, \eee^{-\bs z_\tx{cr}}\ra = \cF(\bs z_\tx{cr}, \pi\bs 1) = \lim_{s\to\infty}\cF(s)$.
Let $\wt\cF(s) := \cF(s) - \cF_0$. By \eqref{eq:dsF}, we have
\begin{equation}
\label{eq:Fs-sdel}
\wt \cF(s) \geq -Cs^{-\delta},\qquad\text{for }s\text{ large enough.}
\end{equation}
Also, we have
\begin{equation}
\wt\cF(s) = \cF(s) - \cF_0 \leq \la \bs 1, \eee^{-\bs z(s)}\ra - \cF_0 \leq C|\bs z(s) - \bs z_\tx{cr}|^2,
\end{equation}
where the last inequality is a consequence of $\bs z_\tx{cr}$ being a
critical point of $\Pi \owns \bs z \mapsto \la \bs 1, \eee^{-\bs z}\ra$, see
Remark~\ref{rem:critical}. Thus, noting that
\begin{equation}\label{ser}
\frac 1C \le se^{-r(s)}\le C
\end{equation}
  from what
we wrote 2 lines before \eqref{eq:a'b'}, we may use \eqref{eq:dsF} and
Lemma~\ref{lem:Ps-lbound} (ii) to write 
\begin{equation}
\wt\cF'(s) = \cF'(s) \leq - C^{-1}s^{-1}(\cF(s) - \cF_0) + Cs^{-1-\delta} = -C^{-1}s^{-1}\wt\cF(s) + Cs^{-1-\delta}
\end{equation}
for some $C > 0$. By modifying $\delta$, we can thus ensure that for $s$ large enough
\begin{equation}
\wt\cF'(s) \leq -2\delta s^{-1}\wt\cF(s) + s^{-1-\delta}\quad\Leftrightarrow\quad \big(s^{2\delta} \wt\cF(s)\big)' \leq s^{-1+\delta},
\end{equation}
and an integration in $s$ yields
\begin{equation}
\wt\cF(s) \leq Cs^{-\delta}.
\end{equation}
Using again \eqref{eq:dsF}, \eqref{ser} and Lemma~\ref{lem:Ps-lbound} (ii), and changing again $\delta$, we deduce that
\begin{equation}
\label{eq:poly-rate-int}
\int_s^\infty \tau^{-1}\big(|\bs z(\tau) - \bs z_\tx{cr}|^2 + |\bs b(\tau) - \pi\bs 1|^2\big)\ud\tau \leq C s^{-4\delta}.
\end{equation}
Below, we deduce \eqref{eq:poly-rate} from \eqref{eq:poly-rate-int} by an argument analogous to the one used in Step 1
(which is often referred to as a \emph{Tauberian argument}).

Suppose there exists a sequence $s_n \to \infty$ such that
\begin{equation}
|\bs z(s_n) - \bs z_\tx{cr}| + |\bs b(s_n) - \pi\bs 1| \geq s_n^{-\delta}.
\end{equation}
Since $|\bs z'(s)| + |\bs b'(s)| \leq Cs_n^{-1}$ for all $s \geq s_n$
(use the argument right after \eqref{contra} for $\bs z'(s)$ and do
the same for $\bs b'(s)$ starting from \eqref{eq:a'b'}) , we have
\begin{equation}
|\bs z(s) - \bs z_\tx{cr}| + |\bs b(s) - \pi\bs 1| \geq \frac 12 s_n^{-\delta},\quad\text{for all }s \in [s_n, s_n + (2C)^{-1}s_n^{1-\delta}],
\end{equation}
which implies
\begin{equation}
\int_{s_n}^{s_n + (2C)^{-1}s_n^{1-\delta}}\tau^{-1}\big(|\bs z(\tau) - \bs z_\tx{cr}|^2 + |\bs b(\tau) - \pi\bs 1|^2\big)\ud\tau
\geq (2C)^{-1}s_n^{1-\delta}s_n^{-1}s_n^{-2\delta}/8,
\end{equation}
contradicting \eqref{eq:poly-rate-int} for $n$ large. Therefore, \eqref{eq:poly-rate} is proved.

From the bound $|\wt\cF(s)| \lesssim s^{-\delta} \Leftrightarrow
|\cF(s) - \cF_0| \lesssim s^{-\delta}$, using \eqref{eq:F-er-E} and \eqref{ser},
we also deduce that $\cE(s) = \eee^{-r}\cF_0 + O(s^{-1-\delta})$. We
thus have from \eqref{eq:dsr}
\begin{equation}
r' = A_0 \eee^{-r} + O(s^{-1-\delta}),
\end{equation}
where $A_0 := \frac{2Ac_0\cF_0}{p-1}$ is an explicit positive
constant. From this equation and \eqref{ser}, we finally obtain
\begin{equation}
r(s) = \log(A_0 s) + O(s^{-\delta}),
\end{equation}
thus, by definition \eqref{defrz} of $r(s)$ and $z(s)$, we see that
$\bs a(s) - \log(A_0 s)\bs 1 = (r(s) - \log(A_0 s))\bs 1 + \bs z(s)$
converges to its limit $\bs z_{cr}$ with rate $O(s^{-\delta})$.

\medskip

\textbf{Step 3.} (Back to original variables.)
Since both $\sum_{j=1}^k \xxi_j(s)$ and $\sum_{j=1}^k \theta_j(s)$
converge with rate $O(s^{-\delta})$ (simply use the equations
\eqref{toda1} and \eqref{toda2} together with the bound
\eqref{eq:J-ubound} on $J(s)$),
it is straightforward to deduce estimates on $\bs \xxi(s)$ and $\bs
\theta(s)$ from estimates on $\bs a(s)$ and $\bs b(s)$, for some
constants $\alpha_{k,j}$ as shown in \eqref{expxxi}. Those constants
are in fact the same as in the real-valued case defined right before
\eqref{toda}, simply because the convergence of $\bs b(s)$ to $\pi \bs
1$  with speed $O(s^{-\delta})$ implies that the equations
\eqref{toda1} reduce to the real-valued case refined in Theorem 1.7
page 1546 in C\^ote and Zaag \cite{CZcpam13}, following the first
expansion by Merle and Zaag in Theorem 6 page 585 in
\cite{MZajm12}. This concludes the proof of Proposition \ref{prop:edo-main}.
\end{proof}

\section{Conclusion of the proof from the real-valued case}\label{SecConcl}
This section is devoted to the proof of Theorems \ref{th1} and
\ref{cor1}. 

\medskip

Let us first derive the following direct consequence of Propositions
\ref{propdecomp}, \ref{prop:edo-main} and \ref{propsize}:
\begin{corollary}[Asymptotically real-valued behavior of the solution
  near characteristic points  when $k\ge 2$] \label{cork2} Consider $x_0\in \q S$
  with $k(x_0)\ge 2$, where $k(x_0)$ is the number of solitons shown
  in the decomposition of Proposition \ref{propdecomp}. Then,
  it holds that 
\[
\left\|\begin{pmatrix} w_{x_0}(s)\\\partial_s w_{x_0}(s) \end{pmatrix}
  -e^{i \theta_{00}}\sum_{i=1}^k(-1)^k\begin{pmatrix} \kappa(\bb_i (s),\cdot)\\0\end{pmatrix} \right\|_{\mathcal{ H}}\rightarrow 0\mbox{ as }s\rightarrow \infty,
\]
with other values of parameters $\bb_i(s) = - \tanh (\zeta_{00}+\bar \zeta_i(s))$,
$\theta_{00}\in \m R$, $\zeta_{00}\in \m R$, where
$(\bar\zeta_1(s),\dots, \bar \zeta_k(s))$ is the solution of system
\eqref{toda} with zero center of mass defined in the introduction.  
  \end{corollary}
  \begin{proof}
Let us first recall the following continuity result related to the
solitons $\kappa(d,y)$ \eqref{defk} (see item (ii) of Lemma A.2 page
2878 in Merle and Zaag \cite{MZdmj12}): 
\begin{equation}\label{contkd}
  \mbox{for all }d_1\mbox{ and }d_2\in(-1,1),\;\;
  \|\kappa(d_1,\cdot)- \kappa(d_2,\cdot)\|_{\q H_0}
  \le C|\arg\tanh d_1-\arg\tanh d_2|.
\end{equation}
Take now $x_0\in \q S$. Note that Proposition \ref{propdecomp}
applies, showing that $(w_{x_0}, \partial_s w_{x_0})$ decomposes into
a sum of $k$ decoupled solitons $(\kappa(d_i(s)),0)$ satisfying
\eqref{separation}, where $k\ge 0$. Assuming that $k\ge 2$, we are in
the framework of Section \ref{secstep2}, where we define new
parameters $d_i(s)$ and $\theta_i(s)$, together with a new function
$q(y,s)$ in \eqref{defq}, satisfying \eqref{orth}. Applying then
Propositions \ref{prop:edo-main} and \ref{propsize}, together with
\eqref{contkd}, we get the result. This concludes the proof of
Corollary \ref{cork2}.
\end{proof}
  From this corollary, we see that whenever $k\ge 2$, the solution is
  asymptotically real-valued, up to a rotation in the complex
  plane. This crucial information enables us to reduce the proof to
  the real-valued case already given in Merle and Zaag \cite{MZajm12}
  and \cite{MZdmj12} together with C\^ote and Zaag in
  \cite{CZcpam13}. This fact should be sufficient to convince the
  expert reader. To be nice to all readers, we sketch in the following
  the main steps of the proof of Theorems \ref{th1} and \ref{cor1}. We
  proceed in 3 subsections, each corresponding to the adaptation of
  one of the three mentioned papers.

\subsection{Decomposition into a decoupled sum of solitons} 

In this section, we adapt the strategy of \cite{MZajm12} in order to
prove Theorem \ref{th1}. Several steps are needed in fact. They are
given numbers in the following. Take $x_0\in \q S$.

(i) As we have just written above in the proof of Corollary
\ref{cork2}, Proposition \ref{propdecomp} holds,  which means that
$(w_{x_0}, \partial_s w_{x_0})$ decomposes into a sum of $k$ decoupled
solitons $(\kappa(d_i(s)),0)$ satisfying \eqref{separation}, where
$k\ge 0$.

(ii) Assuming that $k\ge 2$, we know that Corollary \ref{cork2}
holds. Since equation \eqref{equ} is invariant under phase change, we
may assume
\[
\theta_{00}=0
\]
in the following.

(iii) Furthermore, with the same proof, Proposition 3.13 page 611 in
\cite{MZajm12} holds, with $u$ replaced by its real part $\Re u$,
giving the existence of signed lines for $\Re u$, together with the
``corner property'' for $x\mapsto T(x)$ near $x_0$, and yielding the
following statement:
\begin{proposition}[Existence of signed lines and the corner property near $x_0\in \SS$]
  \label{phat}
   If $x_0\in \SS$ with $k(x_0)\ge 2$, then:\\
(i) For all $j=1,..,k$, 
\[
\Re u(z_j(t),t)\sim e_j^* \kappa_0 \cosh^{\frac 2{p-1}}\zeta_j(s)(T(x_0)-t)^{-\frac 2{p-1}}\mbox{ as }t\to T(x_0),
\]
 where $t\mapsto z_j(t)$ is continuous and defined by
\begin{equation*}
z_j(t)= x_0+(T(x_0)-t)\tanh \zeta_j(s)\mbox{ with }s=-\log(T(x_0)-t).
\end{equation*}
(ii) We have for some $\delta_0>0$ and $C_0>0$, 
\begin{equation*}
\mbox{if }|x-x_0|\le \delta_0,\mbox{ then }T(x_0)-|x-x_0| \le T(x)\le T(x_0)-|x-x_0| + \frac{C_0|x-x_0|}{|\log(x-x_0)|^{\frac{(k-1)(p-1)}2}}.
\end{equation*}
\end{proposition}

(iv) The analogous of Proposition 4.1 page 614 in \cite{MZajm12} also holds:
\begin{proposition}[Properties of $\q S$]\label{propS}
  $ $\\
  (i) The interior of $\q S$ is empty.\\
  (ii) For any $x_0\in \q S$, $k(x_0)\ge 2$, where $k(x_0)$ is the
  number of solitons appearing in the decomposition of $w_{x_0}(y,s)$
  given in Proposition \ref{propdecomp}.
  \end{proposition}
  \begin{proof}
The proof is identical to the real-valued case treated in \cite{MZajm12}. More precisely, it
relies on the corner property stated in Proposition \ref{phat} above
in the case $k(x_0)\ge 2$, together with 4 ingredients to rule out the
case $k(x_0)\le 1$:\\ 
1- Some simple geometrical manipulations on the blow-up curve, related
to its $1$-Lipschitz character and the definition of
non-characteristic points. This argument holds with no modification in
the complex-valued case.\\
2- The fine knowledge of the blow-up behavior in the
non-characteristic case. In the complex-valued case, this was
completely done by Azaiez in \cite{Atams15}.\\
3- A trapping result for equation \eqref{eqw} near the set of
stationary solutions $\kappa(d,y)$ \eqref{defk} in the
non-characteristic case. Again, such a result was proved in
\cite{Atams15}.\\
4- Some continuity result for equation \ref{eqw} near $\kappa(d,y)$,
which remains valid in the complex-valued case.\\
For details, we refer the reader to the real-valued case treated in
\cite{MZajm12}, precisely to Proposition 4.1 page 614 in that paper. 
  \end{proof}
  
(v) \textit{Proof of Theorem \ref{th1}}: Take $x_0\in \q S$. From
Proposition \ref{propS}, we see that $k(x_0)\ge 2$, where $k(x_0)$ is the
  number of solitons appearing in the decomposition of $w_{x_0}(y,s)$
  given in Proposition \ref{propdecomp}. Therefore, Corollary
  \ref{cork2} applies and yields the conclusion of Theorem \ref{th1},
  except for the convergence of the energy, which follows from the
  continuity of the energy \eqref{defE}, itself following from the
  Hardy-Sobolev inequality proved by Merle and Zaag in the appendix of
  \cite{MZajm03}. 

\subsection{Every characteristic point is isolated}

In this section, we adapt the strategy of Merle and Zaag in
\cite{MZdmj12} in order to show the following result:
  \begin{proposition}\label{Sisol}
    The set $\q S$ is made of isolated points.
    \end{proposition}
    \begin{proof}
Starting for Theorem \ref{th1}, one has no difficulty in adapting the
real-valued strategy to the complex-valued one. Indeed, taking $x_0\in
\q S$, and noting the sharp decomposition of $(w_{x_0}, \partial_s
w_{x_0})$ into a sum of $\bar k(x_0)$ decoupled solitons in Theorem
\ref{th1} with $\bar k(x_0)\ge 2$, we consider some close enough
$x_1\neq x_0$, and aim at showing that 
$(w_{x_1}(s), \partial_s w_{x_1}(s))$ will be near one single soliton
at some large time $s=s^*(x_1)$, which by the trapping result and the
study of the non-characteristic case performed by Azaiez in
\cite{Atams15} implies that $x_1$ is a non-characteristic case.

\medskip

The main ingredient of the proof is the \textit{stability of the
  decomposition into a
  decoupled finite sum of generalized solitons} for equation
\eqref{eqw}. By a \textit{generalized} soliton, we mean
$e^{i\theta}\kappa^*(d,\nu,y)$, where  
\begin{equation*}
\kappa_1^*(d,\nu, y) =
\d\kappa_0\frac{(1-d^2)^{\frac 1{p-1}}}{(1+dy+\nu)^{\frac 2{p-1}}},\;\;
\kappa_2^*(d,\nu, y) = \nu \partial_\nu \kappa_1^*(d,\nu, y) =
\d-\frac{2\kappa_0\nu}{p-1}\frac{(1-d^2)^{\frac 1{p-1}}}{(1+dy+\nu)^{\frac {p+1}{p-1}}}.
\end{equation*}
Note that for any $\mu\in\m R$, $e^{i\theta}\kappa^*_1(d,\mu e^s,y)$
is a solution to equation \eqref{eqw}. More precisely:\\
- when $\mu=0$, we recover the standard soliton
$e^{i\theta}\kappa(d,y)$ \eqref{defk}.\\ 
- when $\mu>0$, we have a heteroclinic orbit connecting
$e^{i\theta}\kappa(d,y)$ at $-\infty$ to $0$ at $+\infty$.\\
- when $\mu<0$, we still have the convergence to
$e^{i\theta}\kappa(d,y)$ at $-\infty$, however, the generalized
soliton goes to infinity at time
$s=-\log\left(\frac{1-|d|}{|\mu|}\right)$.

\medskip

The stability of such a decomposition comes from a straightforward
modification of our argument in Section \ref{appdyn}, where we handled
the case with $\nu=0$ for each soliton.  

\medskip

This depicted scenario for $(w_{x_1}(s), \partial_s w_{x_1}(s))$ will be shown
to hold through a cascade of events, starting from some initial time
independent from $x_1$, say $s_0$ large enough (from symmetry, we only
consider the case when $x_1<x_0$). Note first that whether $x_1$ is
characteristic or not, it holds that
\begin{equation}\label{lbe}
\forall s\ge - \log T(x_1),\;\; E(w_{x_1}(s), \partial_s w_{x_1}(s)) \ge E(\kappa_0,0).
\end{equation}
Indeed, just note that $E$ is a Lyapunov functional by the argument of Antonini and
Merle in \cite{AMimrn01}, and use its convergence result given in
Theorem \ref{th1} in the characteristic case, and 4 lines above in the
non-characteristic case.

\medskip

These are the events:\\
1- When $s=s_0$, $(w_{x_1}(s), \partial_s w_{x_1}(s))$ will be close
to $(w_{x_0}, \partial_s w_{x_0})$ for $|x_1-x_0|$ small enough,
hence, close to a sum of $\bar k(x_0)$ decoupled solitons, with $\bar
k(x_0)\ge 2$.\\ 
2- When $s>s_0$, from the stability result for such a decomposition,
the solution will remain close to a sum of
$k$ decoupled \textit{generalized} solitons, until the rightmost soliton
becomes small, which is the case if  $\frac \nu{1-|d|}$ is
large enough. 
This means that the solution becomes close to the sum of just $k-1$
solitons.\\ 
3- Iterating this process, we lose all the solitons, except for the leftmost one.\\
4- Using \eqref{lbe}, we see that the trapping result of Azaiez
\cite{Atams15} applies, and
$(w_{x_1}(s), \partial_s w_{x_1}(s))$ will remain close to only one
soliton. From Theorem \ref{th1}, $x_1$ cannot be a characteristic
point.\\
This concludes the proof of Proposition \ref{Sisol}.
      \end{proof}
An important fact about the above-mentioned strategy lays in the
following slope estimate near a characteristic point, 
which is an analogous statement of Proposition 4.10 page 2869 in
 \cite{MZdmj12}:
\begin{proposition}[Slope estimate near $\q S$]\label{propmore}
For all $x_0\in \q S$, there exists $\delta_0>0$ and $C_0>0$ such that
if $0<|x_1-x_0|\le \delta_0$, then $x_1\in \q R$ and
\[
\frac 1{C_0|\log |x_1-x_0||^{\frac{(\bar k(x_0)-1)(p-1)}2}} \le
T'(x_1)+\frac{x_1-x_0}{|x_1-x_0|} \le \frac{C_0}{|\log |x_1-x_0||^{\frac{(\bar k(x_0)-1)(p-1)}2}},
\]
where $\bar k(x_0)\ge 2$ is the number of solitons shown in the
decomposition of $(w_{x_0}, \partial_s w_{x_0})$ given in Theorem
\ref{th1}. In particular, $T$ is right differentiable with $T_r'(x_0)
=-1$ and left differentiable with $T_l'(x_0) =1$.
\end{proposition}
\begin{proof}
Given the strategy we have just given for the proof of Proposition
\ref{Sisol}, one can check with no difficulty that the proof of the vector-valued
case holds here. See the proof of Proposition 4.10 page 2870 in \cite{MZdmj12}.
\end{proof}

\subsection{Sharp slope estimate near characteristic points}
In this subsection, we follow the strategy of C\^ote and Zaag in
\cite{CZcpam13} in order to refine the slope estimate of Proposition
\ref{propmore}, concluding the proof of Theorem \ref{cor1}.

\medskip

\textit{Proof of Theorem \ref{cor1}}: From Proposition \ref{propmore},
we only need to give a sharper expansion of $T'(x_1)$ for $x_1$ close
to $x_0$ in order to conclude. That improvement comes from Section 2.2
page 1559 in \cite{CZcpam13}, which holds verbatim in the
complex-valued case, and takes as input the improvement of the soliton
center's behavior given in Theorem \ref{th1}.

\appendix

\section{Estimates related to the bilinear form $\varphi$}\label{appequiv}
This section is dedicated to the proof of Lemma \ref{lemequiv}. 

\medskip

\begin{proof}[Proof of Lemma \ref{lemequiv}]
  (i) Forgetting the terms involving $K(y,s)$ \eqref{defK}, the result
  follows the Cauchy-Schwarz inequality, by definition of the
  bilinear form $\varphi$ \eqref{defvarphi} and the norm $\q H$
  \eqref{defnh}. Let us focus then on the the 2 other terms, which are
  both bounded by
  \[
C\iint |K(y,s)|^{p-1}|r_1(y)|\cdot |{\bd r}_1(y)|\rho(y) dy.
  \]
  From the embedding in \eqref{embed}, it is enough to justify that
  \begin{equation}\label{lik}
|K(y,s)|^{p-1}\le \frac C{1-y^2}
    \end{equation}
    in order to conclude. To justify this, we first recall from
    \eqref{F}, \eqref{majoration1} and the definitions \eqref{defnh}
    and \eqref{10} of the norms in $\q H$ and $\q H_0$ that 
  \[
\forall d\in(-1,1),\;\;\|\kappa(d)\|_{\q H_0} \le C.
  \]
  Using the embedding in \eqref{embed}, we see that
  \[
\forall (d,y)\in(-1,1)^2,\;\;|\kappa(d,y)| \le C(1-y^2)^{-\frac 1{p-1}}.
\]
By definition \eqref{defK}, \eqref{lik} follows and so does item
(i).\\
(ii) As in the real case treated in Merle and Zaag \cite{MZajm12},
   noting that the sum of solitons is decoupled from \eqref{orth}, one has to
localize the bilinear form near each soliton and reduce the question
to the single-soliton case, which was already treated in Azaiez
\cite{Atams15}. As a matter of fact, one can see that near $i$-th
soliton where $i=1,\dots,k$, the main contribution of $\varphi(q_-,
q_-)$ reduces to a sum of the 1-soliton versions $\check
\varphi_{d_i(s)}$ \eqref{134'} and $\tilde \varphi_{d_i(s)}$
\eqref{134''} evaluated at adequate projections of
$e^{-i\theta_i(s)}q$.
Since the adaptation is straightforward and in order to keep the paper
in a reasonable length, we omit the proof and kindly refer the reader
to Lemma 3.10 page 605 in \cite{MZajm12}.
\end{proof}

\section{Technical results related to solitons}
In this section, we prove the following result involving solitons:
\begin{lemma}\label{E}{\bf (A table of integrals involving the
    solitons)} There exists some $\delta>0$, such that the following
  estimates hold for all $s$ large enough, where $J$ is defined in
  \eqref{defJ}, $\bar p=\min (p,2)$ and the separators $y_j$ are
  introduced right after \eqref{defvb}:\\
(i) Let $\check A_{i,j,l} =
\d\int_{y_{j-1}}^{y_j}\frac{y+d_i}{1+yd_i}\kappa(d_i)\kappa(d_j)^{p-1}\kappa(d_l)\rho
dy$ with $l\neq j$. Then:\\
- if $i=j$ and $l=i\pm 1$, then $|\check A_{i,i,l}- \sgn(l-i)\check
c_2 e^{-\frac 2{p-1}|\zeta_l-\zeta_i|}|\le  CJ^{1+\delta}$\\
where $\check c_2=2^{\frac 2{p-1}}\kappa_0^{p+1}\int_{\m R}
\cosh^{-\frac{2p}{p-1}}(z) \tanh(z)e^{\frac{2z}{p-1}}dz>0$,\\ 
- otherwise, $|\check A_{i,j,l}| \le CJ^{1+\delta}$.\\
 (ii) Let $\tilde A_{i,j,l}=\int_{y_{j-1}}^{y_j}\kappa(d_i)
 \kappa(d_j)^{p-1}\kappa(d_l)\rho dy$ with $l\neq j$. Then:\\
- if $i=j$ and $l=i\pm 1$, then $|\tilde A_{i,i,l}-\tilde c_2
e^{-\frac{2}{p-1}|\zeta_l-\zeta_i|}|\le C J^{1+\delta}$, where\\
 $\tilde c_2=2^{\frac 2{p-1}}\kappa_0^{p+1}\int_{\m R}
\cosh^{-\frac{2p}{p-1}}(z) e^{\frac{2z}{p-1}}dz>0$,\\ 
- otherwise, $\tilde A_{i,j,l}\le C J^{1+\delta}$.\\
(iii) If $l\neq j$, then $B_{i,j,l}=\int_{y_{j-1}}^{y_j}\kappa(d_i)
\kappa(d_j)^{p-\bar p}\kappa(d_l)^{\bar p}\rho dy\le C J^{1+\delta}$.
\end{lemma}
\begin{proof}
$ $\\ 
(i) See item (iii) in Lemma E.1 page 644 in \cite{MZajm12} except for the exact value of
the constant $\check c_2$ which is given in page 645 of that paper.\\
(ii) One can adapt with no difficulty the the proof of the previous
item given in \cite{MZajm12}.\\
(iii) See item (iv) in Lemma E.1 page 644 in \cite{MZajm12}.
  \end{proof}

  \textbf{Acknowledgements}.
  The authors wish to warmly thank the reviewers for their careful
  reading and valuable
  suggestions which undoubtedly improved the paper. 
  Jacek Jendrej was supported by the ERC Starting Grant ”INSOLIT” 101117126.
    Hatem Zaag wishes to thank Pierre Rapha\"el and the ”SWAT” ERC
project for their support.

\begin{thebibliography}{10}

\bibitem{Apndeta95}
S.~Alinhac.
\newblock {\em Blowup for nonlinear hyperbolic equations}, volume~17 of {\em
  Progress in Nonlinear Differential Equations and their Applications}.
\newblock Birkh\"auser Boston Inc., Boston, MA, 1995.

\bibitem{AMimrn01}
C.~Antonini and F.~Merle.
\newblock Optimal bounds on positive blow-up solutions for a semilinear wave
  equation.
\newblock {\em Internat. Math. Res. Notices}, (21):1141--1167, 2001.

\bibitem{Atams15}
A.~Azaiez.
\newblock Blow-up profile for the complex-valued semilinear wave equation.
\newblock {\em Trans. Amer. Math. Soc.}, 367:5891--5933, 2015.

\bibitem{Acpaa19}
A.~Azaiez.
\newblock Refined regularity for the blow-up set at non characteristic points
  for the vector-valued semilinear wave equation.
\newblock {\em Commun. Pure Appl. Anal.}, 18(5):2397--2408, 2019.

\bibitem{AZbsm17}
A.~Azaiez and H.~Zaag.
\newblock A modulation technique for the blow-up profile of the vector-valued
  semilinear wave equation.
\newblock {\em Bull. Sci. Math.}, 141(4):312--352, 2017.

\bibitem{BetOrlSme}
F.~Bethuel, G.~Orlandi, and D.~Smets.
\newblock Dynamics of multiple degree {G}inzburg-{L}andau vortices.
\newblock {\em Comm. Math. Phys.}, 272:229--261, 2007.

\bibitem{BCSjmp11}
P.~Bizo{\'n}, T.~Chmaj, and N.~Szpak.
\newblock Dynamics near the threshold for blowup in the one-dimensional
  focusing nonlinear {K}lein-{G}ordon equation.
\newblock {\em J. Math. Phys.}, 52(10):103703, 11, 2011.

\bibitem{CFarma85}
L.~A. Caffarelli and A.~Friedman.
\newblock Differentiability of the blow-up curve for one-dimensional nonlinear
  wave equations.
\newblock {\em Arch. Rational Mech. Anal.}, 91(1):83--98, 1985.

\bibitem{CFtams86}
L.~A. Caffarelli and A.~Friedman.
\newblock The blow-up boundary for nonlinear wave equations.
\newblock {\em Trans. Amer. Math. Soc.}, 297(1):223--241, 1986.

\bibitem{CZcpam13}
R.~C\^ote and H.~Zaag.
\newblock Construction of a multisoliton blowup solution to the semilinear wave
  equation in one space dimension.
\newblock {\em Comm. Pure Appl. Math.}, 66(10):1541--1581, 2013.

\bibitem{Pino}
M.~Del~Pino, M.~Kowalczyk, and J.~Wei.
\newblock The {T}oda system and clustering interfaces in the {A}llen–{C}ahn
  equation.
\newblock {\em Arch. Rational Mech. Anal.}, 190:141--187, 2008.

\bibitem{DSScmp12}
R.~Donninger, W.~Schlag, and A.~Soffer.
\newblock On pointwise decay of linear waves on a {S}chwarzschild black hole
  background.
\newblock {\em Comm. Math. Phys.}, 309(1):51--86, 2012.

\bibitem{DuMa}
W.~Dunajski and N.~S. Manton.
\newblock Reduced dynamics of {W}ard solitons.
\newblock {\em Nonlinearity}, 18:1677--1689, 2005.

\bibitem{GuSi06}
S.~Gustafson and I.~M. Sigal.
\newblock Effective dynamics of magnetic vortices.
\newblock {\em Adv. Math.}, 199:448--498, 2006.

\bibitem{HZjde19}
M.~A. Hamza and H.~Zaag.
\newblock Prescribing the center of mass of a multi-soliton solution for a
  perturbed semilinear wave equation.
\newblock {\em J. Differential Equations}, 267(6):3524--3560, 2019.

\bibitem{JL9}
J.~Jendrej and A.~Lawrie.
\newblock Dynamics of kink clusters for scalar fields in dimension 1+1.
\newblock {\em Preprint}, arXiv:2303.11297, 2023.

\bibitem{JerSon1}
R.~L. Jerrard and H.~M. Soner.
\newblock Dynamics of {G}inzburg‐{L}andau vortices.
\newblock {\em Arch. Rat. Mech. Anal.}, 142:99--125, 1998.

\bibitem{JerSon2}
R.~L. Jerrard and H.~M. Soner.
\newblock Scaling limits and regularity results for a class of
  {G}inzburg-{L}andau systems.
\newblock {\em Ann. Inst. H. Poincar{\'e} Anal. Non Lin{\'e}aire},
  16(4):423--466, 1999.

\bibitem{JerSpi}
R.~L. Jerrard and D.~Spirn.
\newblock Refined {J}acobian estimates and {G}ross--{P}itaevsky vortex
  dynamics.
\newblock {\em Arch. Rat. Mech. Anal.}, 190:425--475, 2008.

\bibitem{KV11}
R.~Killip and M.~Vi\c{s}an.
\newblock Smooth solutions to the nonlinear wave equation can blow up on
  {C}antor sets.
\newblock 2011.
\newblock arXiv:1103.5257v1.

\bibitem{Ltams74}
H.~A. Levine.
\newblock Instability and nonexistence of global solutions to nonlinear wave
  equations of the form {$Pu_{tt}=-Au+{\mathcal F}(u)$}.
\newblock {\em Trans. Amer. Math. Soc.}, 192:1--21, 1974.

\bibitem{MS}
N.~Manton and P.~Sutcliffe.
\newblock {\em Topological solitons}.
\newblock Cambridge Monographs on Mathematical Physics. Cambridge University
  Press, Cambridge, 2004.

\bibitem{MZajm03}
F.~Merle and H.~Zaag.
\newblock Determination of the blow-up rate for the semilinear wave equation.
\newblock {\em Amer. J. Math.}, 125:1147--1164, 2003.

\bibitem{MZimrn05}
F.~Merle and H.~Zaag.
\newblock Blow-up rate near the blow-up surface for semilinear wave equations.
\newblock {\em Internat. Math. Res. Notices}, (19):1127--1156, 2005.

\bibitem{MZjfa07}
F.~Merle and H.~Zaag.
\newblock Existence and universality of the blow-up profile for the semilinear
  wave equation in one space dimension.
\newblock {\em J. Funct. Anal.}, 253(1):43--121, 2007.

\bibitem{MZcmp08}
F.~Merle and H.~Zaag.
\newblock Openness of the set of non characteristic points and regularity of
  the blow-up curve for the $1$ d semilinear wave equation.
\newblock {\em Comm. Math. Phys.}, 282:55--86, 2008.

\bibitem{MZxedp10}
F.~Merle and H.~Zaag.
\newblock Isolatedness of characteristic points for a semilinear wave equation
  in one space dimension.
\newblock In {\em S{\'e}minaire sur les {\'E}quations aux D{\'e}riv{\'e}es
  Partielles, 2009--2010}, pages Exp.\ No.\ 11, 10p. {\'E}cole Polytech.,
  Palaiseau, 2010.

\bibitem{MZajm12}
F.~Merle and H.~Zaag.
\newblock Existence and classification of characteristic points at blow-up for
  a semilinear wave equation in one space dimension.
\newblock {\em Amer. J. Math.}, 134(3):581--648, 2012.

\bibitem{MZdmj12}
F.~Merle and H.~Zaag.
\newblock Isolatedness of characteristic points for a semilinear wave equation
  in one space dimension.
\newblock {\em Duke Math. J.}, 161(15):2837--2908, 2012.

\bibitem{MZcmp15}
F.~Merle and H.~Zaag.
\newblock On the stability of the notion of non-characteristic point and
  blow-up profile for semilinear wave equations.
\newblock {\em Comm. Math. Phys.}, pages 1--34, 2015.

\bibitem{MZtams16}
F.~Merle and H.~Zaag.
\newblock Dynamics near explicit stationary solutions in similarity variables
  for solutions of a semilinear wave equation in higher dimensions.
\newblock {\em Trans. Amer. Math. Soc.}, 368(1):27--87, 2016.

\bibitem{MZsls17}
F.~Merle and H.~Zaag.
\newblock Solution to the semilinear wave equation with a pyramid-shaped
  blow-up surface.
\newblock In {\em S\'{e}minaire {L}aurent {S}chwartz---\'{E}quations aux
  d\'{e}riv\'{e}es partielles et applications. {A}nn\'{e}e 2016--2017}, pages
  Exp. No. VI, 13. Ed. \'{E}c. Polytech., Palaiseau, 2017.

\bibitem{MZcpam18}
F.~Merle and H.~Zaag.
\newblock Blowup solutions to the semilinear wave equation with a stylized
  pyramid as a blowup surface.
\newblock {\em Comm. Pure Appl. Math.}, 71(9):1850--1937, 2018.

\bibitem{OvSi}
Yu.~N. Ovchinnikov and I.~M. Sigal.
\newblock The {G}inzburg--{L}andau equation {III}. {V}ortex dynamics.
\newblock {\em Nonlinearity}, 11:1277--1294, 1998.

\bibitem{Stuart}
D.~M.~A. Stuart.
\newblock The geodesic approximation for the {Y}ang-{M}ills-{H}iggs equations.
\newblock {\em Comm. Math. Phys.}, 166:149--190, 1994.

\end{thebibliography}

\def\cprime{$'$} \def\cprime{$'$}

\end{document}